\documentclass[10pt,reqno]{amsart}

\usepackage{srcltx}

\usepackage{amsmath,amssymb,cases,color}
\usepackage{hyperref}
\usepackage[capitalise]{cleveref}
\usepackage{vmargin}
\setmarginsrb{1.5cm}{1cm}{1.5cm}{3cm}{1cm}{1cm}{2cm}{2cm}
\usepackage{dsfont}

\usepackage{amsmath, amssymb,amscd, }
\usepackage{amsfonts}
\usepackage{mathrsfs}
\usepackage{graphicx}

 \usepackage{upref}
\hypersetup{linkcolor=blue, colorlinks=true,citecolor = red}
\newtheorem{theorem}{Theorem}[section]
\newtheorem{proposition}[theorem]{Proposition}

\newtheorem{remark}[theorem]{Remark}
\newtheorem{defin}[theorem]{Definition}

\theoremstyle{definition}

\theoremstyle{remark}
\numberwithin{equation}{section}

\newcommand{\mb}{\mathbb}
\newcommand{\mc}{\mathcal}

\newcommand{\black}{\color{black}}

\renewcommand{\leq}{\leqslant}
\renewcommand{\geq}{\geqslant}

\begin{document}

\title[Motion of a Rigid body in a Compressible Fluid with Navier-slip boundary condition]{Motion of a Rigid body in a Compressible Fluid with Navier-slip boundary condition}

 \date{\today}
 \author{\v S. Ne\u{c}asov\'a} \address{Institute of Mathematics, Czech Academy of Sciences, \v Zitn\' a 25, 11567 Praha 1, Czech Republic} \email{matus@math.cas.cz}

\author{M. Ramaswamy}  \address{NASI Senior Scientist, ICTS-TIFR, Survey No. 151, Sivakote, Bangalore, 560089, India}  \email{mythily.r@icts.res.in}
 
\author{A. Roy}
\address{Institute of Mathematics, Czech Academy of Sciences, \v Zitn\' a 25, 11567 Praha 1, Czech Republic}
\email{royarnab244@gmail.com}

\author{A. Schl\"{o}merkemper} \address{Institute of Mathematics, University of W\"urzburg, Emil-Fischer-Str.~40, 97074 W\"urzburg, Germany}
\email {anja.schloemerkemper@mathematik.uni-wuerzburg.de}

\begin{abstract}
In this work, we study the motion of a rigid body in a bounded domain which is filled with a compressible isentropic fluid. We consider the Navier-slip boundary condition at the interface as well as at the boundary of the domain. This is the first mathematical analysis of a compressible fluid-rigid body
system where  Navier-slip boundary conditions  are considered. We prove existence of a weak solution of the fluid-structure system up to collision.
\end{abstract}

\maketitle


\tableofcontents


\section{Introduction}\label{sec_intro}
Let $\Omega \subset \mathbb{R}^3$ be a bounded smooth domain occupied by a fluid and a  rigid body. Let the rigid body $\mathcal{S}(t)$ be a regular, bounded domain and moving inside $\Omega$. The motion of the rigid body is governed by the balance equations for linear and angular momentum. We assume that the fluid domain $\mathcal{F}(t)=\Omega \setminus \overline{\mathcal{S}(t)}$ is filled with a viscous isentropic compressible fluid. We also assume the slip boundary conditions at the interface of the interaction of the fluid and the rigid body as well as at $\partial \Omega$. 
The evolution of this fluid-structure system can be described by the following equations 
\begin{align}
\frac{\partial {\rho}_{\mc{F}}}{\partial t} + \operatorname{div}({\rho_{\mc{F}}} u_{\mc{F}}) =0, \quad &\forall \ (t,x)\in (0,T)\times\mc{F}(t),\label{mass:comfluid}\\
\frac{\partial ({\rho_{\mc{F}}} u_{\mc{F}})}{\partial t}+ \operatorname{div}({\rho_{\mc{F}}} u_{\mc{F}}\otimes u_{\mc{F}})- \operatorname{div} \mathbb{T}(u_{\mc{F}})+\nabla p_{\mc{F}} =\rho_{\mc{F}}g_{\mc{F}},\quad &\forall \ (t,x)\in (0,T)\times \mc{F}(t),\label{momentum:comfluid}\\
mh''(t)= -\int\limits_{\partial \mc{S}(t)} \left(\mathbb{T}(u_{\mc{F}}\right) - p_{\mc{F}}\mathbb{I}) \nu\, d\Gamma + \int\limits_{\mc{S}(t)} \rho_{\mc{S}}g_{\mc{S}}\, dx,\quad &\forall \ t\in (0,T), \label{linear momentumcomp:body}\\
(J\omega)'(t) = -\int\limits_{\partial \mc{S}(t)} (x-h(t)) \times (\mathbb{T}(u_{\mc{F}}) - p_{\mc{F}}\mathbb{I}) \nu\, d\Gamma + \int\limits_{\mc{S}(t)} (x-h(t)) \times \rho_{\mc{S}}g_{\mc{S}}\, dx,\quad &\forall \ t\in (0,T), \label{angular momentumcomp:body}
\end{align}
the boundary conditions
\begin{align}
u_{\mc{F}}\cdot \nu = u_{\mc{S}} \cdot \nu, \quad &\forall \ (t,x) \in (0,T)\times \partial \mc{S}(t), \label{boundarycomp-1}\\ (\mathbb{T}(u_{\mc{F}}) \nu)\times \nu = -\alpha (u_{\mc{F}}-u_{\mc{S}})\times \nu, \quad &\forall \ (t,x) \in (0,T)\times\partial \mc{S}(t), \label{boundarycomp-2}\\
u_{\mc{F}}\cdot \nu = 0, \quad &\forall \ (t,x) \in (0,T)\times \partial \Omega, \label{boundarycomp-3}\\ (\mathbb{T}(u_{\mc{F}}) \nu)\times \nu = -\alpha (u_{\mc{F}}\times \nu), \quad &\forall \ (t,x) \in (0,T)\times \partial \Omega, \label{boundarycomp-4}
\end{align}
and the initial conditions
\begin{align}
{\rho_{\mc{F}}}(0,x)=\rho_{\mc{F}_{0}}(x),\quad  ({\rho_{\mc{F}}}u_{\mc{F}})(0,x)=q_{\mc{F}_0}(x), & \quad \forall \ x\in \mc{F}_0,\label{initial cond}\\
 h(0)=0,\quad h'(0)=\ell_{0},\quad \omega(0)=\omega_{0}.\label{initial cond:comp}
\end{align}
 The fluid occupies, at $t=0$, the domain $\mc{F}_0=\Omega \setminus \mc{S}_0$, where the initial position of the rigid body is $\mc{S}_0$. In equations \eqref{linear momentumcomp:body}--\eqref{boundarycomp-4}, $\nu(t , x )$ is the unit normal to $\partial\mc{S}(t)$ at the point $x \in \partial\mc{S}(t)$, directed to the interior of the body.
In \eqref{momentum:comfluid} and \eqref{linear momentumcomp:body}--\eqref{angular momentumcomp:body}, $g_{\mc{F}}$ and $g_{\mc{S}}$ are the densities of the volume forces on the fluid and on the rigid body, respectively. Moreover, $\alpha >0$ is a coefficient of friction.
 Here, the notation $u \otimes v$ is the tensor product for two vectors $u,v \in \mathbb{R}^3$ and it is defined as
$u \otimes v=(u_{i}v_{j})_{1\leq i,j \leq 3}$. In the above equations, $\rho_{\mc{F}}$ and $u_{\mc{F}}$ represent respectively the mass  density and the velocity of the fluid, and the pressure of the fluid is denoted by $p_{\mc{F}}$. 

We assume that the flow is in the barotropic regime and we focus on the isentropic case where the relation between $p_{\mc{F}}$ and $\rho_{\mc{F}}$ is given by the constitutive law:
\begin{equation} \label{const-law}
p_{\mc{F}}= a_{\mc{F}}\rho_{\mc{F}}^{\gamma},
\end{equation} with $a_{\mc{F}}>0$ and the adiabatic constant $\gamma > \frac{3}{2}$, which is a necessary assumption for the existence of a  weak solution of  compressible fluids (see for example \cite{EF70}). 

As is common, we set $$\mathbb{T}(u_{\mc{F}})= 2\mu_{\mc{F}} \mathbb{D}(u_{\mc{F}}) + \lambda_{\mc{F}}\operatorname{div}u_{\mc{F}}\mathbb{I},$$ where $\mathbb{D}(u_{\mc{F}})=\frac{1}{2}\left(\nabla u_{\mc{F}} + \nabla u_{\mc{F}}^{\top}\right)$ denotes the symmetric part of the velocity gradient, $\nabla u_{\mc{F}}^{\top}$ is the transpose of the matrix $\nabla u_{\mc{F}}$, and $\lambda_{\mc{F}},\mu_{\mc{F}}$ are the viscosity coefficients satisfying 
\begin{equation*}
\mu_{\mc{F}} > 0, \quad 3\lambda_{\mc{F}} + 2\mu_{\mc{F}} \geq 0.
\end{equation*}
 The Eulerian velocity $u_{\mc{S}}(t,x)$ at each point $x\in \mc{S}(t)$ of the rigid body is given by 
\begin{equation}\label{Svel1}
u_{\mc{S}}(t,x)= h'(t)+ \omega(t) \times (x-h(t)),
\end{equation} 
where $h(t)$ is the centre of mass and $h'(t)$, $\omega(t)$  are the linear and angular velocities of the rigid body. We remark that $\mc{S}(t)$ is uniquely defined by $h(t)$, $\omega(t)$ and $\mc{S}_0$. Similarly, the knowledge of $\mc{S}(t)$ and $\mc{S}_0$ yields $h(t)$ and $\omega(t)$.
The initial velocity of the rigid body is given by 
\begin{equation}\label{Svel2}
u_{\mc{S}}(0,x)= u_{\mc{S}_0}:= \ell_0+ \omega_0 \times x,\quad x\in \mc{S}_0.
\end{equation}
Here the mass density of the body $\rho_{\mc{S}}$ satisfies the following continuity equation
\begin{equation}\label{eq:vrBeq}
    \frac{\partial \rho_{\mc{S}}}{\partial t} + u_\mc{S}\cdot\nabla \rho_{\mc{S}} = 0, \quad \forall \ (t,x)\in (0,T)\times\mc{S}(t),\quad \rho_{\mc{S}}(0,x)=\rho_{\mc{S}_0}(x),\quad \forall \ x\in \mc{S}_0.  
\end{equation}
Moreover, $m$ is the mass and $J(t)$ is the moment of inertia matrix of the solid. We express $h(t)$, $m$ and $J(t)$ in the following way:
\begin{align}
m &= \int\limits_{\mc{S}(t)} \rho_{\mc{S}} \ dx, \label{def:m} \\
    h(t) &= \frac{1}{m} \int\limits_{\mc{S}(t)} \rho_{\mc{S}} \  x \ dx, \\
    J(t) &=  \int\limits_{\mc{S}(t)} \rho_{\mc{S}}  \big[ |x-h(t)|^2\mathbb{I} - (x-h(t)) \otimes (x-h(t)) \big] \ dx. \label{def:J}
\end{align}
In the remainder of this introduction, we present the weak formulation of the system, discuss our  main result regarding the existence of weak solutions and put it in a larger perspective. 

\subsection{Weak formulation}\label{S2}
We derive a weak formulation with the help of multiplication by appropriate test functions and integration by parts by taking care of the boundary conditions. Due to the presence of the Navier-slip boundary condition, the test functions will be discontinuous across the fluid-solid interface. We introduce the set of rigid velocity fields: 
\begin{equation} \label{defR}
\mc{R}=\left\{ \zeta : \Omega \to \mathbb{R}^3 \mid \mbox{There exist   }V,  r, a \in \mathbb{R}^3 \mbox{ such that }\zeta(x)=V+ r \times \left(x-a\right)\mbox{ for any } x\in\Omega\right\}.
\end{equation}
For any $T>0$, we define the test function space $V_{T}$ as follows: 
\begin{equation}\label{def:test}
V_{T}=
\left\{\!\begin{aligned}
&\phi \in C([0,T]; L^2(\Omega))\mbox{ such that  there exist }\phi_{\mc{F}}\in \mc{D}([0,T); \mc{D}(\overline{\Omega})),\, \phi_{\mc{S}}\in \mc{D}([0,T); \mc{R})\\ &\mbox{satisfying  }\phi(t,\cdot)=\phi_{\mc{F}}(t,\cdot)\mbox{ on }\mc{F}(t),\quad \phi(t,\cdot)=\phi_{\mc{S}}(t,\cdot)\mbox{ on }\mc{S}(t)\mbox{ with }\\ &\phi_{\mc{F}}(t,\cdot)\cdot \nu = \phi_{\mc{S}}(t,\cdot)\cdot \nu \mbox{ on }\partial\mc{S}(t),\ \phi_{\mc{F}}(t,\cdot)\cdot \nu=0 \mbox{ on }\partial\Omega\mbox{ for all }t\in [0,T]
\end{aligned}\right\},
\end{equation}
where $\mc{D}$ denotes the sets of
all infinitely differentiable functions that have compact support.
We multiply equation \eqref{momentum:comfluid} by a test function $\phi\in V_T$ and integrate over $\mc{F}(t)$ to obtain
\begin{multline}\label{express1}
\frac{d}{dt} \int\limits_{\mc{F}(t)} \rho_{\mc{F}}u_{\mc{F}}\cdot \phi_{\mc{F}} - \int\limits_{\mc{F}(t)} \rho_{\mc{F}}u_{\mc{F}}\cdot \frac{\partial}{\partial t}\phi_{\mc{F}} - \int\limits_{\mc{F}(t)} (\rho_{\mc{F}}u_{\mc{F}} \otimes u_{\mc{F}}) : \nabla \phi_{\mc{F}} + \int\limits_{\mc{F}(t)} (\mathbb{T}(u_{\mc{F}}) - p_{\mc{F}}\mathbb{I}) : \mathbb{D}(\phi_{\mc{F}}) \\
= \int\limits_{\partial \Omega} (\mathbb{T}(u_{\mc{F}}) - p_{\mc{F}}\mathbb{I}) \nu\cdot \phi_{\mc{F}} + \int\limits_{\partial \mc{S}(t)} (\mathbb{T}(u_{\mc{F}}) - p_{\mc{F}}\mathbb{I}) \nu\cdot \phi_{\mc{F}} +\int\limits_{\mc{F}(t)}\rho_{\mc{F}}g_{\mc{F}}\cdot \phi_{\mc{F}}.
\end{multline}
We use the identity $(A\times B)\cdot (C\times D)=(A\cdot C)(B\cdot D)-(B\cdot C)(A\cdot D)$ to have 
\begin{equation*}
\mathbb{T}(u_{\mc{F}}) \nu\cdot \phi_{\mc{F}}= [\mathbb{T}(u_{\mc{F}}) \nu \cdot \nu](\phi_{\mc{F}}\cdot \nu) + [\mathbb{T}(u_{\mc{F}}) \nu \times \nu]\cdot (\phi_{\mc{F}}\times \nu),
\end{equation*}
\begin{equation*}
\mathbb{T}(u_{\mc{F}}) \nu\cdot \phi_{\mc{S}}= [\mathbb{T}(u_{\mc{F}}) \nu \cdot \nu](\phi_{\mc{S}}\cdot \nu) + [\mathbb{T}(u_{\mc{F}}) \nu \times \nu]\cdot (\phi_{\mc{S}}\times \nu).
\end{equation*}
Now by using the definition of $V_T$ and  the boundary conditions \eqref{boundarycomp-1}--\eqref{boundarycomp-4}, we get
\begin{equation}\label{express2}
\int\limits_{\partial \Omega} (\mathbb{T}(u_{\mc{F}}) - p_{\mc{F}}\mathbb{I}) \nu\cdot \phi_{\mc{F}} = -\alpha \int\limits_{\partial \Omega} (u_{\mc{F}}\times \nu)\cdot (\phi_{\mc{F}}\times \nu),
\end{equation}
\begin{equation}\label{express3}
\int\limits_{\partial \mc{S}(t)} (\mathbb{T}(u_{\mc{F}}) - p_{\mc{F}}\mathbb{I}) \nu\cdot \phi_{\mc{F}}  = -\alpha \int\limits_{\partial \mc{S}(t)} [(u_{\mc{F}}-u_{\mc{S}})\times \nu]\cdot [(\phi_{\mc{F}}-\phi_{\mc{S}})\times \nu] +  \int\limits_{\partial \mc{S}(t)} (\mathbb{T}(u_{\mc{F}}) - p_{\mc{F}}\mathbb{I}) \nu\cdot \phi_{\mc{S}}.
\end{equation}
Using the rigid body equations \eqref{linear momentumcomp:body}--\eqref{angular momentumcomp:body} and some calculations, we obtain
\begin{equation}\label{express4}
\int\limits_{\partial \mc{S}(t)} (\mathbb{T}(u_{\mc{F}}) - p_{\mc{F}}\mathbb{I}) \nu\cdot \phi_{\mc{S}} = -\frac{d}{dt}\int\limits_{\mc{S}(t)} \rho_{\mc{S}}u_{\mc{S}}\cdot \phi_{\mc{S}} + \int\limits_{\mc{S}(t)} \rho_{\mc{S}}u_{\mc{S}}\cdot \frac{\partial}{\partial t} \phi_{\mc{S}} + \int\limits_{\mc{S}(t)} \rho_{\mc{S}}g_{\mc{S}}\cdot \phi_{\mc{S}}.
\end{equation}
Thus by combining the above relations \eqref{express1}--\eqref{express4} and then integrating from $0$ to $T$, we have 
\begin{multline}\label{weak-momentum}
- \int\limits_0^T\int\limits_{\mc{F}(t)} \rho_{\mc{F}}u_{\mc{F}}\cdot \frac{\partial}{\partial t}\phi_{\mc{F}} - \int\limits_0^T\int\limits_{\mc{S}(t)} \rho_{\mc{S}}u_{\mc{S}}\cdot \frac{\partial}{\partial t}\phi_{\mc{S}} - \int\limits_0^T\int\limits_{\mc{F}(t)} (\rho_{\mc{F}}u_{\mc{F}} \otimes u_{\mc{F}}) : \nabla \phi_{\mc{F}} + \int\limits_0^T\int\limits_{\mc{F}(t)} (\mathbb{T}(u_{\mc{F}}) - p_{\mc{F}}\mathbb{I}) : \mathbb{D}(\phi_{\mc{F}}) \\
 + \alpha \int\limits_0^T\int\limits_{\partial \Omega} (u_{\mc{F}}\times \nu)\cdot (\phi_{\mc{F}}\times \nu) + \alpha \int\limits_0^T\int\limits_{\partial \mc{S}(t)} [(u_{\mc{F}}-u_{\mc{S}})\times \nu]\cdot [(\phi_{\mc{F}}-\phi_{\mc{S}})\times \nu] \\
 = \int\limits_0^T\int\limits_{\mc{F}(t)}\rho_{\mc{F}}g_{\mc{F}}\cdot \phi_{\mc{F}} + \int\limits_0^T\int\limits_{\mc{S}(t)} \rho_{\mc{S}}g_{\mc{S}}\cdot \phi_{\mc{S}} + \int\limits_{\mc{F}(0)} (\rho_{\mc{F}}u_{\mc{F}}\cdot \phi_{\mc{F}})(0) + \int\limits_{\mc{S}(0)} (\rho_{\mc{S}}u_{\mc{S}}\cdot \phi_{\mc{S}})(0).
\end{multline}
\begin{remark}
We stress that in the definition of the set $U_T$ (in Definition~\ref{weaksolution-main} below)
the function $u_{\mc{F}}$ on $\Omega$ is a regular  extension of the velocity field $u_{\mc{F}}$  from $\mc{F}(t)$ to $\Omega$, see \eqref{Eu1}--\eqref{Eu2}. Correspondingly, $u_{\mc{S}}\in \mc{R}$ denotes a rigid  extension from $\mc{S}(t)$ to $\Omega$ as in \eqref{Svel1}. Moreover, by the density $\rho _{\mc{F}}$ in \eqref{NO2}, we mean an extended fluid density  $\rho _{\mc{F}}$
 from $\mc{F}(t)$ to $\Omega$ by zero, see \eqref{Erho}--\eqref{15:21}. Correspondingly, $\rho_{\mc{S}}$ refers to an extended solid density from $\mc{S}(t)$ to $\Omega$ by zero.
\end{remark}
\begin{remark}\label{u-initial}
 In \eqref{NO2}, the initial fluid density $\rho _{\mc{F}_0}$ on $\Omega$ represents a zero  extension of $\rho _{\mc{F}_0}$ (defined in \eqref{initial cond})
 from $\mc{F}_0$ to $\Omega$. Correspondingly, $\rho_{\mc{S}_0}$ in equation \eqref{NO5} stands for an extended initial solid density (defined in \eqref{eq:vrBeq}) from $\mc{S}_0$ to $\Omega$ by zero. Obviously, $q_{\mc{F}_0}$ refers to an extended initial momentum from $\mc{F}_0$ to $\Omega$ by zero and $u_{\mc{S}_0}\in \mc{R}$ denotes a rigid extension from $\mc{S}_0$ to $\Omega$ as in \eqref{Svel2}.
\end{remark}
\begin{defin}\label{weaksolution-main}
Let $T> 0$, and let $\Omega$ and $\mc{S}_0 \Subset \Omega$ be two regular bounded domains of $\mathbb{R}^3$.  A triplet $(\mc{S},\rho,u)$ is a finite energy weak solution to system \eqref{mass:comfluid}--\eqref{initial cond:comp} if the following holds:
\begin{itemize}
\item $\mc{S}(t) \Subset \Omega$ is a bounded domain of $\mathbb{R}^3$ for all $t\in [0,T)$ such that 
\begin{equation}\label{NO1}
\chi_{\mc{S}}(t,x) = \mathds{1}_{\mc{S}(t)}(x) \in L^{\infty}((0,T) \times \Omega). 
\end{equation}
\item $u$ belongs to the following space

\begin{equation*}
U_{T}=
\left\{\!\begin{aligned}
&u \in L^{2}(0,T; L^2(\Omega)) \mbox{ such that there exist }u_{\mc{F}}\in L^2(0,T; H^1(\Omega)),\, u_{\mc{S}}\in L^{2}(0,T; \mc{R})\\ &\mbox{satisfying }u(t,\cdot)=u_{\mc{F}}(t,\cdot)\mbox{ on }\mc{F}(t),\quad u(t,\cdot)=u_{\mc{S}}(t,\cdot)\mbox{ on }\mc{S}(t)\mbox{ with }\\ &u_{\mc{F}}(t,\cdot)\cdot \nu = u_{\mc{S}}(t,\cdot)\cdot \nu \mbox{ on }\partial\mc{S}(t),\ u_{\mc{F}}\cdot \nu=0 \mbox{ on }\partial\Omega\mbox{ for a.e }t\in [0,T]
\end{aligned}\right\}.
\end{equation*}
\item $\rho \geq 0$, $\rho \in L^{\infty}(0,T; L^{\gamma}(\Omega))$ with $\gamma>3/2$, $\rho|u|^2 \in L^{\infty}(0,T; L^1(\Omega))$, where
\begin{equation*}
\rho= (1-\mathds{1}_{\mc{S}})\rho_{\mc{F}} + \mathds{1}_{\mc{S}}\rho_{\mc{S}},\quad u= (1-\mathds{1}_{\mc{S}})u_{\mc{F}} + \mathds{1}_{\mc{S}}u_{\mc{S}}.
\end{equation*}  
\item The continuity equation is satisfied in the weak sense, i.e.\ 
\begin{equation}\label{NO2} 
\frac{\partial {\rho_{\mc{F}}}}{\partial t} + \operatorname{div}({\rho}_{\mc{F}} u_{\mc{F}}) =0 \mbox{ in }\, \mc{D}'([0,T)\times {\Omega}),\quad \rho_{\mc{F}}(0,x)=\rho_{\mc{F}_0}(x),\ x\in \Omega.
\end{equation} 
Also, a renormalized continuity equation holds in a weak sense, i.e.\ 
\begin{equation}\label{NO3} 
\partial_t b(\rho_{\mc{F}}) + \operatorname{div}(b(\rho_{\mc{F}})u_{\mc{F}}) + (b'(\rho_{\mc{F}})-b(\rho_{\mc{F}}))\operatorname{div}u_{\mc{F}}=0 \mbox{ in }\, \mc{D}'([0,T)\times {\Omega}) ,
\end{equation} 
for any $b\in C([0,\infty)) \cap C^1((0,\infty))$ satisfying
\begin{equation}\label{eq:b}
|b'(z)|\leq cz^{-\kappa_0},\, z\in (0,1],\ \kappa_0 <1, \qquad  |b'(z)|\leq cz^{\kappa_1},\, z\geq 1,\ -1<\kappa_1 <\infty.
\end{equation}
\item The transport of $\mc{S}$ by the rigid vector field $u_{\mc{S}}$ holds (in the weak sense)
\begin{equation}\label{NO4}
\frac{\partial {\chi}_{\mc{S}}}{\partial t} + \operatorname{div}(u_{\mc{S}}\chi_{\mc{S}}) =0 \, \mbox{ in }(0,T)\times {\Omega},\quad \chi_{\mc{S}}(0,x)=\mathds{1}_{\mc{S}_0}(x),\ x\in \Omega.
\end{equation}
\item The density $\rho_{\mc{S}}$ of the rigid body $\mc{S}$ satisfies (in the weak sense)
\begin{equation}\label{NO5}
\frac{\partial {\rho}_{\mc{S}}}{\partial t} + \operatorname{div}(u_{\mc{S}}\rho_{\mc{S}}) =0 \, \mbox{ in }(0,T)\times {\Omega},\quad \rho_{\mc{S}}(0,x)=\rho_{\mc{S}_0}(x),\ x\in \Omega.
\end{equation}
\item Balance of linear momentum holds in a weak sense, i.e.\ for all  $\phi \in V_{T}$ the relation \eqref{weak-momentum} holds.
\item The following energy inequality holds for almost every $t\in (0,T)$:
\begin{multline}\label{energy}
\int\limits_{\mc{F}(t)}\frac{1}{2} \rho_{\mc{F}}|u_{\mc{F}}(t,\cdot)|^2 + \int\limits_{\mc{S}(t)} \frac{1}{2} \rho_{\mc{S}}|u_{\mc{S}}(t,\cdot)|^2 + \int\limits_{\mc{F}(t)} \frac{a_{\mc{F}}}{\gamma-1}\rho_{\mc{F}}^{\gamma}+ \int\limits_0^t\int\limits_{\mc{F}(\tau)} \Big(2\mu_{\mc{F}} |\mathbb{D}(u_{\mc{F}})|^2 + \lambda_{\mc{F}} |\operatorname{div} u_{\mc{F}}|^2\Big) \\
 + \alpha \int\limits_0^t\int\limits_{\partial \Omega} |u_{\mc{F}}\times \nu|^2 
+ \alpha \int\limits_0^t\int\limits_{\partial \mc{S}(\tau)} |(u_{\mc{F}}-u_{\mc{S}})\times \nu|^2\\ \leq \int\limits_0^t\int\limits_{\mc{F}(\tau)}\rho_{\mc{F}}g_{\mc{F}}\cdot u_{\mc{F}} + \int\limits_0^t\int\limits_{\mc{S}(\tau)} \rho_{\mc{S}}g_{\mc{S}}\cdot u_{\mc{S}} + \int\limits_{\mc{F}_0}\frac{1}{2} \frac{|q_{\mc{F}_0}|^2}{\rho_{\mc{F}_0}}
   + \int\limits_{\mc{S}_0} \frac{1}{2}\rho_{\mc{S}_0}|u_{\mc{S}_0}|^2 + \int\limits_{\mc{F}_0} \frac{a_{\mc{F}}}{\gamma-1}\rho_{\mc{F}_0}^{\gamma}.
\end{multline}
\end{itemize}
  \end{defin}
  \begin{remark}\label{16:37}

  We note that our continuity equation \eqref{NO2} is different from the corresponding one in \cite{F4}. We have to work with $u_{\mc{F}}$ instead of $u$ because of the  Navier boundary condition. The reason is that we need the  $H^{1}(\Omega)$ regularity of the velocity in order to achieve the validity of the continuity equation in $\Omega$. Observe that $u \in L^{2}(0,T; L^2(\Omega))$ but the  extended fluid velocity has better regularity, in particular, $u_{\mc{F}}\in L^2(0,T; H^1(\Omega))$, see \eqref{Eu1}--\eqref{Eu2}.
  \end{remark}
  \begin{remark}
  In the weak formulation \eqref {weak-momentum}, we need to distinguish between the  fluid velocity $u_{\mc{F}}$ and the solid velocity $u_{\mc{S}}$. Due to the presence of the  discontinuities  in the tangential components of $u$  and $\phi$, neither $\partial_t \phi$ nor $\mb{D}(u)$, $\mb{D}(\phi)$ belong to $L^2(\Omega)$. That's why it is not possible to write \eqref {weak-momentum} in a global and condensed form (i.e.\ integrals over $\Omega$).
   \end{remark}
   
   \begin{remark}
   Let us mention that in the whole paper we assume the regularity of domains $\Omega$ and $\mc{S}_0$ as  $C^{2+\kappa}$, $\kappa >0$. However, we expect that our assumption on the regularity of the domain can be relaxed to a less regular domain like in the work of Kuku\v cka \cite{kuku}. 
   \end{remark}
    \subsection{Discussion and main result}
The mathematical analysis of systems describing the motion of a rigid body in a viscous \textit{incompressible} fluid is nowadays well developed.  The proof of existence of weak solutions until a first collision can be found in several papers, see \cite{CST,DEES1,GLSE,HOST,SER3}. Later, the possibility of collisions in the case of a weak solution was included, see \cite{F3,SST}.  Moreover, it was shown that  under Dirichlet boundary conditions collisions cannot occur, which is paradoxical with respect to real situations; for details see \cite{HES, HIL, HT}. Neustupa and Penel showed that under a prescribed motion of the rigid body and under Navier type of boundary conditions collision can occur  \cite{NP}.  After that G\'erard-Varet and Hillairet showed that to construct collisions  one needs    to assume less regularity of the domain or different boundary conditions, see e.g.\ \cite{GH,MR3272367,GHC}. In the case of very high viscosity, under the assumption that rigid bodies are not touching each other or not touching the boundary at the initial time, it was shown that collisions cannot occur in finite time, see \cite{FHN}.
For an introduction we refer to the problem of a fluid coupled with a rigid body  in the  work by Galdi, see \cite{G2}.
Let us also mention results on strong solutions, see e.g.\ \cite{GGH13,T, Wa}. 
 
A few results are available on the motion of a rigid structure in a \textit{compressible} fluid with  Dirichlet boundary conditions.  The existence of strong solutions in the  $L^2$ framework for small data up to a collision was shown in \cite{BG,roy2019stabilization}. The existence of strong solutions in the  $L^p$ setting based on $\mathcal{R}$-bounded operators was applied in the barotropic case \cite{HiMu} and   in the full system \cite{HMTT}.

The existence of a weak solution, also up to a collision but without smallness assumptions, was shown in \cite{DEES2}. Generalization of this result allowing collisions was given in \cite{F4}.
The weak-strong uniqueness of a compressible fluid with a rigid body can be found in  \cite{KrNePi2}. Existence of weak solutions in the case of Navier boundary conditions is not available; we explore it in this article.


 For many years, the \emph{no-slip boundary
condition}
has been the most widely used given its success in reproducing the
standard velocity profiles for incompressible/compressible viscous fluids.
Although the no-slip hypothesis seems to be in good agreement with
experiments, it leads to certain rather surprising conclusions. As we have already mentioned before, the
most striking one being the absence of collisions of rigid objects
immersed in a linearly viscous fluid \cite{HES,HIL}.

The so-called \emph{Navier boundary
conditions}, which allow for slip, 
offer more freedom and are
likely to provide a physically acceptable solution at least to some
of the paradoxical phenomena resulting from the no-slip boundary
condition, see, e.g.\ Moffat \cite{MOF}. Mathematically, the behavior of the tangential component $[{u}]_{tan}$ is a
delicate issue.




\smallskip

The main result of our paper (Theorem~\ref{exist:upto collision}) asserts local-in-time existence of a weak solution for the system involving the motion of a rigid body in a compressible fluid in the case of Navier boundary conditions at the interface and at the outer boundary. It is the first result in the context of rigid body-compressible fluid interaction in the case of Navier type of boundary conditions. Let us mention that the main difficulty which arises in our problem is the jump  in the velocity through the interface boundary between the rigid body and the compressible fluid. This difficulty cannot be resolved by the approach introduced in the work of Desjardins, Esteban \cite{DEES2}, or Feireisl \cite{F4} since they consider the velocity field continuous through the interface. Moreover, G\'erard-Varet, Hillairet \cite{MR3272367} and  Chemetov, Ne\v casov\' a \cite{CN} cannot be used directly as they are in the  incompressible framework. 
Our weak solutions  have to satisfy the jump of the velocity field through the interface boundary. 

Our idea is to introduce a new  approximate scheme which combines the theory of compressible fluids introduced by P. L. Lions \cite{LI4} and then developed by Feireisl \cite{EF70} to get the strong convergence of the density (renormalized continuity equations, effective viscous flux, artificial pressure) together with ideas from  G\'erard-Varet, Hillairet \cite{MR3272367} and  Chemetov, Ne\v casov\' a \cite{CN} concerning a penalization of the jump.   We remark that such type of difficulties do not arise for the existence of weak solutions of   compressible fluids without a rigid body neither for Dirichlet nor for Navier type of boundary conditions. 

  Let us mention the main issues that arise in the analysis of our system and the methodology that we adapt to deal with it:
  \begin{itemize}
  \item It is not possible to define a uniform velocity field as in \cite{DEES2, F4}
  due to the presence of a discontinuity through the interface of interaction. This is the reason why we introduce the regularized fluid velocity $u_{\mc{F}}$ and the solid velocity $u_{\mc{S}}$ and why we treat them separately.
  \item We introduce approximate problems and recover the original problem as a limit of the approximate ones. In fact, we consider several levels of approximations; in each level we ensure that our solution and the test function do not show a jump across the interface so that we can use several available techniques of compressible fluids (without body). In the limit, however, the discontinuity at the interface is recovered. The particular construction of the test functions is a delicate and crucial issue in our proof, which we outline in \cref{approx:test}. 
  \item Recovering the velocity fields for the solid and fluid parts separately is also a challenging issue. We introduce a penalization in such a way that, in the last stage of the limiting process, this term allows us to recover the rigid velocity of the solid, see \eqref{12:04}--\eqref{re:solidvel}.
  The introduction of an appropriate extension operator helps us to recover the fluid velocity, see \eqref{ext:fluid}--\eqref{re:fluidvel}.
  \item Since we consider the compressible case, our penalization with parameter $\delta>0$, see \eqref{approx3}, is  different from  the penalization for the incompressible fluid in  \cite{MR3272367}. 
   
  \item Due to the Navier-slip boundary condition, no $H^1$ bound on the velocity  on the whole domain is possible. We can only obtain the $H^1$ regularity of the extended velocities of  the fluid and solid parts separately. We have introduced an artificial viscosity that vanishes asymptotically on the solid part so that we can capture the $H^1$ regularity for the fluid part (see step 1 of the proof of \cref{exist:upto collision} in Section \ref{S4}).
  \item We have already mentioned that the main difference with \cite{MR3272367} is that we consider compressible fluid whereas they have considered an incompressible fluid. We have encountered several issues that are present due to the compressible nature of the fluid (vanishing viscosity in the continuity equation, recovering the renormalized continuity equation, identification of the pressure). One important point is to see that passing to the limit as $\delta$ tends to zero in the transport for the rigid body is not obvious because our velocity field does not have regularity $L^{\infty}(0,T,L^2(\Omega))$ as in the incompressible case see e.g.\ \cite{MR3272367}  but  $L^2(0,T,L^2(\Omega))$ (here we have $\sqrt{\rho}u \in L^{\infty}(0,T,L^2(\Omega))$ only). To handle this problem, we apply \cref{sequential2} in the $\delta$-level, see \cref{S4}.
\end{itemize}

  
  Next we present the main result of our paper. 
  \begin{theorem}\label{exist:upto collision}
  Let $\Omega$ and $\mc{S}_0 \Subset \Omega$ be two regular bounded domains of $\mathbb{R}^3$. Assume that for some $\sigma > 0$
  \begin{equation*}
  \operatorname{dist}(\mc{S}_0,\partial\Omega) > 2\sigma.
  \end{equation*}
  Let  $g_{\mc{F}}$, $g_{\mc{S}} \in L^{\infty}((0,T)\times \Omega)$ and the pressure $p_{\mc{F}}$ be determined by \eqref{const-law} with $\gamma >3/2$. Assume that the initial data (defined in the sense of \cref{u-initial}) satisfy 
  \begin{align} \label{init}
  \rho_{\mc{F}_{0}} \in L^{\gamma}(\Omega),\quad \rho_{\mc{F}_{0}} \geq 0 &\mbox{ a.e. in }\Omega,\quad \rho_{\mc{S}_0}\in L^{\infty}(\Omega),\quad \rho_{\mc{S}_0}>0\mbox{ a.e. in }\mc{S}_0,\\
  \label{init1}
  q_{\mc{F}_{0}} \in L^{\frac{2\gamma}{\gamma+1}}(\Omega), \quad q_{\mc{F}_{0}}\mathds{1}_{\{\rho_{\mc{F}_0}=0\}}=0 &\mbox{ a.e. in }\Omega,\quad \dfrac{|q_{\mc{F}_{0}}|^2}{\rho_{\mc{F}_0}}\mathds{1}_{\{\rho_{\mc{F}_0}>0\}}\in L^1(\Omega),\\
  \label{init2}
  u_{\mc{S}_0}= \ell_0+ \omega_0 \times x& \quad\forall\ x \in \Omega \mbox{ with }\ell_0,\ \omega_0 \in \mb{R}^3.
  \end{align}
  Then there exists $T > 0$  (depending only on $\rho_{\mc{F}_0}$, $\rho_{\mc{S}_0}$, $q_{\mc{F}_0}$, $u_{\mc{S}_0}$, $g_{\mc{F}}$, $g_{\mc{S}}$,  $\operatorname{dist}(\mc{S}_0,\partial\Omega)$) such that   a finite energy weak solution to \eqref{mass:comfluid}--\eqref{initial cond:comp} exists on $[0,T)$. Moreover, 
\begin{equation*}
  \mc{S}(t) \Subset \Omega,\quad\operatorname{dist}(\mc{S}(t),\partial\Omega) \geq \frac{3\sigma}{2},\quad \forall \ t\in [0,T].
  \end{equation*}
  \end{theorem}
  The outline of the paper is as follows. We introduce three levels of  approximation schemes in \cref{F1351}. In \cref{S5}, we describe some results on the transport equation, which are needed in all the levels of approximation. The existence results of approximate solutions have been proved in \cref{S3}. \cref{sec:Galerkin} and \cref{14:14} are dedicated to the construction  and convergence analysis of the Faedo-Galerkin scheme associated to the finite dimensional approximation level. We discuss the limiting system associated to the vanishing viscosity in \cref{14:18}. \cref{S4} is devoted to the main part: we derive the limit as the parameter $\delta$ tends  to zero.

\section{Approximate Solutions}\label{F1351}

In this section, we present the approximate problems by combining the penalization method, introduced in \cite{MR3272367}, and the approximation scheme developed in \cite{MR1867887}
along with a careful treatment of the boundary terms of the rigid body to solve the original problem \eqref{mass:comfluid}--\eqref{initial cond:comp}. There are three levels of approximations with the parameters $N,\varepsilon,\delta$.
Let us briefly explain these approximations:
\begin{itemize}
    \item  The parameter $N$ is connected with solving the momentum equation using the Faedo-Galerkin approximation. 
    \item The parameter $\varepsilon > 0$ is connected with a new diffusion term $\varepsilon \Delta \rho $ in the continuity equation together with a term $\varepsilon \nabla \rho \nabla u$ in the momentum equation.
    \item The parameter $\delta > 0$ is connected with the approximation in the viscosities (see \eqref{approx-viscosity}) together with a penalization of the boundary of the rigid body to get smoothness through the interface (see \eqref{approx3}) and together with the artificial pressure containing the approximate coefficient, see \eqref{approx-press}.
\end{itemize}
At first, we state the existence results for the different levels of approximation schemes and then we will prove these later on. We start with the $\delta$-level of approximation via an artificial pressure. 
We are going to consider the following approximate problem: Let $\delta>0$. Find a triplet $(S^{\delta}, \rho^{\delta}, u^{\delta})$ such that 
\begin{itemize}
\item  $\mc{S}^{\delta}(t) \Subset \Omega$ is a bounded, regular domain for all $t \in [0,T]$  with 
\begin{equation}\label{approx1}
\chi^{\delta}_{\mc{S}}(t,x)= \mathds{1}_{\mc{S}^{\delta}(t)}(x) \in L^{\infty}((0,T)\times \Omega) \cap C([0,T];L^p(\Omega)),  \, \forall \, 1 \leq p < \infty.
\end{equation}
\item The velocity field $u^{\delta} \in  L^2(0,T; H^1(\Omega))$, and the density function $\rho^{\delta} \in L^{\infty}(0,T; L^{\beta}(\Omega))$, $\rho^{\delta}\geq 0$   satisfy
\begin{equation}\label{approx2} 
\frac{\partial {\rho}^{\delta}}{\partial t} + \operatorname{div}({\rho}^{\delta} u^{\delta}) =0 \mbox{ in } \mc{D}'([0,T)\times {\Omega}).
\end{equation} 
\item 
For all $\phi \in H^1(0,T; L^{2}(\Omega)) \cap L^r(0,T; W^{1,{r}}(\Omega))$, where $r=\max\left\{\beta+1, \frac{\beta+\theta}{\theta}\right\}$, $\beta \geq \max\{8,\gamma\}$ and $\theta=\frac{2}{3}\gamma -1$ with $\phi\cdot\nu=0$ on $\partial\Omega$ and  $\phi|_{t=T}=0$, the following holds:
\begin{multline}\label{approx3}
- \int\limits_0^T\int\limits_{\Omega} \rho^{\delta} \left(u^{\delta}\cdot \frac{\partial}{\partial t}\phi +  u^{\delta} \otimes u^{\delta} : \nabla \phi\right) + \int\limits_0^T\int\limits_{\Omega} \Big(2\mu^{\delta}\mathbb{D}(u^{\delta}):\mathbb{D}(\phi) + \lambda^{\delta}\operatorname{div}u^{\delta}\mathbb{I} : \mathbb{D}(\phi)- p^{\delta}(\rho^{\delta})\mathbb{I} : \mathbb{D}(\phi)\Big) \\
 + \alpha \int\limits_0^T\int\limits_{\partial \Omega} (u^{\delta} \times \nu)\cdot (\phi \times \nu) + \alpha \int\limits_0^T\int\limits_{\partial \mc{S}^{\delta}(t)} [(u^{\delta}-P^{\delta}_{\mc{S}}u^{\delta})\times \nu]\cdot [(\phi-P^{\delta}_{\mc{S}}\phi)\times \nu]
  + \frac{1}{\delta}\int\limits_0^T\int\limits_{\Omega} \chi^{\delta}_{\mc{S}}(u^{\delta}-P^{\delta}_{\mc{S}}u^{\delta})\cdot (\phi-P^{\delta}_{\mc{S}}\phi)\\ = \int\limits_0^T\int\limits_{\Omega}\rho^{\delta} g^{\delta} \cdot \phi
 + \int\limits_{\Omega} (\rho^{\delta} u^{\delta} \cdot \phi)(0),
\end{multline}
where $\mc{P}_{\mc{S}}^{\delta}$ is defined in \eqref{approx:projection} below.
\item ${\chi}^{\delta}_{\mc{S}}(t,x)$ satisfies (in the weak sense)
\begin{equation}\label{approx4}
\frac{\partial {\chi}^{\delta}_{\mc{S}}}{\partial t} + P^{\delta}_{\mc{S}}u^{\delta} \cdot \nabla \chi^{\delta}_{\mc{S}} =0\, \mbox{ in }(0,T)\times {\Omega},\quad  \chi^{\delta}_{\mc{S}}|_{t=0}= \mathds{1}_{\mc{S}_0}\, \mbox{ in } {\Omega}.
\end{equation}
\item $\rho^{\delta}{\chi}^{\delta}_{\mc{S}}(t,x)$ satisfies (in the weak sense)
\begin{equation}\label{approx5}
\frac{\partial }{\partial t}(\rho^{\delta}{\chi}^{\delta}_{\mc{S}}) + P^{\delta}_{\mc{S}}u^{\delta} \cdot \nabla (\rho^{\delta}{\chi}^{\delta}_{\mc{S}})=0\, \mbox{ in }(0,T)\times {\Omega},\quad  (\rho^{\delta}{\chi}^{\delta}_{\mc{S}})|_{t=0}= \rho_0^{\delta}\mathds{1}_{\mc{S}_0}\, \mbox{ in } {\Omega}.
\end{equation}
\item Initial data are given by
\begin{equation}\label{approx:initial}
\rho^{\delta}(0,x)=\rho_0^{\delta}(x),\quad \rho^{\delta} u^{\delta}(0,x) = q_0^{\delta}(x),\quad x\in \Omega .
\end{equation}
\end{itemize}
Above we have used the following quantities:
\begin{itemize}
\item The density of the volume force is defined as
\begin{equation*}
g^{\delta}=(1-\chi^{\delta}_{\mc{S}})g_{\mc{F}} + \chi^{\delta}_{\mc{S}}g_{\mc{S}}.
\end{equation*}
\item The artificial pressure is given by 
\begin{equation}\label{approx-press}
p^{\delta}(\rho)= a^{\delta}\rho^{\gamma} + {\delta} \rho^{\beta},\quad\mbox{ with }\quad a^{\delta} = a_{\mc{F}} (1-\chi^{\delta}_{\mc{S}}),
\end{equation}
where $a_{\mc{F}}>0$ and  the exponents $\gamma$ and $\beta$ satisfy $\gamma > 3/2, \ \beta \geq \max\{8,\gamma\}$.
\item The viscosity coefficients are given by
\begin{equation}\label{approx-viscosity}
 \mu^{\delta} = (1-\chi^{\delta}_{\mc{S}})\mu_{\mc{F}} + {\delta}^2\chi^{\delta}_{\mc{S}},\quad  \lambda^{\delta} = (1-\chi^{\delta}_{\mc{S}})\lambda_{\mc{F}} + {\delta}^2\chi^{\delta}_{\mc{S}}\quad\mbox{ so that }\quad\mu^{\delta} >0,\ 2\mu^{\delta}+ 3\lambda^{\delta} \geq 0.
\end{equation}
\item The orthogonal projection to rigid fields, $P^{\delta}_{\mc{S}}:L^2(\Omega)\rightarrow L^2(\mc{S}^{\delta}(t))$,  is such that, for all $t\in [0,T]$ and $u\in L^2(\Omega)$, we have $P^{\delta}_{\mc{S}}u \in \mc{R}$ and it is given by
\begin{equation}\label{approx:projection}
P^{\delta}_{\mc{S}}u(t,x)= \frac{1}{m^{\delta}} \int\limits_{\Omega} \rho^{\delta}\chi_{\mc{S}}^{\delta} u + \left((J^{\delta})^{-1} \int\limits_{\Omega}\rho^{\delta}\chi_{\mc{S}}^{\delta}((y-h^{\delta}(t)) \times u)\ dy \right)\times (x-h^{\delta}(t)), \quad \forall x\in \Omega,
\end{equation}
where $h^{\delta}$, $m^{\delta}$ and $J^{\delta}$ are defined as
\begin{align*}
    h^{\delta}(t) = \frac{1}{m^{\delta}} \int\limits_{\mathbb{R}^3} \rho^{\delta}\chi_{\mc{S}}^{\delta} x \ dx,\quad m^{\delta} = \int\limits_{\mathbb{R}^3} \rho^{\delta}\chi_{\mc{S}}^{\delta} \ dx, \\
    J^{\delta}(t) = \int\limits_{\mathbb{R}^3} \rho^{\delta}\chi_{\mc{S}}^{\delta}  \left[ |x-h^{\delta}(t)|^2\mathbb{I} - (x-h^{\delta}(t)) \otimes (x-h^{\delta}(t))\right] \ dx. 
\end{align*}
\end{itemize}
\begin{remark}
The penalization which we apply in our case is different from that in \cite{F4}. We do not use the high viscosity limit but our penalization  contains an $L^2$ penalization (see \eqref{approx3}), which is necessary because of  the discontinuity of the velocity field through the fluid-structure interface. Moreover, we consider a penalization of the  viscosity coefficients \eqref{approx-viscosity} together with  the additional regularity of the pressure, see \eqref{approx-press}. This approach is completely new. 
\end{remark}
A weak solution of problem \eqref{mass:comfluid}--\eqref{initial cond:comp} in the sense of \cref{weaksolution-main} will be obtained as a weak limit of  the solution $(\mc{S}^{\delta},\rho^{\delta},u^{\delta})$ of  system \eqref{approx1}--\eqref{approx:initial} as $\delta \rightarrow 0$. The existence result of the approximate system reads:
\begin{proposition}\label{thm:approxn-delta}
Let $\Omega$ and $\mc{S}_0 \Subset \Omega$ be two regular bounded domains of $\mathbb{R}^3$. Assume that for some $\sigma>0$
  \begin{equation*}
  \operatorname{dist}(\mc{S}_0,\partial\Omega) > 2\sigma.
  \end{equation*}
 Let  $g^{\delta}=(1-\chi^{\delta}_{\mc{S}})g_{\mc{F}} + \chi^{\delta}_{\mc{S}}g_{\mc{S}} \in L^{\infty}((0,T)\times \Omega)$ and
\begin{equation}\label{cc:dbg}
{\delta} >0,\ \gamma > 3/2, \ \beta \geq \max\{8,\gamma\}.
\end{equation}
Further, let the pressure $p^{\delta}$ be determined by \eqref{approx-press} and the viscosity coefficients $\mu^{\delta}$, $\lambda^{\delta}$ be given by \eqref{approx-viscosity}. Assume that the initial conditions satisfy 
\begin{align}
\rho_{0}^{\delta} &\in L^{\beta}(\Omega), \quad \rho_0^{\delta}\geq 0 \mbox{ a.e. in }\Omega,\quad \rho_0^{\delta}\mathds{1}_{\mc{S}_0}\in L^{\infty}(\Omega),\quad \rho_0^{\delta}\mathds{1}_{\mc{S}_0}>0\mbox{ a.e. in }\mc{S}_0,\label{rhonot}\\
&q_0^{\delta} \in L^{\frac{2\beta}{\beta+1}}(\Omega), \quad q_0^{\delta}\mathds{1}_{\{\rho_0^{\delta}=0\}}=0 \mbox{ a.e. in }\Omega,\quad \dfrac{|q_0^{\delta}|^2}{\rho_0^{\delta}}\mathds{1}_{\{\rho_0>0\}}\in L^1(\Omega)\label{qnot}.
\end{align}
Let the initial energy  
$$
E^{\delta}[\rho_0^{\delta},q_0^{\delta}] = \int\limits_{\Omega}\Bigg(\frac{1}{2} \frac{|q_0^{\delta}|^2}{\rho_0^{\delta}}\mathds{1}_{\{\rho_0^{\delta}>0\}} + \frac{a^{\delta}(0)}{\gamma-1}(\rho_0^{\delta})^{\gamma} + \frac{\delta}{\beta-1}(\rho_0^{\delta})^{\beta} \Bigg):=E^{\delta}_0$$
be uniformly bounded with respect to $\delta$. Then there exists $T > 0$ (depending only on $E^{\delta}_0$, $g_{\mc{F}}$, $g_{\mc{S}}$,  $\operatorname{dist}(\mc{S}_0,\partial\Omega)$) such that system \eqref{approx1}--\eqref{approx:initial} admits a  finite energy weak solution $(S^{\delta},\rho^{\delta},u^{\delta})$, which satisfies the following energy inequality:
\begin{multline}\label{10:45}
E^{\delta}[\rho ^{\delta}, q^{\delta}] +  \int\limits_0^T\int\limits_{\Omega} \Big(2\mu^{\delta}|\mathbb{D}(u^{\delta})|^2 + \lambda^{\delta}|\operatorname{div}u^{\delta}|^2\Big)  + \alpha \int\limits_0^T\int\limits_{\partial \Omega} |u^{\delta} \times \nu|^2
 + \alpha \int\limits_0^T\int\limits_{\partial \mc{S}^{\delta}(t)} |(u^{\delta}-P^{\delta}_{\mc{S}}u^{\delta})\times \nu|^2 \\+ \frac{1}{\delta}\int\limits_0^T\int\limits_{\Omega} \chi^{\delta}_{\mc{S}}|u^{\delta}-P^{\delta}_{\mc{S}}u^{\delta}|^2 \leq \int\limits_0^T\int\limits_{\Omega}\rho^{\delta}
 g^{\delta} \cdot u^{\delta}
 +  E_0^{\delta}.
\end{multline}
Moreover, 
\begin{equation*}
  \operatorname{dist}(\mc{S}^{\delta}(t),\partial\Omega) \geq {2\sigma},\quad \forall \ t\in [0,T],
  \end{equation*}
and the solution satisfies the following properties:
\begin{enumerate}
\item For $\theta=\frac{2}{3}\gamma-1$, $s=\gamma+\theta$,
\begin{equation}\label{rho:improved}
\|(a^{\delta})^{1/s}\rho^{\delta}\|_{L^{s}((0,T)\times\Omega)} + \delta^{\frac{1}{\beta+\theta}}\|\rho^{\delta}\|_{L^{\beta+\theta}((0,T)\times\Omega)} \leq c.
\end{equation}
\item The couple $(\rho^{\delta},u^{\delta})$ satisfies the identity
\begin{equation}\label{rho:renorm1}
\partial_t b(\rho^{\delta}) + \operatorname{div}(b(\rho^{\delta})u^{\delta})+[b'(\rho^{\delta})\rho^{\delta} - b(\rho^{\delta})]\operatorname{div}u^{\delta}=0,
\end{equation}
a.e.\ in $(0,T)\times\Omega$ for any $b\in C([0,\infty)) \cap C^1((0,\infty))$ satisfying \eqref{eq:b}.
\end{enumerate}
\end{proposition}
In order to prove \cref{thm:approxn-delta}, we consider a problem with another level of approximation: the $\varepsilon$-level approximation is obtained  via the continuity equation with dissipation accompanied by the artificial pressure in the momentum equation. 
We want to find a triplet $(S^{\varepsilon}, \rho^{\varepsilon}, u^{\varepsilon})$ such that we can obtain a weak solution $(\mc{S}^{\delta},\rho^{\delta},u^{\delta})$ of the system \eqref{approx1}--\eqref{approx:initial} as a weak limit of the sequence $(\mc{S}^{\varepsilon},\rho^{\varepsilon},u^{\varepsilon})$ as $\varepsilon \rightarrow 0$. For $\varepsilon >0$, the triplet is supposed to satisfy:
\begin{itemize}
\item  $\mc{S}^{\varepsilon}(t) \Subset \Omega$ is a bounded, regular domain for all $t \in [0,T]$  with 
\begin{equation}\label{varepsilon:approx1}
\chi^{\varepsilon}_{\mc{S}}(t,x)= \mathds{1}_{\mc{S}^{\varepsilon}(t)}(x) \in L^{\infty}((0,T)\times \Omega) \cap C([0,T];L^p(\Omega)),  \, \forall \,  1 \leq p < \infty.
\end{equation}
\item The velocity field $u^{\varepsilon} \in L^2(0,T; H^1(\Omega))$ and the density function $\rho^{\varepsilon} \in L^{\infty}(0,T; L^{\beta}(\Omega)) \cap L^2(0,T;H^1(\Omega))$, $\rho^{\varepsilon}\geq 0$ satisfy
\begin{equation}\label{varepsilon:approx2} 
\frac{\partial {\rho}^{\varepsilon}}{\partial t} + \operatorname{div}({\rho}^{\varepsilon} u^{\varepsilon}) =\varepsilon \Delta\rho^{\varepsilon} \mbox{ in }\, (0,T)\times \Omega, \quad  \frac{\partial \rho^{\varepsilon}}{\partial \nu}=0 \mbox{ on }\, \partial\Omega.
\end{equation} 
\item For all $\phi \in H^1(0,T; L^{2}(\Omega)) \cap L^{\beta+1}(0,T; W^{1,{\beta+1}}(\Omega))$ with $\phi\cdot \nu=0$ on $\partial\Omega$,  $\phi|_{t=T}=0$, where $\beta \geq \max\{8,\gamma\}$, the following holds:
\begin{multline}\label{varepsilon:approx3}
- \int\limits_0^T\int\limits_{\Omega} \rho^{\varepsilon} \left(u^{\varepsilon}\cdot \frac{\partial}{\partial t}\phi +  u^{\varepsilon} \otimes u^{\varepsilon} : \nabla \phi\right) + \int\limits_0^T\int\limits_{\Omega} \Big(2\mu^{{\varepsilon}}\mathbb{D}(u^{\varepsilon}):\mathbb{D}(\phi) + \lambda^{{\varepsilon}}\operatorname{div}u^{\varepsilon}\mathbb{I} : \mathbb{D}(\phi) - p^{\varepsilon}(\rho^{\varepsilon})\mathbb{I} : \mathbb{D}(\phi)\Big) \\
+\int\limits_0^T\int\limits_{\Omega} \varepsilon \nabla u^{\varepsilon} \nabla \rho^{\varepsilon} \cdot \phi
 + \alpha \int\limits_0^T\int\limits_{\partial \Omega} (u^{\varepsilon} \times \nu)\cdot (\phi \times \nu) + \alpha \int\limits_0^T\int\limits_{\partial \mc{S}^{\varepsilon}(t)} [(u^{\varepsilon}-P^{\varepsilon}_{\mc{S}}u^{\varepsilon})\times \nu]\cdot [(\phi-P^{\varepsilon}_{\mc{S}}\phi)\times \nu] \\
  + \frac{1}{\delta}\int\limits_0^T\int\limits_{\Omega} \chi^{\varepsilon}_{\mc{S}}(u^{\varepsilon}-P^{\varepsilon}_{\mc{S}}u^{\varepsilon})\cdot (\phi-P^{\varepsilon}_{\mc{S}}\phi) = \int\limits_0^T\int\limits_{\Omega}\rho^{\varepsilon} g^{\varepsilon} \cdot \phi
 + \int\limits_{\Omega} (\rho^{\varepsilon} u^{\varepsilon} \cdot \phi)(0).
\end{multline}
\item ${\chi}^{\varepsilon}_{\mc{S}}(t,x)$ satisfies (in the weak sense)
\begin{equation}\label{varepsilon:approx4}
\frac{\partial {\chi}^{\varepsilon}_{\mc{S}}}{\partial t} + P^{\varepsilon}_{\mc{S}}u^{\varepsilon} \cdot \nabla \chi^{\varepsilon}_{\mc{S}} =0\mbox{ in }(0,T)\times {\Omega},\quad  \chi^{\varepsilon}_{\mc{S}}|_{t=0}= \mathds{1}_{\mc{S}_0}\mbox{ in } {\Omega}.
\end{equation}
\item $\rho^{\varepsilon}{\chi}^{\varepsilon}_{\mc{S}}(t,x)$ satisfies (in the weak sense)
\begin{equation}\label{varepsilon:approx5}
\frac{\partial }{\partial t}(\rho^{\varepsilon}{\chi}^{\varepsilon}_{\mc{S}}) + P^{\varepsilon}_{\mc{S}}u^{\varepsilon} \cdot \nabla (\rho^{\varepsilon}{\chi}^{\varepsilon}_{\mc{S}})=0\mbox{ in }(0,T)\times {\Omega},\quad  (\rho^{\varepsilon}{\chi}^{\varepsilon}_{\mc{S}})|_{t=0}= \rho^{\varepsilon}_0\mathds{1}_{\mc{S}_0}\mbox{ in }{\Omega}.
\end{equation}
\item The initial data are given by
\begin{equation}\label{varepsilon:initial}
\rho^{\varepsilon}(0,x)=\rho_0^{\varepsilon}(x), \quad \rho^{\varepsilon} u^{\varepsilon}(0,x) = q_0^{\varepsilon}(x)\quad\mbox{in }\Omega,\quad   \frac {\partial \rho_0^{\varepsilon }}{\partial \nu}\big |_{\partial \Omega} =0.
\end{equation}
\end{itemize}
Above we have used the following quantities:
\begin{itemize}
\item The density of the volume force is defined as
\begin{equation}\label{gepsilon}
g^{\varepsilon}=(1-\chi^{\varepsilon}_{\mc{S}})g_{\mc{F}} + \chi^{\varepsilon}_{\mc{S}}g_{\mc{S}}.
\end{equation}
\item The artificial pressure is given by 
\begin{equation}\label{p1}
p^{\varepsilon}(\rho)= a^{\varepsilon}\rho^{\gamma} + {\delta} \rho^{\beta},\quad\mbox{ with }\quad a^{\varepsilon} = a_{\mc{F}} (1-\chi^{\varepsilon}_{\mc{S}}),
\end{equation}
where $a_{\mc{F}},{\delta} >0$, and the exponents $\gamma$ and $\beta$ satisfy $\gamma > 3/2, \ \beta \geq \max\{8,\gamma\}$.
\item The viscosity coefficients are given by
\begin{equation}\label{vis1}
 \mu^{\varepsilon} = (1-\chi^{\varepsilon}_{\mc{S}})\mu_{\mc{F}} + {\delta}^2\chi^{\varepsilon}_{\mc{S}},\quad  \lambda^{\varepsilon} = (1-\chi^{\varepsilon}_{\mc{S}})\lambda_{\mc{F}} + {\delta}^2\chi^{\varepsilon}_{\mc{S}}\quad\mbox{ so that }\quad\mu^{\varepsilon} >0,\ 2\mu^{\varepsilon}+3\lambda^{\varepsilon} \geq 0.
\end{equation}
\item $P^{\varepsilon}_{\mc{S}}:L^2(\Omega)\rightarrow L^2(\mc{S}^{\varepsilon}(t))$ is the orthogonal projection to rigid fields; it is defined as in \eqref{approx:projection} with $\chi^{\delta}_{\mc{S}}$ is replaced by $\chi^{\varepsilon}_{\mc{S}}$.
\end{itemize} 
\begin{remark}
Above, the triplet $\left(\mc{S}^{\varepsilon}, \rho^{\varepsilon}, u^{\varepsilon}\right)$ should actually be denoted by $\left(\mc{S}^{\delta,\varepsilon}, \rho^{\delta,\varepsilon}, u^{\delta,\varepsilon}\right)$. The dependence on $\delta$ is due to the penalization term $\Big(\tfrac{1}{\delta}\int\limits_0^T\int\limits_{\Omega} \chi^{\varepsilon}_{\mc{S}}(u^{\varepsilon}-P^{\varepsilon}_{\mc{S}}u^{\varepsilon})\cdot (\phi-P^{\varepsilon}_{\mc{S}}\phi)\Big)$ in \eqref{varepsilon:approx3} and in the viscosity coefficients $\mu^{\varepsilon}$, $\lambda^{\varepsilon}$  in \eqref{vis1}. To simplify the notation, we omit $\delta$ here. 
\end{remark}
  In \cref{14:14} we will prove the following existence result of the approximate system \eqref{varepsilon:approx1}--\eqref{varepsilon:initial}:
\begin{proposition}\label{thm:approxn}
Let $\Omega$ and $\mc{S}_0 \Subset \Omega$ be two regular bounded domains of $\mathbb{R}^3$. Assume that for some $\sigma >0$,
  \begin{equation*}
  \operatorname{dist}(\mc{S}_0,\partial\Omega) > 2\sigma.
  \end{equation*}
  Let $g^{\varepsilon}=(1-\chi^{\varepsilon}_{\mc{S}})g_{\mc{F}} + \chi^{\varepsilon}_{\mc{S}}g_{\mc{S}} \in L^{\infty}((0,T)\times \Omega)$ and $\beta$, $\delta$, $\gamma$ be given as in \eqref{cc:dbg}. Further, let the pressure $p^{\varepsilon}$ be determined by \eqref{p1} and the viscosity coefficients $\mu^{\varepsilon}$, $\lambda^{\varepsilon}$ be given by \eqref{vis1}. The initial conditions satisfy, for some $\underline{\rho}$, $\overline{\rho}>0$, 
\begin{equation}\label{inteps}
0<\underline{\rho}\leq \rho_0^{\varepsilon} \leq \overline{\rho}\ \mbox{  in  }\ \Omega, \quad \rho_0^{\varepsilon} \in W^{1,\infty}(\Omega),\quad q_0^{\varepsilon}\in L^2(\Omega).
\end{equation}
 Let the initial energy $$E^{\varepsilon}[\rho _0 ^{\varepsilon},q_0^{\varepsilon}] =\int\limits_{\Omega}\Bigg(\frac{1}{2} \frac{|q_0^{\varepsilon}|^2}{\rho^{\varepsilon}_0}\mathds{1}_{\{\rho^{\varepsilon}_0>0\}} + \frac{a^{\varepsilon}(0)}{\gamma-1}(\rho_0^{\varepsilon})^{\gamma} + \frac{\delta}{\beta-1}(\rho_0^{\varepsilon})^{\beta} \Bigg):= E_0^{\varepsilon}$$ be uniformly bounded with respect to $\delta$ and $ \varepsilon$.
   Then there exists $T > 0$ (depending only on $E^{\varepsilon}_0$, $g_{\mc{F}}$, $g_{\mc{S}}$,  $\operatorname{dist}(\mc{S}_0,\partial\Omega)$) such that system \eqref{varepsilon:approx1}--\eqref{varepsilon:initial} admits a weak solution $(S^{\varepsilon},\rho^{\varepsilon},u^{\varepsilon})$, which satisfies the following energy inequality:
\begin{multline}\label{energy-varepsilon}
E^{\varepsilon}[\rho ^{\varepsilon},q^{\varepsilon}]+ \int\limits_0^T\int\limits_{\Omega} \Big(2\mu^{\varepsilon}|\mathbb{D}(u^{\varepsilon})|^2 + \lambda^{\varepsilon}|\operatorname{div}u^{\varepsilon}|^2\Big) + \delta\varepsilon \beta\int\limits_0^T\int\limits_{\Omega} (\rho^{\varepsilon})^{\beta-2}|\nabla \rho^{\varepsilon}|^2 
 + \alpha \int\limits_0^T\int\limits_{\partial \Omega} |u^{\varepsilon} \times \nu|^2 \\
 + \alpha \int\limits_0^T\int\limits_{\partial \mc{S}^{\varepsilon}(t)} |(u^{\varepsilon}-P^{\varepsilon}_{\mc{S}}u^{\varepsilon})\times \nu|^2 
  + \frac{1}{\delta}\int\limits_0^T\int\limits_{\Omega} \chi^{\varepsilon}_{\mc{S}}|u^{\varepsilon}-P^{\delta}_{\mc{S}}u^{\varepsilon}|^2 \leq \int\limits_0^T\int\limits_{\Omega}\rho^{\varepsilon} 
  g^{\varepsilon} \cdot u^{\varepsilon}
  +   E_0^{\varepsilon}.
\end{multline}
Moreover,
\begin{equation*}
  \operatorname{dist}(\mc{S}^{\varepsilon}(t),\partial\Omega) \geq {2\sigma},\quad \forall \ t\in [0,T],
  \end{equation*}
and the solution satisfies
\begin{equation*}
\partial_t\rho^{\varepsilon},\  \Delta \rho^{\varepsilon}\in {L^{\frac{5\beta-3}{4\beta}}((0,T)\times\Omega)},
\end{equation*}
\begin{equation}\label{est:indofepsilon}
\sqrt{\varepsilon} \|\nabla \rho^{\varepsilon}\|_{L^2((0,T)\times\Omega)} + \|\rho^{\varepsilon}\|_{L^{\beta+1}((0,T)\times\Omega)} + \|(a^{\varepsilon})^{\frac{1}{\gamma+1}}\rho^{\varepsilon}\|_{L^{\gamma+1}((0,T)\times\Omega)} \leq c,
\end{equation}
where $c$ is a positive constant depending on $\delta$ but which is independent of $\varepsilon$.
\end{proposition}
To solve the problem \eqref{varepsilon:approx1}--\eqref{varepsilon:initial}, we need yet another level of approximation. The $N$-level approximation is obtained via a Faedo-Galerkin approximation scheme. 

Suppose that $\{e_k\}_{k\geq 1} \subset  \mc{D}(\overline{\Omega})$ with $e_k\cdot\nu=0$ on $\partial\Omega$ is a basis of $ L^2(\Omega)$. We set 
\begin{equation*}
X_N = \mbox{ span}(e_1,\ldots,e_N).
\end{equation*}
$X_N$ is a finite dimensional space with scalar product induced by the scalar product in $L^2(\Omega)$. As $X_N$ is finite dimensional, norms on $X_N$ induced by $W^{k,p}$ norms, $k\in \mathbb{N},\ 1\leq p\leq \infty$ are equivalent. We also assume that
\begin{equation*}
\bigcup_{N}X_N \mbox{ is dense in }\left\{v\in W^{1,p}(\Omega) \mid v\cdot \nu=0\mbox{ on }\partial\Omega\right\},\mbox{ for any }1\leq p < \infty.
\end{equation*}
Such a choice of $X_N$ has been constructed in \cite[Theorem 11.19, page 460]{MR3729430}.

The task is to find a triplet $(S^N, \rho^N, u^N)$ satisfying: 
\begin{itemize}
\item  $\mc{S}^N(t) \Subset \Omega$ is a bounded, regular domain for all $t \in [0,T]$  with 
\begin{equation}\label{galerkin-approx1}
\chi^N_{\mc{S}}(t,x)= \mathds{1}_{\mc{S}^N(t)}(x) \in L^{\infty}((0,T)\times \Omega) \cap C([0,T];L^p(\Omega)), \, \forall \, 1 \leq p < \infty.
\end{equation}
\item The velocity field $u^N (t,\cdot)=\sum\limits_{k=1}^N \alpha_k(t)e_k$ with $(\alpha_1,\alpha_2,\ldots,\alpha_N)\in C([0,T])^N$ and the density function $\rho^{N} \in L^2(0,T;H^{2}(\Omega)) \cap H^{1}(0,T;L^{2}(\Omega))$, $\rho^{N} > 0$ satisfies
\begin{equation}\label{galerkin-approx2} 
\frac{\partial {\rho}^N}{\partial t} + \operatorname{div}({\rho}^N u^N) =\varepsilon \Delta\rho^N \mbox{ in }\, (0,T)\times \Omega, \quad  \frac{\partial \rho^N}{\partial \nu}=0 \mbox{ on }\, \partial\Omega.
\end{equation} 
\item For all $\phi \in \mc{D}([0,T); X_N)$ with $\phi\cdot \nu=0$ on $\partial\Omega$, the following holds:
\begin{multline}\label{galerkin-approx3}
- \int\limits_0^T\int\limits_{\Omega} \rho^N \left(u^N\cdot \frac{\partial}{\partial t}\phi +  u^N \otimes u^N : \nabla \phi\right) + \int\limits_0^T\int\limits_{\Omega} \Big(2\mu^N\mathbb{D}(u^N):\mathbb{D}(\phi) + \lambda^N\operatorname{div}u^N\mathbb{I} : \mathbb{D}(\phi) - p^{N}(\rho^N)\mathbb{I}:\mathbb{D}(\phi)\Big) \\
\int\limits_0^T\int\limits_{\Omega} \varepsilon \nabla u^N \nabla \rho^N \cdot \phi
 + \alpha \int\limits_0^T\int\limits_{\partial \Omega} (u^N \times \nu)\cdot (\phi \times \nu) + \alpha \int\limits_0^T\int\limits_{\partial \mc{S}^N(t)} [(u^N-P^N_{\mc{S}}u^N)\times \nu]\cdot [(\phi-P^N_{\mc{S}}\phi)\times \nu] \\
  + \frac{1}{\delta}\int\limits_0^T\int\limits_{\Omega} \chi^N_{\mc{S}}(u^N-P^N_{\mc{S}}u^N)\cdot (\phi-P^N_{\mc{S}}\phi) = \int\limits_0^T\int\limits_{\Omega} \rho^N g^N \cdot \phi
 + \int\limits_{\Omega} (\rho^N u^N \cdot \phi)(0).
\end{multline}
\item ${\chi}^N_{\mc{S}}(t,x)$ satisfies (in the weak sense)
\begin{equation}\label{galerkin-approx4}
\frac{\partial {\chi}^N_{\mc{S}}}{\partial t} + P^N_{\mc{S}}u^N \cdot \nabla \chi^N_{\mc{S}} =0\mbox{ in }(0,T)\times {\Omega},\quad  \chi^N_{\mc{S}}|_{t=0}= \mathds{1}_{\mc{S}_0}\mbox{ in } {\Omega}.
\end{equation}
\item $\rho^{N}{\chi}^{N}_{\mc{S}}(t,x)$ satisfies (in the weak sense)
\begin{equation}\label{N:approx5}
\frac{\partial }{\partial t}(\rho^{N}{\chi}^{N}_{\mc{S}}) + P^{N}_{\mc{S}}u^{N} \cdot \nabla (\rho^{N}{\chi}^{N}_{\mc{S}})=0\mbox{ in }(0,T)\times {\Omega},\quad  (\rho^{N}{\chi}^{N}_{\mc{S}})|_{t=0}= \rho_{0}^N\mathds{1}_{\mc{S}_0}\mbox{ in } {\Omega}.
\end{equation}
\item The initial data are given by
\begin{equation}\label{galerkin-initial}
\rho^N(0)=\rho_0^N, \quad  u^N(0) = u_0^N \quad\mbox{ in }\Omega,\quad   \frac {\partial \rho_0^{N }}{\partial \nu}\Big |_{\partial \Omega} =0.
\end{equation}
\end{itemize}
Above we have used the following quantities:
\begin{itemize}
\item The density of the volume force is defined as \begin{equation}\label{gN}
g^N=(1-\chi^N_{\mc{S}})g_{\mc{F}} + \chi^N_{\mc{S}}g_{\mc{S}}.
\end{equation}
\item The artificial pressure is given by 
\begin{equation}\label{p2}
p^{N}(\rho)= a^{N}\rho^{\gamma} + {\delta} \rho^{\beta},\quad\mbox{ with }\quad a^{N} = a_{\mc{F}} (1-\chi^{N}_{\mc{S}}),
\end{equation}
where $a_{\mc{F}},{\delta} >0$ and the exponents $\gamma$ and $\beta$ satisfy $\gamma > 3/2, \ \beta \geq \max\{8,\gamma\}$.
\item The viscosity coefficients are given by
\begin{equation}\label{vis2}
 \mu^N = (1-\chi^N_{\mc{S}})\mu_{\mc{F}} + {\delta}^2\chi^N_{\mc{S}},\quad  \lambda^N = (1-\chi^N_{\mc{S}})\lambda_{\mc{F}} + {\delta}^2\chi^N_{\mc{S}}\quad\mbox{ so that }\quad\mu^N >0,\ 2\mu^N+3\lambda^N \geq 0.
 \end{equation}
 \item $P^{N}_{\mc{S}}:L^2(\Omega)\rightarrow L^2(\mc{S}^{N}(t))$ is the orthogonal projection to rigid fields; it is defined as in \eqref{approx:projection} with $\chi^{\delta}_{\mc{S}}$  replaced by $\chi^{N}_{\mc{S}}$.
\end{itemize}
\begin{remark}
Actually the triplet $\left(\mc{S}^N, \rho^N, u^N\right)$ above should be denoted by $\left(\mc{S}^{\delta,\varepsilon,N}, \rho^{\delta,\varepsilon,N}, u^{\delta,\varepsilon,N}\right)$. The dependence on $\delta$ and $\varepsilon$ is due to the penalization term $\left(\tfrac{1}{\delta}\int\limits_0^T\int\limits_{\Omega} \chi^N_{\mc{S}}(u^N-P^N_{\mc{S}}u^N)\cdot (\phi-P^N_{\mc{S}}\phi)\right)$, the viscosity coefficients $\mu^{N}$, $\lambda^{N}$ and the artificial dissipative term $(\varepsilon\Delta\rho)$. To simplify the notation, we omit $\delta$ and $\varepsilon$ here. 
\end{remark}

A weak solution $(S^{\varepsilon},\rho^{\varepsilon},u^{\varepsilon})$ to the system \eqref{varepsilon:approx1}--\eqref{varepsilon:initial} is obtained through the limit of $(S^N, \rho^N, u^N)$ as $N\rightarrow \infty$.
The existence result of the approximate solution of the  Faedo-Galerkin scheme reads:
\begin{proposition}\label{fa}
Let $\Omega$ and $\mc{S}_0 \Subset \Omega$ be two regular bounded domains of $\mathbb{R}^3$.  Assume that for some $\sigma>0$,
  \begin{equation*}
  \operatorname{dist}(\mc{S}_0,\partial\Omega) > 2\sigma.
  \end{equation*}
  Let $g^N=(1-\chi^{N}_{\mc{S}})g_{\mc{F}} + \chi^{N}_{\mc{S}}g_{\mc{S}} \in L^{\infty}((0,T)\times \Omega)$ and $\beta$, $\delta$, $\gamma$ be given by \eqref{cc:dbg}. Further, let the pressure $p^{N}$ be determined by \eqref{p2} and the viscosity coefficients $\mu^N$, $\lambda^N$ be given by \eqref{vis2}.
The initial conditions are assumed to satisfy
\begin{equation}\label{initialcond}
0<\underline{\rho}\leq \rho_0^N \leq \overline{\rho}\ \mbox{  in  }\ \Omega, \quad \rho_0^N \in W^{1,\infty}(\Omega),\quad u_0^N\in X_N. 
\end{equation}
 Let the initial energy $$ 
 E^N(\rho_0^N,q_0^N) =\int\limits_{\Omega} \left( \frac{1}{\rho_0^N}|q_0^N|^2\mathds{1}_{\{\rho_0>0\}}  + \frac{a^N(0)}{\gamma-1}(\rho_0^N)^{\gamma} + \frac{\delta}{\beta-1}(\rho_0^N)^{\beta} \right):= E_0^N$$ 
be uniformly bounded with respect to $N,\varepsilon,\delta$. 
Then there exists $T>0$ (depending only on $E^{N}_0$, $g_{\mc{F}}$, $g_{\mc{S}}$, $\overline{\rho}$, $\underline{\rho}$, $\operatorname{dist}(\mc{S}_0,\partial\Omega)$) such that the problem \eqref{galerkin-approx1}--\eqref{galerkin-initial} admits a solution $(\mc{S}^N,\rho^N,u^N)$ and it satisfies the energy inequality:
\begin{multline*}
E^N[\rho ^N, q^N]  + \int\limits_0^T\int\limits_{\Omega} \Big(2\mu^N|\mathbb{D}(u^N)|^2 + \lambda^N |\operatorname{div}u^N|^2\Big)  + \delta\varepsilon \beta\int\limits_0^T\int\limits_{\Omega} (\rho^N)^{\beta-2}|\nabla \rho^N|^2
 + \alpha \int\limits_0^T\int\limits_{\partial \Omega} |u^N \times \nu|^2 \\
 + \alpha \int\limits_0^T\int\limits_{\partial \mc{S}^N(t)} |(u^N-P^N_{\mc{S}}u^N)\times \nu|^2 
  + \frac{1}{\delta}\int\limits_0^T\int\limits_{\Omega} \chi^N_{\mc{S}}|u^N-P^N_{\mc{S}}u^N|^2 \leq \int\limits_0^T\int\limits_{\Omega}\rho^N g^N \cdot u^N 
 + E^N_0.
 \end{multline*}
 Moreover, 
\begin{equation*}
  \operatorname{dist}(\mc{S}^{N}(t),\partial\Omega) \geq {2\sigma},\quad \forall \ t\in [0,T].
  \end{equation*}
\end{proposition}  

We prove the above proposition in \cref{sec:Galerkin}.

  \section{Isometric propagators and the motion of the body} \label{S5}
In this section, we state and prove some results regarding the transport equation that we use in our analysis. We mainly concentrate on the following equation:
\begin{equation}\label{transport1}
\frac{\partial {\chi}_{\mc{S}}}{\partial t} + \operatorname{div}(P_{\mc{S}}u\chi_{\mc{S}}) =0 \mbox{  in  }(0,T)\times\mathbb{R}^3,\quad {\chi}_{\mc{S}}|_{t=0}=\mathds{1}_{\mc{S}_0}\mbox{  in  }\mathbb{R}^3,
\end{equation}
where $P_{\mc{S}}u \in \mc{R}$. Note that here $\mc{R}$ is referring to the set of rigid fields on $\mathbb{R}^3$ in the spirit of \eqref{defR}. It is given by
\begin{equation}\label{projection:P}
P_{\mc{S}}u (t,x)= \frac{1}{m} \int\limits_{\Omega} \rho\chi_{\mc{S}} u + \left(J^{-1} \int\limits_{\Omega}\rho\chi_{\mc{S}}((y-h(t)) \times u)\ dy \right)\times (x-h(t)),\quad \forall \ (t,x)\in (0,T)\times\mathbb{R}^3.
\end{equation}
   In \cite[Proposition 3.1]{MR3272367}, the existence of a solution to \eqref{transport1} and the characterization of the transport of the rigid body have been established with constant $\rho$ in the expression \eqref{projection:P} of $P_{\mc{S}}u$. Here we deal with the case when $\rho$ is evolving.  We start with some existence results when the velocity field and the density satisfy certain regularity assumptions. 
\begin{proposition}\label{reg:chiS}
Let $u \in C([0,T];\mc{D}(\overline{\Omega}))$ and $\rho \in L^2(0,T;H^2(\Omega)) \cap C([0,T];H^1(\Omega))$. Then the following holds true:
\begin{enumerate}
\item There is a unique solution $\chi_{\mc{S}} \in L^{\infty}((0,T)\times \mathbb{R}^3) \cap C([0,T];L^p(\mathbb{R}^3))$  $\forall \ 1\leq p<\infty$ to \eqref{transport1}. More precisely, $$\chi_{\mc{S}}(t,x)= \mathds{1}_{\mc{S}(t)}(x),\quad\forall \ t\geq 0,\ \forall \ x\in \mathbb{R}^3.$$ Moreover, $\mc{S}(t)=\eta_{t,0}(\mc{S}_0)$ for the isometric propagator $\eta_{t,s}$ associated to $P_{\mc{S}}u$:
\begin{equation*}
(t,s)\mapsto \eta_{t,s} \in C^1([0,T]^2; C^{\infty}_{loc}(\mb{R}^3)),
\end{equation*}
where 
 $\eta_{t,s}$ is defined by
\begin{equation}\label{ODE-propagator}
\frac{\partial \eta_{t,s}}{\partial t}(y)=P_{\mc{S}}u (t,\eta_{t,s}(y)),\quad \forall\ (t,s,y)\in (0,T)^2 \times \mb{R}^3, \quad \eta_{s,s}(y)=y,\quad \forall\ y\in  \mb{R}^3.
\end{equation}
\item
 Let $\rho_{0}\mathds{1}_{\mc{S}_0} \in L^{\infty}(\mathbb{R}^3)$. Then there is a unique solution $\rho\chi_{\mc{S}} \in L^{\infty}((0,T)\times \mathbb{R}^3) \cap C([0,T];L^p(\mathbb{R}^3))$,  $\forall \ 1\leq p<\infty$ to the following equation:
\begin{equation}\label{re:transport1}
\frac{\partial (\rho {\chi}_{\mc{S}})}{\partial t} + \operatorname{div}((\rho\chi_{\mc{S}})P_{\mc{S}}u) =0 \mbox{  in  }(0,T)\times\mathbb{R}^3,\quad \rho{\chi}_{\mc{S}}|_{t=0}=\rho_{0}\mathds{1}_{\mc{S}_0}\mbox{  in  }\mathbb{R}^3.
\end{equation}
\end{enumerate}
\end{proposition} 
\begin{proof}
Following \cite[Proposition 3.1]{MR3272367}, we observe that proving existence of solution to \eqref{transport1} is equivalent to establishing the well-posedness of the ordinary differential equation
\begin{equation}\label{ODE-cauchy}
\frac{d}{dt}\eta_{t,0}=U_{\mc{S}}(t,\eta_{t,0}),\quad \eta_{0,0}=\mathbb{I},
\end{equation}
where $U_{\mc{S}}\in \mc{R}$ is given by
\begin{equation*}
U_{\mc{S}}(t,\eta_{t,0})= \frac{1}{m} \int\limits_{\mc{S}_0} \rho(t,\eta_{t,0}(y))\mathds{1}_{\Omega} u(t,\eta_{t,0}(y)) + \left(J^{-1} \int\limits_{\mc{S}_0}\rho(t,\eta_{t,0}(y))\mathds{1}_{\Omega}((\eta_{t,0}(y)-h(t)) \times u(t,\eta_{t,0}(y)))\ dy \right)\times (x-h(t)).
\end{equation*}
  According to the Cauchy-Lipschitz theorem, equation \eqref{ODE-cauchy} admits the unique $C^1$ solution if  $U_{\mc{S}}$ is continuous in $(t,\eta)$ and uniformly Lipschitz in $\eta$. Thus, it is enough to establish the following result analogous to \cite[Lemma 3.2]{MR3272367}: Let $u \in C([0,T];\mc{D}(\overline{\Omega}))$ and $\rho \in L^2(0,T;H^2(\Omega)) \cap C([0,T];H^1(\Omega))$. Then the function 
  \begin{equation*}
  \mc{M}:[0,T]\times \operatorname{Isom}(\mathbb{R}^3)\mapsto \mathbb{R},\quad \mc{M}(t,\eta)=\int\limits_{\mc{S}_0} \rho(t,\eta(y))\mathds{1}_{\Omega}(\eta(y)) u(t,\eta(y))
  \end{equation*}
  is continuous in $(t,\eta)$ and uniformly Lipschitz in $\eta$ over $[0,T]$.
  Observe that the continuity in the $t$-variable is obvious. Moreover, for two isometries $\eta_1$ and $\eta_2$, we have
  \begin{align*}
  \mc{M}(t,\eta_1)-&\mc{M}(t,\eta_2)\\=& \int\limits_{\mc{S}_0} \rho(t,\eta_1(y))\mathds{1}_{\Omega}(\eta_1(y)) (u(t,\eta_1(y))-u(t,\eta_2(y))) + \int\limits_{\mc{S}_0} \rho(t,\eta_1(y))(\mathds{1}_{\Omega}(\eta_1(y))-\mathds{1}_{\Omega}(\eta_2(y))) u(t,\eta_2(y)) \\&+ \int\limits_{\mc{S}_0} (\rho(t,\eta_1(y))-\rho(t,\eta_2(y)))\mathds{1}_{\Omega}(\eta_2(y)) u(t,\eta_2(y)):=M_1 +M_2 + M_3.
  \end{align*}
  As $\rho \in L^2(0,T;H^2(\Omega)) \cap C([0,T];H^1(\Omega))$, the estimates of the terms $M_1$ and $M_2$ are similar to \cite[Lemma 3.2]{MR3272367}. The term $M_3$ can be estimated in the following way:
  \begin{equation*}
  |M_3|\leq C\|\rho\|_{L^{\infty}(0,T;H^1(\Omega))}\|u\|_{L^{\infty}(0,T;L^2(\Omega))}\|\eta_1 -\eta_2\|_{\infty}.
  \end{equation*}
  This finishes the proof of the first part of \cref{reg:chiS}. The second part of this Proposition is similar and we skip it here.
\end{proof}
Next we prove the analogous result of \cite[Proposition 3.3, Proposition 3.4]{MR3272367} on strong and weak sequential continuity which are essential to establish the existence result of the Galerkin approximation scheme in \cref{sec:Galerkin}. The result obtained in the next proposition is used to establish the continuity of the fixed point map in the proof of the existence of Galerkin approximation. 
\begin{proposition}\label{sequential1}
Let $\rho^N_0 \in W^{1,\infty}(\Omega)$, let $\rho^k \in L^2(0,T;H^{2}(\Omega)) \cap C([0,T]; H^{1}(\Omega)) \cap H^{1}(0,T;L^{2}(\Omega))$ be the solution to 
\begin{equation}\label{eq:rhoN} 
\frac{\partial {\rho^k}}{\partial t} + \operatorname{div}({\rho}^k u^k) = \Delta\rho^k \mbox{ in }\, (0,T)\times \Omega, \quad  \frac{\partial \rho^k}{\partial \nu}=0 \mbox{ on }\, \partial\Omega, \quad\rho^k(0,x)=\rho_0^N(x)\quad\mbox{in }\Omega,\quad\frac {\partial \rho_0^{k}}{\partial \nu}\big |_{\partial \Omega} =0.
\end{equation} 
\begin{equation*}
u^k \rightarrow u \mbox{ strongly in }C([0,T];\mc{D}(\overline{\Omega})),\quad \chi_{\mc{S}}^k \mbox{ is bounded in }L^{\infty}((0,T)\times \mb{R}^3)\mbox{ satisfying }
\end{equation*}
\begin{equation}\label{n:transport}
\frac{\partial {\chi}^k_{\mc{S}}}{\partial t} + \operatorname{div}(P^k_{\mc{S}}u^k\chi^k_{\mc{S}}) =0 \mbox{  in  }(0,T)\times\mathbb{R}^3,\quad {\chi}^k_{\mc{S}}|_{t=0}=\mathds{1}_{\mc{S}_0}\mbox{  in  }\mathbb{R}^3,
\end{equation}
and let $\{\rho^{k}\chi_{\mc{S}}^{k}\}$ be a bounded sequence in $L^{\infty}((0,T)\times\mathbb{R}^3)$ satisfying
\begin{equation}\label{N:rhotrans}
\frac{\partial}{\partial t}(\rho^{k}{\chi}^{k}_{\mc{S}}) + \operatorname{div}(P^{k}_{\mc{S}}u^{k}(\rho^{k}\chi^{k}_{\mc{S}})) =0 \mbox{  in  }(0,T)\times\mathbb{R}^3,\quad \rho^{k}{\chi}^{k}_{\mc{S}}|_{t=0}=\rho^N_0\mathds{1}_{\mc{S}_0}\mbox{  in  }\mathbb{R}^3,
\end{equation}
 where $P^{k}_{\mc{S}}:L^2(\Omega)\rightarrow L^2(\mc{S}^{k}(t))$ is the orthogonal projection to rigid fields with $\mc{S}^{k}(t) \Subset \Omega$ being a bounded, regular domain for all $t \in [0,T]$.
Then 
\begin{align*}
& \chi_{\mc{S}}^k \rightarrow  \chi_{\mc{S}} \mbox{ weakly-}* \mbox{ in }L^{\infty}((0,T)\times \mb{R}^3) \mbox{ and}\mbox{ strongly} \mbox{ in }C([0,T]; L^p_{loc}(\mb{R}^3)), \ \forall \ 1\leq p<\infty,\\
&\rho^k\chi_{\mc{S}}^k \rightarrow \rho \chi_{\mc{S}} \mbox{ weakly-}* \mbox{ in }L^{\infty}((0,T)\times \mb{R}^3) \mbox{ and}\mbox{ strongly} \mbox{ in }C([0,T]; L^p_{loc}(\mb{R}^3)), \ \forall \ 1\leq p<\infty,
\end{align*}
where $\chi_{\mc{S}}$ and $\rho\chi_{\mc{S}}$ are satisfying \eqref{transport1} and \eqref{re:transport1} with initial data $\mathds{1}_{\mc{S}_0}$ and $\rho^{N}_{0}\mathds{1}_{\mc{S}_0}$, respectively. Moreover,
\begin{align*}
& P_{\mc{S}}^k u^k \rightarrow P_{\mc{S}} u \mbox{ strongly} \mbox{ in }C([0,T]; C^{\infty}_{loc}(\mb{R}^3)), \\
&\eta_{t,s}^k \rightarrow \eta_{t,s} \mbox{ strongly} \mbox{ in }C^{1}([0,T]^2; C^{\infty}_{loc}(\mb{R}^3)).
\end{align*}
\end{proposition}
 \begin{proof}
 As $\{u^{k}\}$ converges strongly in $C([0,T];\mc{D}(\overline{\Omega}))$
  and $\{ \rho^{k} \chi_{\mc{S}}^{k}\} \mbox{ is bounded in }L^{\infty}((0,T)\times \mb{R}^3)$,
 we obtain that $P^{k}_{\mc{S}}u^{k}$ is bounded in $L^2(0,T;\mc{R})$. Thus, up to a subsequence,
 \begin{equation}\label{PN:weak}
 P^{k}_{\mc{S}}u^{k} \rightarrow \overline{u_{\mc{S}}} \mbox{ weakly in }L^2(0,T;\mc{R}).
 \end{equation}
 Here, obviously $P^{k}_{\mc{S}}u^{k} \in L^1(0,T; L^{\infty}_{loc}(\mb{R}^3))$, $\operatorname{div}(P^{k}_{\mc{S}}u^{k})=0$ and $\overline{u_{\mc{S}}} \in L^{1}(0,T;W^{1,1}_{loc}(\mb{R}^3))$ satisfies 
\begin{equation*}
\frac{\overline{u_{\mc{S}}}}{1+|x|} \in L^1(0,T;L^1(\mb{R}^3)).
\end{equation*}
Moreover, $\{\chi_{\mc{S}}^{k}\} \mbox{ is bounded in }L^{\infty}((0,T)\times \mb{R}^3)$, $\chi_{\mc{S}}^{k}$ satisfies \eqref{n:transport} and $\{\rho^{N}\chi_{\mc{S}}^{k}\}$ is bounded in $L^{\infty}((0,T)\times\mathbb{R}^3)$, $\rho^{k}\chi_{\mc{S}}^{k}$ satisfies \eqref {N:rhotrans}. As we have verified all the required conditions, we can apply \cite[Theorem II.4, Page 521]{DiPerna1989} to  obtain 
$$\chi_{\mc{S}}^{k} \mbox{ converges weakly-}*\mbox{ in  }L^{\infty}((0,T)\times \mb{R}^3),\mbox{ strongly in }C([0,T]; L^p_{loc}(\mb{R}^3)) \ (1\leq p<\infty),$$
$$\rho^{k}\chi_{\mc{S}}^{k} \mbox{ converges weakly-}*\mbox{ in  }L^{\infty}((0,T)\times \mb{R}^3),\mbox{ strongly in }C([0,T]; L^p_{loc}(\mb{R}^3)) \ (1\leq p<\infty).$$
Let the limit of $\chi_{\mc{S}}^{k}$ be denoted by ${\chi_{\mc{S}}}$, it satisfies  
\begin{equation*}
\frac{\partial {\chi_{\mc{S}}}}{\partial t} + \operatorname{div}(\overline{u_{\mc{S}}}\ {\chi_{\mc{S}}}) =0 \mbox{  in  }(0,T)\times \mathbb{R}^3,\quad {\chi_{\mc{S}}}|_{t=0}=\mathds{1}_{\mc{S}_0}\mbox{  in  }\mathbb{R}^3.
\end{equation*} 
Let the weak limit of $\rho^{k}\chi_{\mc{S}}^{k}$ be denoted by $\overline{\rho\chi_{\mc{S}}}$; it satisfies  
\begin{equation*}
\frac{\partial (\overline{\rho\chi_{\mc{S}}})}{\partial t} + \operatorname{div}(\overline{u_{\mc{S}}}\ \overline{\rho\chi_{\mc{S}}}) =0 \mbox{  in  }(0,T)\times\mathbb{R}^3,\quad \overline{\rho\chi_{\mc{S}}}|_{t=0}=\rho^N_0\mathds{1}_{\mc{S}_0}\mbox{  in  }\mathbb{R}^3.
\end{equation*}
We follow the similar analysis as for the fluid case explained in \cite[Section 7.8.1, Page 362]{MR2084891} to conclude that
\begin{equation*}
 \rho^k \rightarrow \rho\mbox{ strongly in } L^p((0,T)\times \Omega), \quad\forall \ 1\leq p< \frac{4}{3}\beta\mbox{ with } \ \beta \geq \max\{8,\gamma\},\ \gamma > 3/2. \end{equation*}
The strong convergences of $\rho^{k}$ and $\chi_{\mc{S}}^{k}$ help us to identify the limit:
\begin{equation*}
\overline{\rho\chi_{\mc{S}}}= \rho\chi_{\mc{S}}.
\end{equation*}
Using the convergences of $\rho^{k}\chi_{\mc{S}}^{k}$ and $u^{k}$ in the equation
\begin{equation*}
P_{\mc{S}}^{k}u^{k}(t,x)= \frac{1}{m^{k}} \int\limits_{\Omega} \rho^{k}\chi_{\mc{S}}^{k} u^{k} + \left((J^k)^{-1} \int\limits_{\Omega}\rho^{k}\chi_{\mc{S}}^{k}((y-h^{k}(t)) \times u^{k})\ dy \right)\times (x-h^{k}(t)), 
\end{equation*}
and the convergence in \eqref{PN:weak}, we conclude that 
\begin{equation*}
\overline{u_{\mc{S}}}={P_{\mc{S}}}u.
\end{equation*}
 The convergence of the isometric propagator $\eta_{t,s}^{k}$ follows from the convergence of $P_{\mc{S}}^{k} u^{k}$ and equation \eqref{ODE-propagator}.
 \end{proof}
 We need the next result on weak sequential continuity to analyze the limiting system of Faedo-Galerkin as $N\rightarrow\infty$ in \cref{14:14}. The proof is similar to that of \cref{sequential1} and we skip it here.
\begin{proposition}\label{sequential11}
 Let us assume that $\rho^N_0 \in W^{1,\infty}(\Omega)$ with $\rho^N_0 \rightarrow \rho_0$ in $W^{1,\infty}(\Omega)$,
 $\rho^N$ satisfies \eqref{eq:rhoN} and
\begin{equation*}
\rho^N \rightarrow \rho\mbox{ strongly in  }L^p((0,T)\times \Omega),\ 1\leq p< \frac{4}{3}\beta\mbox{ with } \ \beta \geq \max\{8,\gamma\},\ \gamma > 3/2.
\end{equation*} 
Let $\{u^N,\chi_{\mc{S}}^N\}$ be a bounded sequence in $L^{\infty}(0,T; L^2(\Omega)) \times L^{\infty}((0,T)\times \mb{R}^3)$ satisfying
\eqref{n:transport}. Let $\{\rho^{N}\chi_{\mc{S}}^{N}\}$ be a bounded sequence in $L^{\infty}((0,T)\times\mathbb{R}^3)$ satisfying
\eqref{N:rhotrans}. Then, up to a subsequence, we have
\begin{align*}
& u^N \rightarrow u \mbox{ weakly-}* \mbox{ in }L^{\infty}(0,T; L^{2}(\Omega)),\\
& \chi_{\mc{S}}^N \rightarrow  \chi_{\mc{S}} \mbox{ weakly-}* \mbox{ in }L^{\infty}((0,T)\times \mb{R}^3) \mbox{ and}\mbox{ strongly} \mbox{ in }C([0,T]; L^p_{loc}(\mb{R}^3)), \ \forall\ 1\leq p<\infty,\\
&\rho^N\chi_{\mc{S}}^N \rightarrow  \rho\chi_{\mc{S}} \mbox{ weakly-}* \mbox{ in }L^{\infty}((0,T)\times \mb{R}^3) \mbox{ and}\mbox{ strongly} \mbox{ in }C([0,T]; L^p_{loc}(\mb{R}^3)), \ \forall \ 1\leq p<\infty,
\end{align*}
where $\chi_{\mc{S}}$ and $\rho\chi_{\mc{S}}$ satisfying \eqref{transport1} and \eqref{re:transport1}, respectively. Moreover,
\begin{align*}
& P_{\mc{S}}^N u^N \rightarrow P_{\mc{S}} u \mbox{ weakly-}* \mbox{ in }L^{\infty}(0,T; C^{\infty}_{loc}(\mb{R}^3)),\\
&\eta_{t,s}^N \rightarrow \eta_{t,s} \mbox{ weakly-}* \mbox{ in }W^{1,\infty}((0,T)^2; C^{\infty}_{loc}(\mb{R}^3)).
\end{align*}
 \end{proposition}
 
   At the level of the Galerkin approximation, we have boundedness of $\sqrt{\rho^N}u^N$ in $L^{\infty}(0,T;L^2(\Omega))$ and $\rho^N$ is strictly positive, which means that we get the boundedness of $u^N$ in $L^{\infty}(0,T; L^2(\Omega))$. So, we can use \cref{sequential11} in the convergence analysis of the Galerkin scheme. In the case of the $\varepsilon$-level for the compressible fluid, we have boundedness of $\sqrt{\rho
^{\varepsilon}}u^{\varepsilon}$ in $L^{\infty}(0,T;L^2(\Omega))$ but $\rho^{\varepsilon}$ is only non-negative. On the other hand, we establish boundedness of $u^{\varepsilon}$ in $L^{2}(0,T;H^1(\Omega))$. We need the following result for the convergence analysis of the vanishing viscosity limit in \cref{14:18}.
\begin{proposition}\label{sequential-varepsilon}
Let $\rho^{\varepsilon}_0 \in W^{1,\infty}(\Omega)$ with $\rho^{\varepsilon}_0 \rightarrow \rho_0$ in $L^{\beta}(\Omega)$,
 $\rho^{\varepsilon}$ satisfies 
 \begin{equation*} 
\frac{\partial {\rho^{\varepsilon}}}{\partial t} + \operatorname{div}({\rho}^{\varepsilon} u^{\varepsilon}) = \Delta\rho^{\varepsilon} \mbox{ in }\, (0,T)\times \Omega, \quad  \frac{\partial \rho^{\varepsilon}}{\partial \nu}=0 \mbox{ on }\, \partial\Omega, \quad\rho^{\varepsilon}(0,x)=\rho_0^{\varepsilon}(x)\mbox{ in }\ \Omega,\quad\frac {\partial \rho_0^{\varepsilon}}{\partial \nu}\big |_{\partial \Omega} =0., 
\end{equation*} 
and
\begin{equation}\label{epsilon-rhoweak}
\rho^{\varepsilon}\rightarrow \rho\mbox{ weakly in }L^{\beta+1}((0,T)\times \Omega),\mbox{  with } \ \beta \geq \max\{8,\gamma\},\ \gamma > 3/2.
\end{equation}
Let $\{u^{\varepsilon},\chi_{\mc{S}}^{\varepsilon}\}$ be a bounded sequence in $L^{2}(0,T; H^1(\Omega)) \times L^{\infty}((0,T)\times \mb{R}^3)$ satisfying
 \begin{equation}\label{epsilon:transport}
\frac{\partial {\chi}^{\varepsilon}_{\mc{S}}}{\partial t} + \operatorname{div}(P^{\varepsilon}_{\mc{S}}u^{\varepsilon}\chi^{\varepsilon}_{\mc{S}}) =0 \mbox{  in  }(0,T)\times\mathbb{R}^3,\quad {\chi}^{\varepsilon}_{\mc{S}}|_{t=0}=\mathds{1}_{\mc{S}_0}\mbox{  in  }\mathbb{R}^3,
\end{equation}
and let $\{\rho^{\varepsilon}\chi_{\mc{S}}^{\varepsilon}\}$ be a bounded sequence in $L^{\infty}((0,T)\times\mathbb{R}^3)$ satisfying
\begin{equation}\label{epsilon:rhotrans}
\frac{\partial}{\partial t}(\rho^{\varepsilon}{\chi}^{\varepsilon}_{\mc{S}}) + \operatorname{div}(P^{\varepsilon}_{\mc{S}}u^{\varepsilon}(\rho^{\varepsilon}\chi^{\varepsilon}_{\mc{S}})) =0 \mbox{  in  }(0,T)\times\mathbb{R}^3,\quad \rho^{\varepsilon}{\chi}^{\varepsilon}_{\mc{S}}|_{t=0}=\rho^{\varepsilon}_0\mathds{1}_{\mc{S}_0}\mbox{  in  }\mathbb{R}^3,
\end{equation}
where $P^{\varepsilon}_{\mc{S}}:L^2(\Omega)\rightarrow L^2(\mc{S}^{\varepsilon}(t))$ is the orthogonal projection onto rigid fields with $\mc{S}^{\varepsilon}(t) \Subset \Omega$ being a bounded, regular domain for all $t \in [0,T]$. Then up to a subsequence, we have
\begin{align*}
& u^{\varepsilon} \rightarrow u \mbox{ weakly } \mbox{ in }L^{2}(0,T; H^{1}(\Omega)),\\
& \chi_{\mc{S}}^{\varepsilon} \rightarrow  \chi_{\mc{S}} \mbox{ weakly-}* \mbox{ in }L^{\infty}((0,T)\times \mb{R}^3) \mbox{ and }\mbox{ strongly } \mbox{ in }C([0,T]; L^p_{loc}(\mb{R}^3)) \ (1\leq p<\infty),\\
&\rho^{\varepsilon}\chi_{\mc{S}}^{\varepsilon} \rightarrow  \rho\chi_{\mc{S}} \mbox{ weakly-}* \mbox{ in }L^{\infty}((0,T)\times \mb{R}^3) \mbox{ and }\mbox{ strongly } \mbox{ in }C([0,T]; L^p_{loc}(\mb{R}^3)) \ (1\leq p<\infty),
\end{align*}
with $\chi_{\mc{S}}$ and $\rho\chi_{\mc{S}}$ satisfying \eqref{transport1} and \eqref{re:transport1} respectively. Moreover,
\begin{align*}
& P_{\mc{S}}^{\varepsilon} u^{\varepsilon} \rightarrow P_{\mc{S}} u \mbox{ weakly } \mbox{ in }L^{2}(0,T; C^{\infty}_{loc}(\mb{R}^3)),\\
&\eta_{t,s}^{\varepsilon} \rightarrow \eta_{t,s} \mbox{ weakly } \mbox{ in }H^{1}((0,T)^2; C^{\infty}_{loc}(\mb{R}^3)).
\end{align*}
 \end{proposition}
 \begin{proof}
 As $\{u^{\varepsilon}\}$ is a bounded sequence in $L^2(0,T;H^1(\Omega))$
  and $\{ \rho^{\varepsilon} \chi_{\mc{S}}^{\varepsilon}\} \mbox{ is bounded in }L^{\infty}((0,T)\times \mb{R}^3)$,
 we obtain that $\{P^{\varepsilon}_{\mc{S}}u^{\varepsilon}\}$ is bounded in $L^2(0,T;\mc{R})$. Thus, up to a subsequence,
 \begin{equation}\label{Pepsilon:weak}
 P^{\varepsilon}_{\mc{S}}u^{\varepsilon} \rightarrow \overline{u_{\mc{S}}} \mbox{ weakly in }L^2(0,T;\mc{R}).
 \end{equation}
 Here, obviously $P^{\varepsilon}_{\mc{S}}u^{\varepsilon} \in L^1(0,T; L^{\infty}_{loc}(\mb{R}^3))$, $\operatorname{div}(P^{\varepsilon}_{\mc{S}}u^{\varepsilon})=0$ and $\overline{u_{\mc{S}}} \in L^{1}(0,T;W^{1,1}_{loc}(\mb{R}^3))$ satisfies 
\begin{equation*}
\frac{\overline{u_{\mc{S}}}}{1+|x|} \in L^1(0,T;L^1(\mb{R}^3)).
\end{equation*}
Moreover, $\{\chi_{\mc{S}}^{\varepsilon}\} \mbox{ is bounded in }L^{\infty}((0,T)\times \mb{R}^3)$, $\chi_{\mc{S}}^{\varepsilon}$ satisfies \eqref{epsilon:transport} and $\{\rho^{\varepsilon}\chi_{\mc{S}}^{\varepsilon}\}$ is bounded in $L^{\infty}((0,T)\times\mathbb{R}^3)$, $\rho^{\varepsilon}\chi_{\mc{S}}^{\varepsilon}$ satisfies \eqref {epsilon:rhotrans}. As we have verified all the required conditions, we can apply \cite[Theorem II.4, Page 521]{DiPerna1989} to  obtain 
$$\chi_{\mc{S}}^{\varepsilon} \mbox{ converges weakly-}*\mbox{ in  }L^{\infty}((0,T)\times \mb{R}^3),\mbox{ strongly in }C([0,T]; L^p_{loc}(\mb{R}^3)) \ (1\leq p<\infty),$$
$$\rho^{\varepsilon}\chi_{\mc{S}}^{\varepsilon} \mbox{ converges weakly-}*\mbox{ in  }L^{\infty}((0,T)\times \mb{R}^3),\mbox{ strongly in }C([0,T]; L^p_{loc}(\mb{R}^3)) \ (1\leq p<\infty).$$
Let the limit of $\chi_{\mc{S}}^{\varepsilon}$ be denoted by ${\chi_{\mc{S}}}$; it satisfies  
\begin{equation*}
\frac{\partial {\chi_{\mc{S}}}}{\partial t} + \operatorname{div}(\overline{u_{\mc{S}}}\ {\chi_{\mc{S}}}) =0 \mbox{  in  }\mathbb{R}^3,\quad {\chi_{\mc{S}}}|_{t=0}=\mathds{1}_{\mc{S}_0}\mbox{  in  }\mathbb{R}^3.
\end{equation*} 
Let the limit of $\rho^{\varepsilon}\chi_{\mc{S}}^{\varepsilon}$ be denoted by $\overline{\rho\chi_{\mc{S}}}$; it satisfies  
\begin{equation*}
\frac{\partial (\overline{\rho\chi_{\mc{S}}})}{\partial t} + \operatorname{div}(\overline{u_{\mc{S}}}\ \overline{\rho\chi_{\mc{S}}}) =0 \mbox{  in  }(0,T)\times\mathbb{R}^3,\quad \overline{\rho\chi_{\mc{S}}}|_{t=0}=\rho_0\mathds{1}_{\mc{S}_0}\mbox{  in  }\mathbb{R}^3.
\end{equation*}
The weak convergence of $\rho^{\varepsilon}$ and strong convergence of $\chi_{\mc{S}}^{\varepsilon}$ help us to identify the limit:
\begin{equation*}
\overline{\rho\chi_{\mc{S}}}= \rho\chi_{\mc{S}}.
\end{equation*}
Using the convergences of $\rho^{\varepsilon}\chi_{\mc{S}}^{\varepsilon}$ and  $u^{\varepsilon}$ in the equation
\begin{equation*}
P_{\mc{S}}^{\varepsilon}u^{\varepsilon}(t,x)= \frac{1}{m^{\varepsilon}} \int\limits_{\Omega} \rho^{\varepsilon}\chi_{\mc{S}}^{\varepsilon} u^{\varepsilon} + \left((J^{\varepsilon})^{-1} \int\limits_{\Omega}\rho^{\varepsilon}\chi_{\mc{S}}^{\varepsilon}((y-h^{\varepsilon}(t)) \times u^{\varepsilon})\ dy \right)\times (x-h^{\varepsilon}(t)), 
\end{equation*}
and the convergence in \eqref{Pepsilon:weak}, we conclude that 
\begin{equation*}
\overline{u_{\mc{S}}}={P_{\mc{S}}}u.
\end{equation*}
 The convergence of the isometric propagators $\eta_{t,s}^{\varepsilon}$ follows from the convergence of $P_{\mc{S}}^{\varepsilon} u^{\varepsilon}$ and equation \eqref{ODE-propagator}.
 \end{proof}
 
 In the limit of $u^{\delta}$, we can expect the  boundedness of the limit only in $L^{2}(0,T;L^2(\Omega))$ but not in $L^{2}(0,T;H^1(\Omega))$.
 That is why we need a different sequential continuity result, which we use in \cref{S4}.
\begin{proposition}\label{sequential2}
Let $\rho^{\delta}_0 \in L^{\beta}(\Omega)$ with $\rho^{\delta}_0 \rightarrow \rho_0$ in $L^{\gamma}(\Omega)$, let
 $\rho^{\delta}$ satisfy 
 \begin{equation*} 
\frac{\partial {\rho^{\delta}}}{\partial t} + \operatorname{div}({\rho}^{\delta} u^{\delta}) = 0 \mbox{ in }\, (0,T)\times \Omega, \quad\rho^{\delta}(0,x)=\rho_0^{\delta}(x)\mbox{ in }\ \Omega, \end{equation*} 
and
\begin{equation}\label{delta-rhoweak}
\rho^{\delta}\rightarrow \rho\mbox{ weakly in }L^{\gamma+\theta}((0,T)\times \Omega),\mbox{  with  }\gamma>3/2,\, \theta=\frac{2}{3}\gamma-1.
\end{equation}
 Let $\{u^{\delta},\chi_{\mc{S}}^{\delta}\}$ be a bounded sequence in $L^{2}(0,T; L^2(\Omega)) \times L^{\infty}((0,T)\times \mb{R}^3)$ satisfying
 \begin{equation}\label{delta:transport}
\frac{\partial {\chi}^{\delta}_{\mc{S}}}{\partial t} + \operatorname{div}(P^{\delta}_{\mc{S}}u^{\delta}\chi^{\delta}_{\mc{S}}) =0 \mbox{  in  }(0,T)\times\mathbb{R}^3,\quad {\chi}^{\delta}_{\mc{S}}|_{t=0}=\mathds{1}_{\mc{S}_0}\mbox{  in  }\mathbb{R}^3,
\end{equation}
and let $\{\rho^{\delta}\chi_{\mc{S}}^{\delta}\}$ be a bounded sequence in $L^{\infty}((0,T)\times\mathbb{R}^3)$ satisfying
\begin{equation}\label{delta:rhotrans}
\frac{\partial}{\partial t}(\rho^{\delta}{\chi}^{\delta}_{\mc{S}}) + \operatorname{div}(P^{\delta}_{\mc{S}}u^{\delta}(\rho^{\delta}\chi^{\delta}_{\mc{S}})) =0 \mbox{  in  }(0,T)\times\mathbb{R}^3,\quad \rho^{\delta}{\chi}^{\delta}_{\mc{S}}|_{t=0}=\rho^{\delta}_0\mathds{1}_{\mc{S}_0}\mbox{  in  }\mathbb{R}^3,
\end{equation}
where $P^{\delta}_{\mc{S}}:L^2(\Omega)\rightarrow L^2(\mc{S}^{\delta}(t))$ is the orthogonal projection onto rigid fields with $\mc{S}^{\delta}(t) \Subset \Omega$ being a bounded, regular domain for all $t \in [0,T]$. Then, up to a subsequence, we have
\begin{align*}
& u^{\delta} \rightarrow u \mbox{ weakly } \mbox{ in }L^{2}(0,T; L^{2}(\Omega)),\\
& \chi_{\mc{S}}^{\delta} \rightarrow  \chi_{\mc{S}} \mbox{ weakly-}* \mbox{ in }L^{\infty}((0,T)\times \mb{R}^3) \mbox{ and }\mbox{ strongly } \mbox{ in }C([0,T]; L^p_{loc}(\mb{R}^3)) \ (1\leq p<\infty),\\
& \rho^{\delta}\chi_{\mc{S}}^{\delta} \rightarrow  \rho\chi_{\mc{S}} \mbox{ weakly-}* \mbox{ in }L^{\infty}((0,T)\times \mb{R}^3) \mbox{ and }\mbox{ strongly } \mbox{ in }C([0,T]; L^p_{loc}(\mb{R}^3)) \ (1\leq p<\infty),
\end{align*}
with $\chi_{\mc{S}}$ and $\rho\chi_{\mc{S}}$ satisfying \eqref{transport1} and \eqref{re:transport1} respectively. Moreover,
\begin{align*}
& P_{\mc{S}}^{\delta} u^{\delta} \rightarrow P_{\mc{S}} u \mbox{ weakly } \mbox{ in }L^{2}(0,T; C^{\infty}_{loc}(\mb{R}^3)),\\
& \eta_{t,s}^{\delta} \rightarrow \eta_{t,s} \mbox{ weakly } \mbox{ in }H^{1}((0,T)^2; C^{\infty}_{loc}(\mb{R}^3)).
\end{align*}
 \end{proposition}
 \begin{proof}
 As $\{u^{\delta}\}$ is a bounded sequence in $L^2(0,T;L^2(\Omega))$  and $\{ \rho^{\delta} \chi_{\mc{S}}^{\delta}\} \mbox{ is bounded in }L^{\infty}((0,T)\times \mb{R}^3)$, we obtain that $\{P^{\delta}_{\mc{S}}u^{\delta}\}$ is bounded in $L^2(0,T;\mc{R})$. Thus, up to a subsequence,
 \begin{equation}\label{P:weak}
 P^{\delta}_{\mc{S}}u^{\delta} \rightarrow \overline{u_{\mc{S}}} \mbox{ weakly in }L^2(0,T;\mc{R}).
 \end{equation}
 Here, obviously $P^{\delta}_{\mc{S}}u^{\delta} \in L^1(0,T; L^{\infty}_{loc}(\mb{R}^3))$, $\operatorname{div}(P^{\delta}_{\mc{S}}u^{\delta})=0$ and $\overline{u_{\mc{S}}} \in L^{1}(0,T;W^{1,1}_{loc}(\mb{R}^3))$ satisfies 
\begin{equation*}
\frac{\overline{u_{\mc{S}}}}{1+|x|} \in L^1(0,T;L^1(\mb{R}^3)).
\end{equation*}
Moreover, $\{\chi_{\mc{S}}^{\delta}\} \mbox{ is bounded in }L^{\infty}((0,T)\times \mb{R}^3)$, $\chi_{\mc{S}}^{\delta}$ satisfies \eqref{delta:transport}  and $\{\rho^{\delta}\chi_{\mc{S}}^{\delta}\}$ is bounded in $L^{\infty}((0,T)\times\mathbb{R}^3)$, $\rho^{\delta}\chi_{\mc{S}}^{\delta}$ satisfies \eqref {delta:rhotrans}. Now we can apply \cite[Theorem II.4, Page 521]{DiPerna1989} to  obtain 
$$\chi_{\mc{S}}^{\delta} \mbox{ converges weakly-}*\mbox{ in  }L^{\infty}((0,T)\times \mb{R}^3),\mbox{ and strongly in }C([0,T]; L^p_{loc}(\mb{R}^3)) \ (1\leq p<\infty),$$
$$\rho^{\delta}\chi_{\mc{S}}^{\delta} \mbox{ converges weakly-}* \mbox{ in }L^{\infty}((0,T)\times \mb{R}^3) \mbox{ and }\mbox{ strongly } \mbox{ in }C([0,T]; L^p_{loc}(\mb{R}^3)) \ (1\leq p<\infty).$$
Let the weak limit of $\chi_{\mc{S}}^{\delta}$ be denoted by ${\chi_{\mc{S}}}$. Then it satisfies  
\begin{equation*}
\frac{\partial {\chi_{\mc{S}}}}{\partial t} + \operatorname{div}(\overline{u_{\mc{S}}}\ {\chi_{\mc{S}}}) =0 \mbox{  in  }(0,T)\times\mathbb{R}^3,\quad {\chi_{\mc{S}}}|_{t=0}=\mathds{1}_{\mc{S}_0}\mbox{  in  }\mathbb{R}^3,
\end{equation*}
Let the limit of $\rho^{\delta}\chi_{\mc{S}}^{\delta}$ be denoted by $\overline{\rho\chi_{\mc{S}}}$; it satisfies  
\begin{equation*}
\frac{\partial (\overline{\rho\chi_{\mc{S}}})}{\partial t} + \operatorname{div}(\overline{u_{\mc{S}}}\ \overline{\rho\chi_{\mc{S}}}) =0 \mbox{  in  }(0,T)\times\mathbb{R}^3,\quad \overline{\rho\chi_{\mc{S}}}|_{t=0}=\rho_0\mathds{1}_{\mc{S}_0}\mbox{  in  }\mathbb{R}^3.
\end{equation*}
From \eqref{delta-rhoweak}, we know that\begin{equation*}
\rho^{\delta}\rightarrow \rho\mbox{ weakly in }L^{\gamma+\theta}((0,T)\times \Omega),\mbox{  with  }\gamma>3/2,\, \theta=\frac{2}{3}\gamma-1.
\end{equation*}
The weak convergence of $\rho^{\delta}$ to $\rho$ and strong convergence of $\chi_{\mc{S}}^{\delta}$ to $\chi_{\mc{S}}$ help us to identify the limit:
\begin{equation*}
\overline{\rho\chi_{\mc{S}}}= \rho\chi_{\mc{S}},
\end{equation*} 
Using the convergences of $\rho^{\delta}\chi_{\mc{S}}^{\delta}$ and $u^{\delta}$ in the equation
\begin{equation*}
P_{\mc{S}}^{\delta}u^{\delta}(t,x)= \frac{1}{m^{\delta}} \int\limits_{\Omega} \rho^{\delta}\chi_{\mc{S}}^{\delta} u^{\delta} + \left((J^{\delta})^{-1} \int\limits_{\Omega}\rho^{\delta}\chi_{\mc{S}}^{\delta}((y-h^{\delta}(t)) \times u^{\delta})\ dy \right)\times (x-h^{\delta}(t)), 
\end{equation*}
and the convergence in \eqref{P:weak}, we conclude that 
\begin{equation*}
\overline{u_{\mc{S}}}={P_{\mc{S}}}u.
\end{equation*}
 The convergence of the isometric propagator $\eta_{t,s}^{\delta}$ follows from the convergence of $P_{\mc{S}}^{\delta} u^{\delta}$ and equation \eqref{ODE-propagator}.
 \end{proof}

 \section{Existence proofs of Approximate solutions}\label{S3}
In this section, we present the proofs of the existence results of the three approximation levels. We start with the $N$-level approximation in \cref{sec:Galerkin} and the limit as $N\to\infty$ in \cref{14:14}, which yields existence at the $\varepsilon$-level. The convergence of $\varepsilon\to 0$, considered in \cref{14:18}, then shows existence of solutions at the $\delta$-level. The final limit problem as $\delta\to 0$ is the topic of \cref{S4}.

\subsection{Existence of the Faedo-Galerkin approximation}\label{sec:Galerkin}
In this subsection, we construct a solution $(\mc{S}^N,\rho^N,u^{N})$ to the problem \eqref{galerkin-approx1}--\eqref{galerkin-initial}. 
First we recall a known maximal regularity result for the parabolic problem \eqref{galerkin-approx2}:
\begin{proposition}\cite[Proposition 7.39, Page 345]{MR2084891}\label{parabolic}
Suppose that $\Omega$ is a regular bounded domain  and assume $\rho_0 \in W^{1,\infty}(\Omega)$, $\underline{\rho} \leq \rho_0 \leq \overline{\rho}$, $u \in L^{\infty}(0,T;W^{1,\infty}(\Omega))$. Then the parabolic problem \eqref{galerkin-approx2} admits a unique solution in the solution space
\begin{equation*}
\rho \in L^2(0,T;H^{2}(\Omega)) \cap C([0,T]; H^{1}(\Omega)) \cap H^{1}(0,T;L^{2}(\Omega))
\end{equation*} and it satisfies 
\begin{equation} \label{bounds-on-rho}
\underline{\rho}\exp \left(-\int\limits_0^{\tau} \|\operatorname{div}u(s)\|_{L^{\infty}(\Omega)}\ ds\right)\leq \rho(\tau,x)\leq \overline{\rho}\exp \left(\int\limits_0^{\tau} \|\operatorname{div}u(s)\|_{L^{\infty}(\Omega)}\ ds\right)
\end{equation}
for any $\tau \in [0,T]$.
\end{proposition}
 
\begin{proof}[Proof of \cref{fa}]
The idea is to view our Galerkin approximation as a fixed point problem and then apply Schauder's fixed point theorem to it. We set  
\begin{equation*}
B_{R,T}=\{u\in C([0,T]; X_N),\ \|u\|_{L^{\infty}(0,T;L^2(\Omega))}\leq R\},
\end{equation*}
for $R$ and $T$ positive which will be fixed in Step 3.

\underline{Step 1: Continuity equation and transport of the body.} 
Given $u \in B_{R,T}$, let $\rho$ be the solution to 
\begin{equation}\label{eq:rho} 
\frac{\partial {\rho}}{\partial t} + \operatorname{div}({\rho} u) =\varepsilon \Delta\rho \mbox{ in }\, (0,T)\times \Omega, \quad  \frac{\partial \rho}{\partial \nu}=0 \mbox{ on }\, \partial\Omega, \quad\rho(0)=\rho_0^N,\quad 0<\underline{\rho}\leq \rho_0^N \leq \overline{\rho}, 
\end{equation} 
and let ${\chi}_{\mc{S}}$ satisfy
\begin{equation}\label{eq:chinew} 
\frac{\partial {\chi}_{\mc{S}}}{\partial t} + P_{\mc{S}}u \cdot \nabla \chi_{\mc{S}} =0,\quad  \chi_{\mc{S}}|_{t=0}= \mathds{1}_{\mc{S}_0},
\end{equation}
and
\begin{equation}\label{eq:rhochinew} 
\frac{\partial }{\partial t}(\rho{\chi}_{\mc{S}}) + P_{\mc{S}}u \cdot \nabla (\rho{\chi}_{\mc{S}})=0,\quad  (\rho{\chi}_{\mc{S}})|_{t=0}= \rho_0^{N}\mathds{1}_{\mc{S}_0},
\end{equation}
where $P_{\mc{S}}u \in \mc{R}$ and it is given by
\eqref{projection:P}.

Since $\rho_0^N \in W^{1,\infty}(\Omega)$, $u\in B_{R,T}$ in \eqref{eq:rho}, we can apply \cref{parabolic} to conclude that $\rho >0$ and 
\begin{equation*}
\rho \in L^2(0,T;H^{2}(\Omega)) \cap C([0,T]; H^{1}(\Omega)) \cap H^{1}(0,T;L^{2}(\Omega)).
\end{equation*}
Moreover, by \cref{reg:chiS} we obtain
\begin{align*}
&\chi_{\mc{S}} \in L^{\infty}((0,T)\times \Omega) \cap C([0,T];L^p(\Omega)),  \, \forall \, 1 \leq p < \infty,\\
&\rho\chi_{\mc{S}} \in L^{\infty}((0,T)\times \Omega) \cap C([0,T];L^p(\Omega)),  \, \forall \, 1 \leq p < \infty.
\end{align*}
Consequently, we define 
\begin{align*}
& \mu = (1-\chi_{\mc{S}})\mu_{\mc{F}} + \delta^2\chi_{\mc{S}},\quad  \lambda = (1-\chi_{\mc{S}})\lambda_{\mc{F}} + \delta^2\chi_{\mc{S}}\mbox{ so that }\mu >0,\ 2\mu+3\lambda \geq 0, \\
& g=(1-\chi_{\mc{S}})g_{\mc{F}} + \chi_{\mc{S}}g_{\mc{S}},\quad
p(\rho)= a\rho^{\gamma} + {\delta} \rho^{\beta}\quad
\mbox{ with } \quad
a = a_{\mc{F}} (1-\chi_{\mc{S}}).
\end{align*}
\underline{Step 2: Momentum equation.} Given $u\in B_{R,T}$,  let us consider the following equation satisfied by $\widetilde{u}: [0,T]\mapsto X_N$:
\begin{multline}\label{tilde-momentum}
- \int\limits_0^T\int\limits_{\Omega} \rho \Big(\widetilde{u}'(t)\cdot e_j +  (u \cdot \nabla e_j)\cdot \widetilde{u} \Big) + \int\limits_0^T\int\limits_{\Omega} \Big(2\mu\mathbb{D}(\widetilde{u}):\mathbb{D}(e_j) + \lambda\operatorname{div}\widetilde{u}\mathbb{I} : \mathbb{D}(e_j) - p(\rho)\mathbb{I}:\mathbb{D}(e_j)\Big) \\
+\int\limits_0^T\int\limits_{\Omega} \varepsilon \nabla e_j \nabla \rho \cdot \widetilde{u}
 + \alpha \int\limits_0^T\int\limits_{\partial \Omega} (\widetilde{u} \times \nu)\cdot (e_j \times \nu) + \alpha \int\limits_0^T\int\limits_{\partial \mc{S}^N(t)} [(\widetilde{u}-P_{\mc{S}}\widetilde{u})\times \nu]\cdot [(e_j-P_{\mc{S}}e_j)\times \nu] \\
  + \frac{1}{\delta}\int\limits_0^T\int\limits_{\Omega} \chi_{\mc{S}}(\widetilde{u}-P_{\mc{S}}\widetilde{u})\cdot (e_j-P_{\mc{S}}e_j) = \int\limits_0^T\int\limits_{\Omega}\rho g \cdot e_j,
\end{multline}
where $\rho$, $\chi_{\mc{S}}$ are defined as in Step 1. We can write
\begin{equation*}
\widetilde{u}(t,\cdot)= \sum\limits_{i=1}^N g_{i}(t) e_i, \quad \widetilde{u}(0)=u_{0}^N= \sum\limits_{i=1}^N \left(\int\limits_{\Omega} u_{0} \cdot e_i\right)e_i.
\end{equation*} 
Thus, we can identify the function $\widetilde{u}$ with its coefficients $\{g_{i}\}$ which satisfy the ordinary differential equation,
\begin{equation}\label{tildeu-ODE}
\sum\limits_{i=1}^N a_{i,j}g'_{i}(t) + \sum\limits_{i=1}^N b_{i,j}g_{i}(t) = f_j(t),\quad g_{i}(0)= \int\limits_{\Omega} u_{0}^N \cdot e_i,
\end{equation} 
where $a_{i,j}$, $b_{i,j}$ and $f_j$
are given by
\begin{align*}
a_{i,j} &= \int\limits_0^T\int\limits_{\Omega} \rho e_i e_j, \\
b_{i,j} &= \int\limits_0^T\int\limits_{\Omega} \rho (u\cdot \nabla e_j)\cdot e_i + \int\limits_0^T\int\limits_{\Omega} \Big(2\mu\mathbb{D}(e_i):\mathbb{D}(e_j) + \lambda\operatorname{div}e_i\mathbb{I} : \mathbb{D}(e_j) \Big) + \int\limits_0^T\int\limits_{\Omega} \varepsilon \nabla e_j \nabla \rho \cdot e_i \\
 &+ \alpha \int\limits_0^T\int\limits_{\partial \Omega} (e_i \times \nu)\cdot (e_j \times \nu) + \alpha \int\limits_0^T\int\limits_{\partial \mc{S}(t)} [(e_i-P_{\mc{S}}e_i)\times \nu]\cdot [(e_j-P_{\mc{S}}e_j)\times \nu] + \frac{1}{\delta}\int\limits_0^T\int\limits_{\Omega} \chi_{\mc{S}}(e_i-P_{\mc{S}}e_i)\cdot (e_j-P_{\mc{S}}e_j),\\
  f_j &= \int\limits_0^T\int\limits_{\Omega} \rho g\cdot e_j +  \int\limits_0^T\int\limits_{\Omega}p(\rho)\mathbb{I}:\mathbb{D}(e_j).
\end{align*}
Observe that the positive lower bound of $\rho$ in \cref{parabolic} guarantees the invertibility of the matrix $(a_{i,j}(t))_{1\leq i,j\leq N}$. We use the regularity of $\rho$ (\cref{parabolic}), of $\chi_{\mc{S}}$ and of the propagator associated to $P_{\mc{S}}u$  (\cref{reg:chiS}) to conclude the continuity of $(a_{i,j}(t))_{1\leq i,j\leq N}$, $(b_{i,j}(t))_{1\leq i,j\leq N}$, $(f_{i}(t))_{1\leq i \leq N}$. The existence and uniqueness theorem for ordinary differential equations gives that system \eqref{tildeu-ODE} has a unique solution defined on $[0,T]$
and therefore equation \eqref{tilde-momentum} has a unique solution 
\begin{equation*}
\widetilde{u} \in C([0,T]; X_N).
\end{equation*}

\underline{Step 3: Well-definedness of $\mc{N}$.}  
Let us define a map   
\begin{align*}
\mc{N}:  B_{R,T} &\rightarrow C([0,T],X_N) \\
  u &\mapsto \widetilde{u},
\end{align*}
where $\widetilde{u}$ satisfies \eqref{tilde-momentum}. 
Since we know the existence of $\widetilde{u} \in C([0,T]; X_N)$ to the problem \eqref{tilde-momentum}, we have that $\mc{N}$ is well-defined from $B_{R,T}$ to $C([0,T]; X_N)$. Now we establish the fact that $\mc{N}$ maps $B_{R,T}$ to itself
for suitable $R$ and $T$.

 We fix 
\begin{equation*}
0< \sigma < \frac{1}{2}\operatorname{dist}(\mc{S}_0,\partial \Omega).
\end{equation*}
Given $u\in B_{R,T}$, we want to estimate $\|\widetilde{u}\|_{L^{\infty}(0,T;L^2(\Omega))}$.
We have the following identities via a simple integration by parts:
\begin{align}\label{id:1}
\int\limits_0^t\int\limits_{\Omega} \rho \widetilde{u}'\cdot \widetilde{u} &=-\frac{1}{2}\int\limits_0^t\int\limits_{\Omega}\frac{\partial \rho}{\partial t}|\widetilde{u}|^2 + \frac{1}{2}(\rho|\widetilde{u}|^2)(t)-\frac{1}{2}\rho_0|u_0|^2, \\
\label{id:2}
\int\limits_0^T\int\limits_{\Omega} \rho (u\cdot\nabla)\widetilde{u}\cdot \widetilde{u} &= - \frac{1}{2}\int\limits_0^T\int\limits_{\Omega} \operatorname{div}(\rho u)|\widetilde{u}|^2,\\
\begin{split}\label{id:3}
\int\limits_{\Omega} \nabla (\rho^{\gamma})\cdot \widetilde{u} &= \frac{\gamma}{\gamma-1} \int\limits_{\Omega} \nabla (\rho^{\gamma-1})\cdot \rho \widetilde{u} = -\frac{\gamma}{\gamma-1}\int\limits_{\Omega}\rho^{\gamma-1} \operatorname{div}(\rho \widetilde{u})=\frac{1}{\gamma-1} \frac{d}{dt}\int\limits_{\Omega} \rho^{\gamma} - \frac{\varepsilon\gamma}{\gamma-1}\int\limits_{\Omega} \rho^{\gamma-1}\Delta\rho \\ 
&= \frac{1}{\gamma-1} \frac{d}{dt}\int\limits_{\Omega} \rho^{\gamma} + \varepsilon \gamma\int\limits_{\Omega} \rho^{\gamma-2}|\nabla \rho|^2 \geq \frac{1}{\gamma-1} \frac{d}{dt}\int\limits_{\Omega} \rho^{\gamma}.
\end{split}
\end{align}
Similarly,
\begin{equation}\label{nid:4}
\int\limits_{\Omega} \nabla (\rho^{\beta})\cdot \widetilde{u} = \frac{1}{\beta-1} \frac{d}{dt}\int\limits_{\Omega} \rho^{\beta} + \varepsilon \beta\int\limits_{\Omega} \rho^{\beta-2}|\nabla \rho|^2.
\end{equation}
We multiply equation \eqref{tilde-momentum} by $g_{j}$, add these equations for $j=1,2,...,N$, use the relations \eqref{id:1}--\eqref{nid:4} and the continuity equation \eqref{eq:rho} to obtain the following energy estimate:
\begin{multline}\label{energy:tildeu}
\int\limits_{\Omega}\Big(\frac{1}{2} \rho |\widetilde{u}|^2 + \frac{a}{\gamma-1}\rho^{\gamma} + \frac{\delta}{\beta-1}\rho^{\beta}\Big) + \int\limits_0^T\int\limits_{\Omega} \Big(2\mu|\mathbb{D}(\widetilde{u})|^2 + \lambda |\operatorname{div}\widetilde{u}|^2\Big)  + \delta\varepsilon \beta\int\limits_0^T\int\limits_{\Omega} \rho^{\beta-2}|\nabla \rho|^2
 + \alpha \int\limits_0^T\int\limits_{\partial \Omega} |\widetilde{u} \times \nu|^2 \\
 + \alpha \int\limits_0^T\int\limits_{\partial \mc{S}(t)} |(\widetilde{u}-P_{\mc{S}}\widetilde{u})\times \nu|^2 
  + \frac{1}{\delta}\int\limits_0^T\int\limits_{\Omega} \chi_{\mc{S}}|\widetilde{u}-P_{\mc{S}}\widetilde{u}|^2 \leq \int\limits_0^T\int\limits_{\Omega}\rho g \cdot \widetilde{u}
 + \int\limits_{\Omega} \Bigg( \frac{1}{2}\frac{\rho_0^N}{|q_0^N|^2}\mathds{1}_{\{\rho_0>0\}}  + \frac{a}{\gamma-1}(\rho_0^N)^{\gamma} + \frac{\delta}{\beta-1}(\rho_0^N)^{\beta} \Bigg)\\
 \leq \sqrt{\overline{\rho}}T\left(\frac{1}{2\widetilde{\varepsilon}}\|g\|^2_{L^{\infty}(0,T;L^2(\Omega))} + \frac{\widetilde{\varepsilon}}{2}\|\sqrt{\rho}\widetilde{u}\|^2_{L^{\infty}(0,T;L^2(\Omega))}\right) + \int\limits_{\Omega} \Bigg( \frac{1}{2}\frac{\rho_0^N}{|q_0^N|^2}\mathds{1}_{\{\rho_0>0\}}  + \frac{a}{\gamma-1}(\rho_0^N)^{\gamma} + \frac{\delta}{\beta-1}(\rho_0^N)^{\beta} \Bigg).
\end{multline}
  An appropriate choice of $\widetilde{\varepsilon}$ in \eqref{energy:tildeu} gives us 
 \begin{equation*}
\|\widetilde{u}\|^2_{L^{\infty}(0,T;L^2(\Omega))} \leq \frac{4{\overline{\rho}}}{\underline{\rho}}T^2\|g\|^2_{L^{\infty}(0,T;L^2(\Omega))} + \frac{4}{\underline{\rho}}E_0^N,
\end{equation*}
where $\overline{\rho}$ and $\underline{\rho}$ are the upper and lower bounds of $\rho$. In order to get $\|\widetilde{u}\|_{L^{\infty}(0,T;L^2(\Omega))} \leq R$,
  we need
 \begin{equation}\label{choice-R}
 R^2 \geq \frac{4{\overline{\rho}}}{\underline{\rho}}T^2\|g\|^2_{L^{\infty}(0,T;L^2(\Omega))} + \frac{4}{\underline{\rho}}E_0^N.
 \end{equation}
We also need to verify that for $T$ small enough and for any $u\in B_{R,T}$,  
\begin{equation}\label{no-collision}
\inf_{u\in B_{R,T}} \operatorname{dist}(\mc{S}(t),\partial \Omega) \geq 2\sigma> 0
\end{equation}
holds. We follow \cite[Proposition 4.6, Step 2]{MR3272367} and write $\mc{S}(t)=\eta_{t,0}(\mc{S}_0)$ with the isometric propagator  $\eta_{t,s}$ associated to the rigid field $P_{\mc{S}}u=h'(t) + \omega(t)\times (y-h(t))$. Then, proving \eqref{no-collision} is equivalent to establishing the following bound:
\begin{equation}\label{equivalent-T}
\sup_{t\in [0,T]}|\partial_t \eta_{t,0}(t,y)| < \frac{ \operatorname{dist}(\mc{S}_0,\partial\Omega) - 2\sigma}{T},\quad t\in [0,T],\, y\in \mc{S}_0.
\end{equation}
We have 
\begin{equation*}
|\partial_t \eta_{t,0}(t,y)|=|P_{\mc{S}}u(t, \eta_{t,0}(t,y))| \leq |h'(t)| + |\omega(t)||y-h(t)|.
\end{equation*}
Furthermore, if $\overline{\rho}$ is the upper bound of $\rho$, then for $u\in B_{R,T}$
\begin{equation}\label{18:37}
|h'(t)|^2 + J(t)\omega(t)\cdot \omega(t)= \int\limits_{\mc{S}(t)} \rho|P_{\mc{S}}u(t,\cdot)|^2 \leq \int\limits_{\Omega} \rho|u(t,\cdot)|^2 \leq \overline{\rho}R^2
\end{equation}
 for any $R$ and $t\in (0,T)$. As $J(t)$ is congruent to $J(0)$, they have the same eigenvalues and we have 
 \begin{equation*}
 \lambda_0|\omega(t)|^2 \leq J(t)\omega(t)\cdot \omega(t),
 \end{equation*}
 where $\lambda_0$ is the smallest eigenvalue of $J(0)$. Observe that for $t\in [0,T],\, y\in \mc{S}_0$, 
 \begin{align}
 \begin{split}\label{18:39}
 |h'(t)| + |\omega(t)||y-h(t)|&\leq \sqrt{2}(|h'(t)|^2 + |\omega(t)|^2|y-h(t)|^2)^{1/2} \leq \sqrt{2}\max\{1,|y-h(t)|\}(|h'(t)|^2 + |\omega(t)|^2)^{1/2}
 \\ &\leq C_0\left(|h'(t)|^2 + J(t)\omega(t)\cdot \omega(t)\right)^{1/2},
 \end{split}
 \end{align}
 where $C_0=\sqrt{2}\frac{\max\{1,|y-h(t)|\}}{\min\{1,\lambda_0\}
 ^{1/2}}$.
 Thus, with the help of \eqref{18:37}--\eqref{18:39} and the relation of $R$ in \eqref{choice-R}, we can conclude that any 
 \begin{equation}\label{choice-T}
 T < \frac{ \operatorname{dist}(\mc{S}_0,\partial\Omega) - 2\sigma}{C_0 |\overline{\rho}|^{1/2}[\frac{4{\overline{\rho}}}{\underline{\rho}}T^2\|g\|^2_{L^{\infty}(0,T;L^2(\Omega))} + \frac{4}{\underline{\rho}}E_0^N]^{1/2}},
 \end{equation}
 satisfies the relation \eqref{no-collision}.
Thus, we choose $T$  satisfying \eqref{choice-T} and fix it. Then we choose $R$ as in \eqref{choice-R} to conclude that $\mc{N}$ maps $B_{R,T}$ to itself.

\underline{Step 4: Continuity of $\mc{N}$.}  We show that if a sequence $\{u^k\} \subset B_{R,T}$ is such that $u^k \rightarrow u$ in $B_{R,T}$, then $\mc{N}(u^k) \rightarrow \mc{N}(u)$ in $B_{R,T}$. As $\mbox{span}(e_1,e_2,...,e_N)$ is a finite dimensional subspace of $\mc{D}(\overline{\Omega})$, we have $u^k \rightarrow u$ in $C([0,T];\mc{D}(\overline{\Omega}))$. Given $\{u^k\} \subset B_{R,T}$, we have that $\rho^k \in L^2(0,T;H^{2}(\Omega)) \cap C([0,T]; H^{1}(\Omega)) \cap H^{1}(0,T;L^{2}(\Omega))$ is the solution to 
\eqref{eq:rho}, $\chi_{\mc{S}}^k \mbox{ is bounded in }L^{\infty}((0,T)\times \mb{R}^3)\mbox{ satisfying }$ \eqref{eq:chinew}
and $\{\rho^{k}\chi_{\mc{S}}^{k}\}$ is a bounded sequence in $L^{\infty}((0,T)\times\mathbb{R}^3)$ satisfying
\eqref{eq:rhochinew}.
We apply \cref{sequential1} to obtain 
\begin{align*}
& \chi_{\mc{S}}^k \rightarrow  \chi_{\mc{S}} \mbox{ weakly-}* \mbox{ in }L^{\infty}((0,T)\times \mb{R}^3) \mbox{ and }\mbox{ strongly } \mbox{ in }C([0,T]; L^p_{loc}(\mb{R}^3)),  \, \forall \, 1 \leq p < \infty,\\
& P_{\mc{S}}^k u^k \rightarrow P_{\mc{S}} u \mbox{ strongly } \mbox{ in }C([0,T]; C^{\infty}_{loc}(\mb{R}^3)),\\
& \eta_{t,s}^k \rightarrow \eta_{t,s} \mbox{ strongly } \mbox{ in }C^{1}([0,T]^2; C^{\infty}_{loc}(\mb{R}^3)).
\end{align*} 
We use the continuity argument as in Step 2 to conclude 
\begin{equation*}
a^k_{i,j}\rightarrow a_{i,j}, \quad b^k_{i,j}\rightarrow b_{i,j}, \quad f^k_j \rightarrow f_j \mbox{ strongly in }C([0,T]),
\end{equation*}
and so we obtain 
\begin{equation*}
\mc{N}(u^k)=\widetilde{u}^k \rightarrow \widetilde{u}=\mc{N}(u)\mbox{ strongly in }C([0,T]; X_N). 
\end{equation*}

\underline{Step 5: Compactness of $\mc{N}$.}  If $\widetilde{u}(t)=\sum\limits_{i=1}^N g_i(t) e_i$, we can view \eqref{tildeu-ODE} as
\begin{equation*}
A(t)G'(t) + B(t)G(t) = F(t),
\end{equation*}
where $A(t)=(a_{i,j}(t))_{1\leq i,j\leq N},\quad B(t)=(b_{i,j}(t))_{1\leq i,j\leq N},\quad F(t)=(f_{i}(t))_{1\leq i \leq N},\quad G(t)=(g_i(t))_{1\leq i\leq N}$. We deduce
\begin{equation*}
|g'_i(t)| \leq R|A^{-1}(t)||B(t)| + |A^{-1}(t)||F(t)|.
\end{equation*}
Thus, we have 
\begin{equation*}
\sup_{t\in [0,T]} \Big(|g_i(t)| + |g'_i(t)|\Big) \leq C.
\end{equation*}
This also implies 
\begin{equation*}
\sup_{u\in B_{R,T}} \|\mc{N}(u)\|_{C^1([0,T]; X_N)} \leq C.
\end{equation*}
The $C^1([0,T]; X_N)$-boundedness of $\mc{N}(u)$ allows us to apply the Arzela-Ascoli theorem to obtain compactness of $\mc{N}$ in $B_{R,T}$.

Now we are in a position to apply Schauder's fixed point theorem to
$\mc{N}$ to conclude the existence of a fixed point $u^N \in B_{R,T}$. Then we define $\rho^N$ satisfying the continuity equation \eqref{galerkin-approx2} on $(0,T)\times \Omega$, and $\chi_{\mc{S}}^N=\mathds{1}_{\mc{S}^N}$ is the corresponding solution to the transport equation \eqref{galerkin-approx4} on $(0,T)\times \mathbb{R}^3$. It only remains to justify the momentum equation \eqref{galerkin-approx3}.  We multiply equation \eqref{tilde-momentum} by $\psi\in \mc{D}([0,T))$ to obtain:
\begin{multline}\label{22:49}
- \int\limits_0^T\int\limits_{\Omega} \rho^N \Big((u^N)'(t)\cdot \psi(t)e_j +  (u^N \cdot \nabla (\psi(t)e_j))\cdot {u}^N \Big)+\int\limits_0^T\int\limits_{\Omega} \varepsilon \nabla (\psi(t)e_j) \nabla \rho^N \cdot u^N
 + \alpha \int\limits_0^T\int\limits_{\partial \Omega} (u^N \times \nu)\cdot (\psi(t)e_j \times \nu)\\ + \int\limits_0^T\int\limits_{\Omega} \Big(2\mu^N\mathbb{D}({u}^N):\mathbb{D}(\psi(t)e_j) + \lambda^N\operatorname{div}{u}^N\mathbb{I} : \mathbb{D}(\psi(t)e_j) - p^{N}(\rho^N)\mathbb{I}:\mathbb{D}(\psi (t)e_j)\Big) \\
 + \alpha \int\limits_0^T\int\limits_{\partial \mc{S}^N(t)} [({u}^N-P^N_{\mc{S}}{u}^N)\times \nu]\cdot [(\psi(t)e_j-P^N_{\mc{S}}\psi(t)e_j)\times \nu] 
  + \frac{1}{\delta}\int\limits_0^T\int\limits_{\Omega} \chi_{\mc{S}}({u}^N-P^N_{\mc{S}}{u}^N)\cdot (\psi(t)e_j-P^N_{\mc{S}}\psi(t)e_j)\\ = \int\limits_0^T\int\limits_{\Omega}\rho^N g^N \cdot \psi(t)e_j,
\end{multline}
We have the following identities via integration by parts:
\begin{equation}\label{id:4}
\int\limits_0^T \rho^N (u^N)'(t)\cdot \psi(t)e_j = -\int\limits_0^T (\rho^N)' u^N\cdot \psi(t)e_j - \int\limits_0^T \rho^N u^N\cdot \psi'(t)e_j - (\rho^N u^N\cdot \psi e_j)(0),
\end{equation}
\begin{equation}\label{id:5}
\int\limits_{\Omega} \rho^N (u^N \cdot \nabla (\psi(t)e_j))\cdot {u}^N = -\int\limits_{\Omega} \operatorname{div}(\rho^N u^N) (\psi(t)e_j \cdot {u}^N) - \int\limits_{\Omega}{ \rho^N (u^N \cdot \nabla ) {u}^N\cdot \psi(t)e_j.}
\end{equation}
Thus we can use the relations \eqref{id:4}--\eqref{id:5} and continuity equation \eqref{galerkin-approx2} in the identity \eqref{22:49} to obtain equation \eqref{galerkin-approx3} for all $\phi \in \mc{D}([0,T); X_N)$.
\end{proof}
\subsection{Convergence of the Faedo-Galerkin scheme and the limiting system}\label{14:14}
In \cref{fa}, we have already constructed a solution $(\mc{S}^N,\rho^N,u^{N})$ to the problem \eqref{galerkin-approx1}--\eqref{galerkin-initial}. In this section, we establish \cref{thm:approxn} by passing to the limit in \eqref{galerkin-approx1}--\eqref{galerkin-initial} as $N\rightarrow\infty$ to recover the solution of \eqref{varepsilon:approx1}--\eqref{varepsilon:initial}, i.e.\ of the $\varepsilon$-level approximation. 
\begin{proof} [Proof of \cref{thm:approxn}]

 If we multiply \eqref{galerkin-approx3} by $u^N$, then as in \eqref{energy:tildeu}, we derive
\begin{multline}\label{energy:uN}
E^N[\rho ^N, q^N]  + \int\limits_0^T\int\limits_{\Omega} \Big(2\mu^N|\mathbb{D}(u^N)|^2 + \lambda^N |\operatorname{div}u^N|^2\Big)  + \delta\varepsilon \beta\int\limits_0^T\int\limits_{\Omega} (\rho^N)^{\beta-2}|\nabla \rho^N|^2
 + \alpha \int\limits_0^T\int\limits_{\partial \Omega} |u^N \times \nu|^2 
 \\+ \alpha \int\limits_0^T\int\limits_{\partial \mc{S}^N(t)} |(u^N-P^N_{\mc{S}}u^N)\times \nu|^2 
  + \frac{1}{\delta}\int\limits_0^T\int\limits_{\Omega} \chi^N_{\mc{S}}|u^N-P^N_{\mc{S}}u^N|^2 \leq \int\limits_0^T\int\limits_{\Omega}\rho^N g^N \cdot u^N 
 + E^N_0,
\end{multline}
where
 $$E^N[\rho ^N, q^N] = \int\limits_{\Omega}\Big(\frac{1}{2} \rho^N |u^N|^2 + \frac{a^N}{\gamma-1}(\rho^N)^{\gamma} + \frac{\delta}{\beta-1}(\rho^N)^{\beta}\Big).$$
Following the idea of the footnote in \cite[Page 368]{MR2084891}, the initial data $(\rho_0^N, u_0^N)$ is constructed in such a way that 
\begin{equation*}
\rho_0^N \rightarrow \rho_0^{\varepsilon} \mbox{ in }W^{1,\infty}(\Omega),\quad \rho_0^N u_0^N \rightarrow q_0^{\varepsilon} \mbox{ in }L^{2}(\Omega)
\end{equation*}
and 
\begin{equation} \label{lim}
\int\limits_{\Omega}\Bigg( \frac{1}{2}{\rho_0^N}|u_0^N|^2\mathds{1}_{\{\rho_0^N>0\}} + \frac{a^N}{\gamma-1}(\rho_0^N)^{\gamma} + \frac{\delta}{\beta-1}(\rho_0^N)^{\beta} \Bigg) \rightarrow \int\limits_{\Omega}\Bigg(\frac{1}{2} \frac{|q_0^{\varepsilon}|^2}{\rho_0^{\varepsilon}}\mathds{1}_{\{\rho_0^{\varepsilon}>0\}} + \frac{a^{\varepsilon}}{\gamma-1}(\rho_0^{\varepsilon})^{\gamma} + \frac{\delta}{\beta-1}(\rho_0^{\varepsilon})^{\beta} \Bigg)\mbox{ as }N\rightarrow \infty.
\end{equation}
 Precisely, we approximate $q_0^{\varepsilon}$ by a sequence $q_0^N$ satisfying \eqref{initialcond} and such that 
\eqref{lim} is valid. It is sufficient to take 
$ u_{0}^N = P_N(\frac{q_0^{\varepsilon}}{\rho_0^{\varepsilon}})$, where by $P_N$ we denote the orthogonal projection of $L^2(\Omega) \mbox { onto } X_N$.  \cref{fa} is valid with these new initial data. Therefore we can apply the arguments which we will explain below to get Proposition \ref{thm:approxn}.

The construction of $\rho^N$ and \eqref{bounds-on-rho} imply that $\rho^N >0$. Thus the energy estimate \eqref {energy:uN} yields that up to a subsequence 
\begin{enumerate}
\item $u^N\rightarrow u^{\varepsilon}$ weakly-$*$ in $L^{\infty}(0,T;L^2(\Omega))$ and weakly in $L^2(0,T;H^1(\Omega))$,
\item $\rho^N \rightarrow \rho^{\varepsilon}$ weakly-$*$ in $L^{\infty}(0,T; L^{\beta}(\Omega))$,
\item $\nabla\rho^N \rightarrow \nabla\rho^{\varepsilon}$ weakly in $L^{2}((0,T)\times\Omega)$.
\end{enumerate}
We follow the similar analysis as for the fluid case explained in \cite[Section 7.8.1, Page 362]{MR2084891} to conclude that
\begin{itemize}
\item $\rho^N \rightarrow \rho^{\varepsilon}$ in $C([0,T]; L^{\beta}_{weak}(\Omega))$ and $\rho^N \rightarrow \rho^{\varepsilon}$ strongly in $L^p((0,T)\times \Omega)$, $\forall \ 1\leq p< \frac{4}{3}\beta$,
\item $\rho^N u^N \rightarrow \rho^{\varepsilon} u^{\varepsilon}$ weakly in $L^2(0,T; L^{\frac{6\beta}{\beta+6}})$  and   weakly-$*$ in $L^{\infty}(0,T; L^{\frac{2\beta}{\beta+1}})$.  
\end{itemize}
We also know that $\chi_{\mc{S}}^N$ is a bounded sequence in $ L^{\infty}((0,T)\times \mb{R}^3)$ satisfying
\eqref{galerkin-approx4} and $\{\rho^{N}\chi_{\mc{S}}^{N}\}$ is a bounded sequence in $L^{\infty}((0,T)\times\mathbb{R}^3)$ satisfying
\eqref{N:approx5}. We use \cref{sequential11} to conclude 
\begin{align}\label{xi}
\chi_{\mc{S}}^N \rightarrow  \chi_{\mc{S}}^{\varepsilon} \mbox{ weakly-}* \mbox{ in }L^{\infty}((0,T)\times \mb{R}^3) &\mbox{ and }\mbox{ strongly } \mbox{ in } C([0,T]; L^p_{loc}(\mb{R}^3)),\ \forall  \ 1 \leq p <\infty,
\end{align}
with $\chi_{\mc{S}}^{\varepsilon}$ satisfying \eqref{varepsilon:approx4} along with \eqref{varepsilon:approx1}. Thus, we have recovered the transport equation for the body \eqref{varepsilon:approx4}. From \eqref{xi} and the definitions of $g^N$ and $g^{\varepsilon}$ in \eqref{gN} and \eqref{gepsilon}, it follows that
\begin{equation}\label{g}
g^N     \rightarrow g^{\varepsilon} \mbox{ weakly-}* \mbox{ in }L^{\infty}((0,T)\times \mb{R}^3) \mbox{ and }\mbox{ strongly } \mbox{ in }C([0,T]; L^p_{loc}(\mb{R}^3)) \ \forall \ 1 \leq p <\infty.
\end{equation}
These convergence results make it possible to pass to the limit $N\rightarrow \infty$ in \eqref{galerkin-approx2} to achieve \eqref{varepsilon:approx2}. Now we concentrate on the limit of the momentum equation \eqref{galerkin-approx3}. The four most difficult terms are:
\begin{align*}
A^N(t,e_k)&= \int\limits_{\partial \mc{S}^N(t)} [(u^N-P^N_{\mc{S}}u^N)\times \nu]\cdot [(e_k-P^N_{\mc{S}}e_k)\times \nu],\quad
B^N(t,e_k)= \int\limits_{\Omega} \rho^N u^N \otimes u^N : \nabla e_k,\\ C^N(t,e_k)&= \int\limits_{\Omega} \varepsilon \nabla u^N \nabla \rho^N \cdot e_k,\quad
D^N(t,e_k)=  \int\limits_{\Omega}  (\rho^N)^{\beta}\mathbb{I}: \mathbb{D}(e_k),\quad 1\leq k\leq N.
\end{align*}
To analyze the term $A^N(t,e_k)$, we do a change of variables to rewrite it in a fixed domain and use the convergence results from \cref{sequential1} for the projection and the isometric propagator: 
\begin{align*}
&P_{\mc{S}}^N u^N \rightarrow P_{\mc{S}}^{\varepsilon} u^{\varepsilon} \mbox{ weakly-}* \mbox{ in }L^{\infty}(0,T; C^{\infty}_{loc}(\mb{R}^3)),\\
&\eta_{t,s}^N \rightarrow \eta_{t,s}^{\varepsilon} \mbox{ weakly-}* \mbox{ in }W^{1,\infty}((0,T)^2; C^{\infty}_{loc}(\mb{R}^3)).
\end{align*}
We follow a similar analysis as in \cite[Page 2047--2048]{MR3272367} to conclude that $A^N$ converges weakly in $L^1(0,T)$ to  
\begin{equation*}
A(t,e_k)= \int\limits_{\partial \mc{S}^{\varepsilon}(t)} [(u^{\varepsilon}-P^{\varepsilon}_{\mc{S}}u^{\varepsilon})\times \nu]\cdot [(e_k-P^{\varepsilon}_{\mc{S}}e_{k})\times \nu].
\end{equation*}
We proceed as explained in the fluid case \cite[Section 7.8.2, Page 363--365]{MR2084891} to analyze the limiting process for the other terms $B^N(t,e_k)$, $C^N(t,e_k)$, $D^N(t,e_k)$. The limit of $B^N(t,e_k)$ follows from the fact \cite[Equation (7.8.22), Page 364]{MR2084891} that 
\begin{equation}\label{conv:convective}
\rho^N u^N \otimes u^N \rightarrow \rho^{\varepsilon} u^{\varepsilon} \otimes u^{\varepsilon} \mbox{ weakly in }L^2(0,T; L^{\frac{6\beta}{4\beta +3}}(\Omega)).
\end{equation}
To get the limit of $C^N(t,e_k)$, we use \cite[Equation (7.8.26), Page 365]{MR2084891}:
\begin{equation*}
\varepsilon\nabla u^N \nabla \rho^N \rightarrow \varepsilon\nabla u^{\varepsilon} \nabla \rho^{\varepsilon} \mbox{ weakly in }L^2(0,T; L^{\frac{5\beta-3}{4\beta}}(\Omega)),
\end{equation*}
and the limit of $D^N(t,e_k)$ is obtained by using \cite[Equation (7.8.8), Page 362]{MR2084891}:
\begin{equation}\label{conv:rho}
\rho^N \rightarrow \rho^{\varepsilon} \mbox{ strongly in }L^p(0,T; \Omega),\quad 1\leq p < \frac{4}{3}\beta.
\end{equation}
Thus, using the above convergence results for $B^N$, $C^N$, $D^N$ and the fact that
\begin{equation*}
\bigcup_{N}X_N\mbox{ is dense in }\left\{v\in W^{1,p}(\Omega) \mid v\cdot \nu=0\mbox{ on }\partial\Omega\right\}\mbox{ for any }p\in [1,\infty),
\end{equation*}
we conclude the following weak convergences in $L^1(0,T)$: 
\begin{equation*}
B^N(t,\phi^N)\rightarrow B(t,\phi^{\varepsilon})= \int\limits_{\Omega} \rho^{\varepsilon} u^{\varepsilon} \otimes u^{\varepsilon} : \nabla \phi^{\varepsilon},
\end{equation*}
\begin{equation*}
C^N(t,\phi^N)\rightarrow C(t,\phi^{\varepsilon})=\int\limits_{\Omega} \varepsilon \nabla u^{\varepsilon} \nabla \rho^{\varepsilon} \cdot \phi^{\varepsilon},
\end{equation*}
\begin{equation*}
D^N(t,\phi^N)\rightarrow D(t,\phi^{\varepsilon})= \int\limits_{\Omega}  (\rho^{\varepsilon})^{\beta}\mathbb{I}: \mathbb{D}(\phi^{\varepsilon}).
\end{equation*}
Thus we have achieved \eqref{varepsilon:approx2} as a limit of equation \eqref{galerkin-approx3} as $N\rightarrow \infty$. Hence, we have established the existence of a solution $(\mc{S}^{\varepsilon},\rho^{\varepsilon},u^{\varepsilon})$ to system \eqref{varepsilon:approx1}--\eqref{varepsilon:initial}. Now we establish energy inequality \eqref{energy-varepsilon} and estimates independent of $\varepsilon$:
\begin{itemize}
\item Notice that the solution $(\rho^N,u^N)$ of the Galerkin scheme satisfies \eqref{energy:uN} uniformly in $N$.  The convergence of $\rho^N|u^N|^2$ in \eqref{conv:convective} and $\rho^N$ in \eqref{conv:rho} ensures that, up to the extraction of a subsequence,
\begin{equation*}
\int\limits_{\Omega}\Big(\frac{1}{2} \rho^N |u^N|^2 + \frac{a^N}{\gamma-1}(\rho^N)^{\gamma} + \frac{\delta}{\beta-1}(\rho^N)^{\beta}\Big) \rightarrow \int\limits_{\Omega}\Big(\frac{1}{2} \rho^{\varepsilon} |u^{\varepsilon}|^2 + \frac{a^{\varepsilon}}{\gamma-1}(\rho^{\varepsilon})^{\gamma} + \frac{\delta}{\beta-1}(\rho^{\varepsilon})^{\beta}\Big) \mbox{ as }N\rightarrow\infty.
\end{equation*}
\item Due to the weak lower semicontinuity of convex functionals, the weak convergence of $u^N$ in $L^2(0,T;H^1(\Omega))$, the strong convergence of $\chi_{\mc{S}}^N$ in $C([0,T];L^p(\Omega))$ and the strong convergence of $P_{\mc{S}}^N$ in $C([0,T]; C^{\infty}_{loc}(\mb{R}^3))$, we obtain 
\begin{equation}\label{N1}
\int\limits_0^T\int\limits_{\Omega} \Big(2\mu^{\varepsilon}|\mathbb{D}(u^{\varepsilon})|^2 + \lambda^{\varepsilon}|\operatorname{div}u^{\varepsilon}|^2\Big) \leq\liminf_{N\rightarrow \infty}\int\limits_0^T\int\limits_{\Omega} \Big(2\mu^N|\mathbb{D}(u^N)|^2 + \lambda^N |\operatorname{div}u^N|^2\Big),
\end{equation}
\begin{equation}\label{N2}
\int\limits_0^T\int\limits_{\Omega} \chi^{\varepsilon}_{\mc{S}}|u^{\varepsilon}-P^{\delta}_{\mc{S}}u^{\varepsilon}|^2\leq\liminf_{N\rightarrow \infty}\int\limits_0^T\int\limits_{\Omega} \chi^N_{\mc{S}}|u^N-P_{\mc{S}}u^N|^2.
\end{equation}
\item Using the fact that $\nabla\rho^N\rightarrow\nabla\rho$ strongly in $L^2((0,T)\times\Omega)$ (by \cite[Equation (7.8.25), Page 365]{MR2084891}), strong convergence of $\rho^N$ in \eqref{conv:rho} and Fatou's lemma, we have
\begin{equation}\label{N3}
\int\limits_0^T\int\limits_{\Omega} (\rho^{\varepsilon})^{\beta-2}|\nabla \rho^{\varepsilon}|^2 \leq \liminf_{N\rightarrow \infty}\int\limits_0^T\int\limits_{\Omega} (\rho^N)^{\beta-2}|\nabla \rho^N|^2.
\end{equation}
\item For passing to the limit in the boundary terms, we follow the idea of \cite{MR3272367}. Define the extended velocities $U^N$, $U_{\mc{S}}^N$ to whole $\mathbb{R}^3$ associated with $u^N$, $P_{\mc{S}}u^N$ respectively. According to \cite[Lemma A.2]{MR3272367}, we have the weak convergences of $U^N$, $U_{\mc{S}}^N$ to $U^{\varepsilon}$, $U_{\mc{S}}^{\varepsilon}$ in $L^2(0,T;H^1_{loc}(\mathbb{R}^3))$. These facts along with the lower semicontinuity of the $L^2$-norm yield
\begin{align}
\int\limits_0^T\int\limits_{\partial \mc{S}^{\varepsilon}(t)} |(u^{\varepsilon}-P^{\varepsilon}_{\mc{S}}u^{\varepsilon})\times \nu|^2 &=\int\limits_0^T\int\limits_{\partial\mc{S}_0} |(U^{\varepsilon}-U_{\mc{S}}^{\varepsilon})\times \nu|^2\notag\\ &\leq \liminf_{N\rightarrow\infty} \int\limits_0^T\int\limits_{\partial\mc{S}_0} |(U^{N}-U_{\mc{S}}^{N})\times \nu|^2 \leq \liminf_{N\rightarrow\infty} \int\limits_0^T\int\limits_{\partial \mc{S}^N(t)} |(u^N-P_{\mc{S}}u^N)\times \nu|^2\label{N4}.
\end{align}
Similar arguments also help us to obtain
\begin{equation}\label{N5}
\int\limits_0^T\int\limits_{\partial \Omega} |u^{\varepsilon} \times \nu|^2 \leq \liminf_{N\rightarrow\infty}\int\limits_0^T\int\limits_{\partial \Omega} |u^N \times \nu|^2.
\end{equation}
\item Regarding the term on the right hand side of \eqref{energy:uN}, the weak convergence of $u^N$ in $L^2(0,T;H^1(\Omega))$, the strong convergence of $\rho^N$ in \eqref{conv:rho} and the strong convergence of $g^N$ in \eqref{g}  yield 
\begin{equation}\label{N6}
\int\limits_0^T\int\limits_{\Omega}  \rho^N g^N \cdot u^N \rightarrow\int\limits_0^T\int\limits_{\Omega} \rho^{\varepsilon} g^{\varepsilon} \cdot u^{\varepsilon}, \; 
\mbox{ as }N\rightarrow\infty.
\end{equation}
\end{itemize}
Thus, we have established energy inequality \eqref{energy-varepsilon}:
\begin{multline}\label{re:epsilon-energy}
E^{\varepsilon}[\rho ^{\varepsilon},q^{\varepsilon}]+ \int\limits_0^T\int\limits_{\Omega} \Big(2\mu^{\varepsilon}|\mathbb{D}(u^{\varepsilon})|^2 + \lambda^{\varepsilon}|\operatorname{div}u^{\varepsilon}|^2\Big) + \delta\varepsilon \beta\int\limits_0^T\int\limits_{\Omega} (\rho^{\varepsilon})^{\beta-2}|\nabla \rho^{\varepsilon}|^2 \\
 + \alpha \int\limits_0^T\int\limits_{\partial \Omega} |u^{\varepsilon} \times \nu|^2 
 + \alpha \int\limits_0^T\int\limits_{\partial \mc{S}^{\varepsilon}(t)} |(u^{\varepsilon}-P^{\varepsilon}_{\mc{S}}u^{\varepsilon})\times \nu|^2 
  + \frac{1}{\delta}\int\limits_0^T\int\limits_{\Omega} \chi^{\varepsilon}_{\mc{S}}|u^{\varepsilon}-P^{\delta}_{\mc{S}}u^{\varepsilon}|^2 \leq \int\limits_0^T\int\limits_{\Omega}\rho^{\varepsilon} 
  g^{\varepsilon} \cdot u^{\varepsilon}
  +   E^{\varepsilon}_0,
\end{multline}
where
 $$E^{\varepsilon}[\rho^{\varepsilon} ,q^{\varepsilon}] =\int\limits_{\Omega}\left(\frac{1}{2}\frac{|q^{\varepsilon}|^2}{\rho^{\varepsilon}} + \frac{a^{\varepsilon}}{\gamma-1}(\rho^{\varepsilon})^{\gamma} + \frac{\delta}{\beta-1}(\rho^{\varepsilon})^{\beta}\right).$$
We obtain as in \cite[Equation (7.8.14), Page 363]{MR2084891}:
\begin{equation*}
\partial_t\rho^{\varepsilon},\  \Delta \rho^{\varepsilon}\in {L^{\frac{5\beta-3}{4\beta}}((0,T)\times\Omega)}.
\end{equation*}
Regarding the $\sqrt{\varepsilon} \|\nabla \rho^{\varepsilon}\|_{L^2((0,T)\times\Omega)}$ estimate in \eqref{est:indofepsilon}, we have to multiply \eqref{varepsilon:approx2} by $\rho^{\varepsilon}$ and integrate by parts to obtain
\begin{equation*}
\frac{1}{2}\int\limits_{\Omega} |\rho^{\varepsilon}(t)|^2 + \varepsilon\int\limits_0^T\int\limits_{\Omega} |\nabla\rho^{\varepsilon}(t)|^2 = \frac{1}{2}\int\limits_{\Omega} |\rho_0^{\varepsilon}|^2 - \frac{1}{2}\int\limits_0^T\int\limits_{\Omega} |\rho^{\varepsilon}|^2\operatorname{div} u^{\varepsilon} \leq \frac{1}{2}\int\limits_{\Omega} |\rho_0^{\varepsilon}|^2 + \sqrt{T}\||\rho^{\varepsilon}\|^2_{L^{\infty}(0,T;L^4(\Omega))}\|\operatorname{div}u^{\varepsilon}\|_{L^2(0,T;L^2(\Omega))}.
\end{equation*}
Now, the pressure estimates $\|\rho^{\varepsilon}\|_{L^{\beta+1}((0,T)\times\Omega)}$ and $\|\rho^{\varepsilon}\|_{L^{\gamma+1}((0,T)\times\Omega)}$ in \eqref{est:indofepsilon} can be derived by means of the test function $\phi(t,x) = \psi(t)\Phi(t,x)$ with  $\Phi(t,x)=\mc{B}[ \rho^{\varepsilon}-\overline{m}]$ in \eqref{varepsilon:approx3}, where 
\begin{equation*}
\psi \in \mc{D}(0,T),\quad \overline{m}=|\Omega|^{-1}\int\limits_{\Omega} \rho^{\varepsilon},
\end{equation*}
and $\mc{B}$ is the Bogovskii operator related to $\Omega$ (for details about $\mc{B}$, see \cite[Section 3.3, Page 165]{MR2084891}). After taking this special test function and integrating by parts, we obtain
\begin{multline}\label{bogovski:mom}
\int\limits_0^T \psi\int\limits_{\Omega}\Big(a^{\varepsilon}(\rho^{\varepsilon})^{\gamma} + {\delta} (\rho^{\varepsilon})^{\beta}\Big) \rho^{\varepsilon}= \int\limits_0^T \psi\int\limits_{\Omega}\Big(a^{\varepsilon}(\rho^{\varepsilon})^{\gamma} + {\delta} (\rho^{\varepsilon})^{\beta}\Big) \overline{m} + \int\limits_0^T 2\psi\int\limits_{\Omega} \mu^{{\varepsilon}}\mathbb{D}(u^{\varepsilon}):\mathbb{D}(\Phi) + \int\limits_0^T \psi\int\limits_{\Omega}\lambda^{\varepsilon}\rho^{\varepsilon}\operatorname{div}u^{\varepsilon}\\- \overline{m}\int\limits_0^T \psi\int\limits_{\Omega}\lambda^{\varepsilon}\operatorname{div}u^{\varepsilon} + \int\limits_0^T \psi\int\limits_{\Omega} \varepsilon \nabla u^{\varepsilon} \nabla \rho^{\varepsilon} \cdot \Phi + \alpha \int\limits_0^T \psi\int\limits_{\partial \mc{S}^{\varepsilon}(t)} [(u^{\varepsilon}-P^{\varepsilon}_{\mc{S}}u^{\varepsilon})\times \nu]\cdot [(\Phi-P^{\varepsilon}_{\mc{S}}\Phi)\times \nu] \\
+ \frac{1}{\delta}\int\limits_0^T \psi\int\limits_{\Omega} \chi^{\varepsilon}_{\mc{S}}(u^{\varepsilon}-P^{\varepsilon}_{\mc{S}}u^{\varepsilon})\cdot (\Phi-P^{\varepsilon}_{\mc{S}}\Phi) + \int\limits_0^T \psi\int\limits_{\Omega} \rho^{\varepsilon} g^{\varepsilon} \cdot \Phi.
\end{multline}
We see that all the terms can be estimated as in \cite[Section 7.8.4, Pages 366--368]{MR2084891} except the penalization term. Using H\"{o}lder's inequality and bounds from energy estimate \eqref{energy-varepsilon}, the penalization term can be dealt with in the following way 
\begin{equation}\label{bogovski:extra}
\int\limits_0^T \psi\int\limits_{\Omega} \chi^{\varepsilon}_{\mc{S}}(u^{\varepsilon}-P^{\varepsilon}_{\mc{S}}u^{\varepsilon})\cdot (\Phi-P^{\varepsilon}_{\mc{S}}\Phi) \leq |\psi|_{C[0,T]} \left(\int\limits_0^T\int\limits_{\Omega} \chi^{\varepsilon}_{\mc{S}}|(u^{\varepsilon}-P^{\varepsilon}_{\mc{S}}u^{\varepsilon})|^2\right)^{1/2}\|\Phi\|_{L^2((0,T)\times\Omega)}\leq C |\psi|_{C[0,T]}, 
\end{equation}
where in the last inequality we have used $\|\Phi\|_{L^2(\Omega)}\leq c\|\rho^{\varepsilon}\|_{L^2(\Omega)}$ and the energy inequality \eqref{energy-varepsilon}. Thus, we have an improved regularity of the density and we have established the required estimates of \eqref{est:indofepsilon}. 

The only remaining thing is to check the following fact: there exists $T$ small enough such that if $\operatorname{dist}(\mc{S}_0,\partial \Omega) > 2\sigma$, then
\begin{equation}\label{epsilon-collision}
 \operatorname{dist}(\mc{S}^{\varepsilon}(t),\partial \Omega) \geq 2\sigma> 0 \quad \forall \ t\in [0,T].
\end{equation}
It is equivalent to establishing the following bound:
\begin{equation}\label{equivalent-Tag}
\sup_{t\in [0,T]}|\partial_t \eta_{t,0}(t,y)| < \frac{ \operatorname{dist}(\mc{S}_0,\partial\Omega) - 2\sigma}{T},\quad  y\in \mc{S}_0.
\end{equation}
We show as in Step 3 of the proof of \cref{fa} that (see \eqref{no-collision}--\eqref{18:39}): 
\begin{equation}\label{00:32}
|\partial_t \eta^{\varepsilon}_{t,0}(t,y)| \leq |(h{^\varepsilon})'(t)| + |\omega^{\varepsilon}(t)||y-h^{\varepsilon}(t)|\leq C_0\left(\int\limits_{\Omega} \rho^{\varepsilon} |u^{\varepsilon}(t)|^2\right)^{1/2},
\end{equation}
where $C_0=\sqrt{2}\frac{\max\{1,|y-h(t)|\}}{\min\{1,\lambda_0\}
 ^{1/2}}$. Moreover, the energy estimate  \eqref{re:epsilon-energy} yields
 \begin{equation*}
 \frac{d}{dt}E^{\varepsilon}[\rho ^{\varepsilon},q^{\varepsilon}]+ \int\limits_{\Omega} \Big(2\mu^{\varepsilon}|\mathbb{D}(u^{\varepsilon})|^2 + \lambda^{\varepsilon}|\operatorname{div}u^{\varepsilon}|^2\Big) \leq \int\limits_{\Omega}\rho^{\varepsilon} 
  g^{\varepsilon} \cdot u^{\varepsilon}\\
  \leq E^{\varepsilon}[\rho ^{\varepsilon},q^{\varepsilon}] + \frac{1}{2\gamma_1}\left(\frac{\gamma -1}{2\gamma}\right)^{\gamma_1/\gamma}\|g^{\varepsilon}\|^{2\gamma_1}_{L^{\frac{2\gamma}{\gamma -1}}(\Omega)},
  \end{equation*}
with $\gamma_1=1-\frac{1}{\gamma}$, which implies
 \begin{equation}\label{00:33}
 E^{\varepsilon}[\rho ^{\varepsilon},q^{\varepsilon}] \leq e^{{T}}E^{\varepsilon}_0 + C{T} \|g^{\varepsilon}\|^{2\gamma_1}_{L^{\infty}((0,T)\times\Omega)}.
 \end{equation}
 Thus, with the help of \eqref{equivalent-Tag} and \eqref{00:32}--\eqref{00:33}, we can conclude that for any $T$ satisfying
 \begin{equation*}
 T < \frac{ \operatorname{dist}(\mc{S}_0,\partial\Omega) - 2\sigma}{C_0 \left[e^{{T}}E^{\varepsilon}_0 + C{T} \|g^{\varepsilon}\|^{2\gamma_1}_{L^{\infty}((0,T)\times\Omega)}\right]^{1/2}},
 \end{equation*}
 the relation \eqref{epsilon-collision} holds.
This completes the proof of \cref{thm:approxn}.
\end{proof}
\subsection{Vanishing dissipation in the continuity equation and the limiting system}\label{14:18}
In this section, we prove \cref{thm:approxn-delta} by taking $\varepsilon\rightarrow 0$  in the system \eqref{varepsilon:approx1}--\eqref{varepsilon:initial}. In order to do so, we have to deal with the problem of identifying the pressure corresponding to the limiting density. 
First of all, following the idea of the footnote in \cite[Page 381]{MR2084891}, the initial data $(\rho_0^{\varepsilon}, q_0^{\varepsilon})$ is constructed in such a way that 
\begin{equation*}
\rho_0^{\varepsilon}>0,\quad \rho_0^{\varepsilon} \in W^{1,\infty}(\Omega),\quad \rho_0^{\varepsilon} \rightarrow \rho_0^{\delta} \mbox{ in }L^{\beta}(\Omega),\quad q_0^{\varepsilon} \rightarrow q_0^{\delta} \mbox{ in }L^{\frac{2\beta}{\beta + 1}}(\Omega)
\end{equation*}
and 
\begin{equation*}
\int\limits_{\Omega}\Bigg( \frac{|q_0^{\varepsilon}|^2}{\rho_0^{\varepsilon}}\mathds{1}_{\{\rho_0^{\varepsilon}>0\}} + \frac{a}{\gamma-1}(\rho_0^{\varepsilon})^{\gamma} + \frac{\delta}{\beta-1}(\rho_0^{\varepsilon})^{\beta} \Bigg) \rightarrow \int\limits_{\Omega}\Bigg( \frac{|q_0^{\delta}|^2}{\rho_0^{\delta}}\mathds{1}_{\{\rho_0^{\delta}>0\}} + \frac{a}{\gamma-1}(\rho_0^{\delta})^{\gamma} + \frac{\delta}{\beta-1}(\rho_0^{\delta})^{\beta} \Bigg)\mbox{ as }{\varepsilon}\rightarrow 0.
\end{equation*}
More precisely, let $(\rho^{\delta}_0,q^{\delta}_0)$ satisfy \eqref{rhonot}--\eqref{qnot}; then, following \cite[Section 7.10.7, Page 392]{MR2084891}, we can find $\rho^{\varepsilon}_{0} \in W^{1,\infty}({\Omega)}$, $\rho^{\varepsilon}_0 > 0$ by defining
\begin{equation*}
     \rho^{\varepsilon}_{0}= \mc{K}_{\varepsilon}(\rho^{\delta}_{0}) + \varepsilon,
    \end{equation*}
    where $\mc{K}_{\varepsilon}$ is the standard regularizing operator in the space variable.
   Then our initial density satisfies
    \begin{equation*}
        \begin{array}{l}
   \rho^{\varepsilon}_{0} \to \rho^{\delta}_{0} \mbox{ strongly in } L^{\beta}(\Omega)  .
    \end{array}
\end{equation*}
We define
\begin{align*}
\overline{{q}^{\varepsilon}_0}= \begin{cases} q_{0}^{\delta}\sqrt{\frac{\rho^{\varepsilon}_{0}}{\rho_{0}^{\delta}}} &\mbox { if } \rho^{\delta}_{0} >0,\\
0 \quad \quad \quad \quad \quad &\mbox { if } \rho^{\delta}_{0} =0.
\end{cases}
\end{align*}
From \eqref{qnot}, we know that
\begin{equation*}
\frac{|\overline{{q}^{\varepsilon}_0}|}{\sqrt{\rho^{\varepsilon}_{0}}} \in { L^2(\Omega)}.
\end{equation*}
Due to a density argument, there exists
 $h^{\varepsilon} \in W^{1,\infty}({\Omega})$ such that
\begin{equation*}
    \left\|\frac{q^{\varepsilon}_0}{\sqrt{\rho^{\varepsilon}_{0}}} -h^{\varepsilon} \right\|_{L^2(\Omega)}< \varepsilon.
\end{equation*}
Now, we set $ q^{\varepsilon}_0= h^{\varepsilon}\sqrt{\rho^{\varepsilon}_{0}}$, which implies that
 \begin{equation*}
    q^{\varepsilon}_0 \to q_{0}^{\delta} \mbox { in } L^{\frac{2\beta}{\beta +1}}(\Omega),
\end{equation*}
and 
\begin{equation*}
    E^{\varepsilon}_0  \to E^{\delta}_0. 
\end{equation*}  

\begin{proof} [Proof of \cref{thm:approxn-delta}] The estimates \eqref{energy-varepsilon} and  \eqref{est:indofepsilon} help us to conclude that, up to an extraction of a subsequence, we have 
\begin{align}
 & u^{\varepsilon}\rightarrow u^{\delta}\mbox{ weakly in }L^2(0,T; H^1(\Omega)),\label{conv1}\\
 &\rho^{\varepsilon}\rightarrow \rho^{\delta}\mbox{ weakly in }L^{\beta+1}((0,T)\times \Omega),\mbox{ weakly-}*\mbox{ in } L^{\infty}(0,T;L^{\beta}(\Omega)),\label{conv2}\\
 & (\rho^{\varepsilon})^{\gamma}\rightarrow \overline{ (\rho^{\delta})^{\gamma}}\mbox{ weakly in  }L^{\frac{\beta+1}{\gamma}}((0,T)\times\Omega),\label{conv3}\\
 & (\rho^{\varepsilon})^{\beta}\rightarrow \overline{ (\rho^{\delta})^{\beta}}\mbox{ weakly in } L^{\frac{\beta+1}{\beta}}((0,T)\times\Omega),\label{conv4}\\
 & \varepsilon\nabla\rho^{\varepsilon}\rightarrow 0 \mbox{ strongly in }L^2((0,T)\times \Omega)\label{conv5}
\end{align}
as $\varepsilon\to 0$. Below, we denote by $\left(\rho^{\delta},u^{\delta}, \overline{ (\rho^{\delta})^{\gamma}},\overline{ (\rho^{\delta})^{\beta}}\right)$ also the extended version of the corresponding quantities in $(0,T)\times \mb{R}^3$. 

\underline{Step 1: Limit of the transport equation.}
We obtain from \cref{thm:approxn} that $\rho^{\varepsilon}$ satisfies \eqref{varepsilon:approx2}, $\{u^{\varepsilon},\chi_{\mc{S}}^{\varepsilon}\}$ is a bounded sequence in $L^{2}(0,T; H^1(\Omega)) \times L^{\infty}((0,T)\times \mb{R}^3)$ satisfying \eqref{varepsilon:approx4}
and $\{\rho^{\varepsilon}\chi_{\mc{S}}^{\varepsilon}\}$ is a bounded sequence in $L^{\infty}((0,T)\times\mathbb{R}^3)$ satisfying \eqref{varepsilon:approx5}. Thus, we can use \cref{sequential-varepsilon} to conclude that up to a subsequence:
\begin{align}\label{13:02}
&\chi_{\mc{S}}^{\varepsilon} \rightarrow  \chi_{\mc{S}}^{\delta} \mbox{ weakly-}* \mbox{ in }L^{\infty}((0,T)\times \mb{R}^3) \mbox{ and }\mbox{ strongly } \mbox{ in }C([0,T]; L^p_{loc}(\mb{R}^3)) \ (1 \leq  p < \infty), \\
\label{18:11}
& \rho^{\varepsilon}\chi_{\mc{S}}^{\varepsilon} \rightarrow  \rho^{\delta}\chi_{\mc{S}}^{\delta} \mbox{ weakly-}* \mbox{ in }L^{\infty}((0,T)\times \mb{R}^3) \mbox{ and }\mbox{ strongly } \mbox{ in }C([0,T]; L^p_{loc}(\mb{R}^3)) \ (1 \leq  p < \infty),
\end{align}
with $\chi_{\mc{S}}^{\delta}$ and $\rho^{\delta}\chi_{\mc{S}}^{\delta}$  satisfying \eqref{approx4} and \eqref{approx5} respectively. Moreover,
\begin{equation}\label{13:01}
P_{\mc{S}}^{\varepsilon} u^{\varepsilon} \rightarrow P_{\mc{S}}^{\delta} u^{\delta} \mbox{ weakly } \mbox{in }L^{2}(0,T; C^{\infty}_{loc}(\mb{R}^3)).
\end{equation}
Hence, we have recovered  the regularity of $\chi_{\mc{S}}^{\delta}$ in \eqref{approx1} and the transport equations \eqref{approx4} and \eqref{approx5} as $\varepsilon\rightarrow 0$.

\underline{Step 2: Limit of the continuity and the momentum equation.}
We follow the ideas of \cite[Auxiliary lemma 7.49]{MR2084891} to conclude: if $\rho^{\delta}, u^{\delta}, \overline{ (\rho^{\delta})^{\gamma}}, \overline{ (\rho^{\delta})^{\beta}}$ are defined by \eqref{conv1}--\eqref{conv4}, we have 
\begin{itemize}
\item $(\rho^{\delta},u^{\delta})$ satisfies:
\begin{equation}\label{rho:delta} 
\frac{\partial {\rho}^{\delta}}{\partial t} + \operatorname{div}({\rho}^{\delta} u^{\delta}) =0 \mbox{ in }\mc{D}'([0,T)\times \mb{R}^3).
\end{equation} 
\item For all $\phi \in H^1(0,T; L^{2}(\Omega)) \cap L^r(0,T; W^{1,{r}}(\Omega))$, where $r=\max\left\{\beta+1, \frac{\beta+\theta}{\theta}\right\}$, $\beta \geq \max\{8,\gamma\}$ and $\theta=\frac{2}{3}\gamma -1$ with $\phi\cdot\nu=0$ on $\partial\Omega$ and  $\phi|_{t=T}=0$, the following holds: 
\begin{multline}\label{mom:delta}
- \int\limits_0^T\int\limits_{\Omega} \rho^{\delta} \left(u^{\delta}\cdot \frac{\partial}{\partial t}\phi +  u^{\delta} \otimes u^{\delta} : \nabla \phi\right) + \int\limits_0^T\int\limits_{\Omega} \Big(2\mu^{\delta}\mathbb{D}(u^{\delta}):\mathbb{D}(\phi) + \lambda^{\delta}\operatorname{div}u^{\delta}\mathbb{I} : \mathbb{D}(\phi) -  \left(a^{\delta}\overline{ (\rho^{\delta})^{\gamma}}+\delta \overline{ (\rho^{\delta})^{\beta}}\right)\mathbb{I}: \mathbb{D}(\phi)\Big) \\
 + \alpha \int\limits_0^T\int\limits_{\partial \Omega} (u^{\delta} \times \nu)\cdot (\phi \times \nu) + \alpha \int\limits_0^T\int\limits_{\partial \mc{S}^{\delta}(t)} \left[(u^{\delta}-P^{\delta}_{\mc{S}}u^{\delta})\times \nu\right]\cdot \left[(\phi-P^{\delta}_{\mc{S}}\phi)\times \nu\right] \\
  + \frac{1}{\delta}\int\limits_0^T\int\limits_{\Omega} \chi^{\delta}_{\mc{S}}(u^{\delta}-P^{\delta}_{\mc{S}}u^{\delta})\cdot (\phi-P^{\delta}_{\mc{S}}\phi) = \int\limits_0^T\int\limits_{\Omega}\rho^{\delta} g^{\delta} \cdot \phi
 + \int\limits_{\Omega} (\rho^{\delta} u^{\delta} \cdot \phi)(0).
\end{multline}
\item The couple $(\rho^{\delta},u^{\delta})$ satisfies the identity
\begin{equation}\label{renorm:delta}
\partial_t b(\rho^{\delta}) + \operatorname{div}(b(\rho^{\delta})u^{\delta})+[b'(\rho^{\delta})\rho^{\delta} - b(\rho^{\delta})]\operatorname{div}u^{\delta}=0 \mbox{ in }\mc{D}'([0,T)\times \mb{R}^3),
\end{equation}
 with any $b\in C([0,\infty)) \cap C^1((0,\infty))$ satisfying
\eqref{eq:b}.
\item $\rho^{\delta} \in C([0,T];L^p(\Omega))$, $1\leq p< \beta$.
\end{itemize}
We outline the main lines of the proof of the above mentioned result. We prove \eqref{rho:delta} by passing to the limit $\varepsilon\rightarrow 0$ in equation \eqref{varepsilon:approx2} with the help of the convergence of the density in \eqref{conv2}, \eqref{conv5} and the convergence of the momentum \cite[Section 7.9.1, page 370]{MR2084891}
\begin{align}\label{product1}
\rho^{\varepsilon}u^{\varepsilon}\rightarrow \rho^\delta u^\delta \mbox{ weakly-}* \mbox{ in }L^{\infty}(0,T;L^{\frac{2\beta}{\beta+1}}(\Omega)),\mbox{ weakly in }L^2(0,T;L^{\frac{6\beta}{\beta+6}}(\Omega)).
\end{align}
We obtain identity \eqref{mom:delta}, corresponding to the momentum equation, by passing to the limit in \eqref{varepsilon:approx3}.  To pass to the limit, we use the convergences of the density and  the velocity \eqref{conv1}--\eqref{conv4} and of the transport part \eqref{13:02}--\eqref{13:01}  along with the convergence of the product of the density and the velocity \eqref{product1} and the convergence of the following terms \cite[Section 7.9.1, page 371]{MR2084891}: 
\begin{align*}
&\rho^{\varepsilon}u^{\varepsilon}_i u^{\varepsilon}_j \rightarrow \rho^\delta u^\delta_i u^\delta_j\mbox{ weakly in }L^2(0,T;L^{\frac{6\beta}{4\beta + 3}}(\Omega)), \quad i,j=1,2,3,\\
&\varepsilon (\nabla \rho^{\varepsilon}\cdot \nabla)u^{\varepsilon}\rightarrow 0 \mbox{ weakly in }L^{\frac{5\beta-3}{4\beta}}((0,T)\times\Omega).
\end{align*}
Since, we have already established the continuity equation \eqref{rho:delta} and the function $b\in C([0,\infty)) \cap C^1((0,\infty))$ satisfies
\eqref{eq:b}, the renormalized continuity equation \eqref{renorm:delta} follows from the application of \cite[Lemma 6.9, page 307]{MR2084891}. Moreover, the regularity of the density $\rho^{\delta} \in C([0,T];L^p(\Omega))$, $1\leq p< \beta$ follows from \cite[Lemma 6.15, page 310]{MR2084891} via the appropriate choice of the renormalization function $b$ in \eqref{renorm:delta} and with the help of the regularities of $\rho^\delta\in L^{\infty}(0,T; L^{\beta}_{loc}(\mb{R}
^3)) \cap C([0,T];L^{\beta}_{loc}(\Omega))$, $u^\delta\in L^2(0,T; H^1_{loc}(\mb{R}^3))$. Hence we have established the continuity equation \eqref{approx2} and the renormalized one \eqref{rho:renorm1}.

\underline{Step 3: Limit of the pressure term.}
 In this step, our aim is to identify the term $\left(\overline{ (\rho^{\delta})^{\gamma}}+\delta \overline{ (\rho^{\delta})^{\beta}}\right)$ by showing that $\overline{ (\rho^{\delta})^{\gamma}}=(\rho^{\delta})^{\gamma}$ and $\overline{ (\rho^{\delta})^{\beta}}=(\rho^{\delta})^{\beta}$. To prove this, we need some compactness of $\rho^{\varepsilon}$, which is not available. However, the quantity $(\rho^{\varepsilon})^{\gamma}+\delta  (\rho^{\varepsilon})^{\beta}-(2\mu+\lambda)\rho
^{\varepsilon}\mathrm{div}u^{\varepsilon}$, called ``effective viscous flux", possesses a convergence property that helps us to identify the limit of our required quantity. We have the following weak and weak-$*$ convergences from the boundedness of their corresponding norms \cite[Section 7.9.2, page 373]{MR2084891}:
 \begin{align}
 \rho
^{\varepsilon}\mathrm{div}u^{\varepsilon}\rightarrow \overline{\rho^{\delta}\mathrm{div}u^{\delta}}&\mbox{ weakly in }L^2(0,T;L^{\frac{2\beta}{2+\beta}}(\Omega)),\label{conv6}\\
(\rho^{\varepsilon})^{\gamma+1}\rightarrow \overline{ (\rho^{\delta})^{\gamma+1}}&\mbox{ weakly-}*\mbox{ in } [C((0,T)\times\Omega)]',\label{conv7}\\
(\rho^{\varepsilon})^{\beta+1}\rightarrow \overline{ (\rho^{\delta})^{\beta+1}}&\mbox{ weakly-}*\mbox{ in } [C((0,T)\times\Omega)]'\label{conv8}.
 \end{align}
 We apply the following result regarding the "effective viscous flux" from \cite[Lemma 7.50, page 373]{MR2084891}:
 Let $u^{\delta}$, $\rho^{\delta}$, $\overline{ (\rho^{\delta})^{\gamma}}$, $\overline{ (\rho^{\delta})^{\beta}}$, $\overline{ (\rho^{\delta})^{\gamma+1}}$, $\overline{ (\rho^{\delta})^{\beta+1}}$, $\overline{\rho^{\delta}\mathrm{div}u^{\delta}}$ be defined in \eqref{conv1}--\eqref{conv4}, \eqref{conv6}--\eqref{conv8}. Then we have 
 \begin{align}
 &\overline{ (\rho^{\delta})^{\gamma+1}}\in L^{\frac{\beta+1}{\gamma+1}}((0,T)\times\Omega),\quad \overline{ (\rho^{\delta})^{\beta+1}} \in L^1((0,T)\times \Omega),\label{effective1}\\
 &\overline{ (\rho^{\delta})^{\gamma+1}} + \delta\overline{ (\rho^{\delta})^{\beta+1}} - (2\mu+\lambda)\overline{\rho^{\delta}\mathrm{div}u^{\delta}}=\overline{ (\rho^{\delta})^{\gamma}}\rho^{\delta} + \delta\overline{ (\rho^{\delta})^{\beta}}\rho^{\delta} - (2\mu+\lambda)\rho^{\delta}\mathrm{div}u^{\delta}\mbox{ a.e. in }(0,T)\times\Omega \label{effective2}.
 \end{align}
 Using the above relation \eqref{effective1} and an appropriate choice of the renormalization function in \eqref{renorm:delta}, we deduce the strong convergence of the density as in \cite[Lemma 7.51, page 375]{MR2084891}: Let $\rho^{\delta}$, $\overline{ (\rho^{\delta})^{\gamma}}$, $\overline{ (\rho^{\delta})^{\beta}}$, $\overline{ (\rho^{\delta})^{\gamma+1}}$, $\overline{ (\rho^{\delta})^{\beta+1}}$ be defined in \eqref{conv2}--\eqref{conv4}, \eqref{conv6}--\eqref{conv7}. Then we have
 \begin{equation*}
 \overline{ (\rho^{\delta})^{\gamma}}= (\rho^{\delta})^{\gamma},\quad \overline{ (\rho^{\delta})^{\beta}}= (\rho^{\delta})^{\beta} \mbox{ a.e. in }(0,T)\times\Omega.
 \end{equation*} 
 In particular,
 \begin{equation}\label{strong:rhoepsilon}
 \rho^{\varepsilon}\rightarrow \rho^{\delta}\mbox{ strongly in }L^p((0,T)\times\Omega),\ 1\leq p < \beta+1.
 \end{equation}
 Thus, we have identified the pressure term in equation \eqref{mom:delta}. Hence, we have recovered the momentum equation \eqref{approx3} and we have proved the existence of a  weak solution $(\mc{S}^{\delta},\rho^{\delta},u^{\delta})$ to system \eqref{approx1}--\eqref{approx:initial}. It remains to prove the energy inequality \eqref{10:45} and the improved regularity for the density \eqref{rho:improved}.

 \underline{Step 4: Energy inequality and improved regularity of the density.} Due to the convergences 
 \begin{align*}
&\rho^{\varepsilon}u^{\varepsilon}_i u^{\varepsilon}_j \rightarrow \rho^\delta u^{\delta}_i u^{\delta}_j \mbox{ weakly in }L^2(0,T;L^{\frac{6\beta}{4\beta + 3}}(\Omega)), \quad i,j=1,2,3,\\
&\rho^{\varepsilon}\rightarrow \rho^{\delta}\mbox{ strongly in }L^p((0,T)\times\Omega),\ 1\leq p < \beta+1,
\end{align*}
we have 
\begin{align*}
&\int\limits_{\Omega} \rho^{\varepsilon}|u^{\varepsilon}|^2 \rightarrow \int\limits_{\Omega} \rho^{\delta}|u^{\delta}|^2 \mbox{ weakly in }L^2(0,T),\\
&\int\limits_{\Omega} \left((\rho^{\varepsilon})^{\gamma} +\delta(\rho^{\varepsilon})^{\beta}\right) \rightarrow \int\limits_{\Omega} \left((\rho^{\delta})^{\gamma} +\delta(\rho^{\delta})^{\beta}\right) \black\mbox{ weakly in }L^{\frac{\beta+1}{\beta}}(0,T).
\end{align*} 
In particular,
\begin{equation*}
 \int\limits_{\Omega}\Big( \rho^{\varepsilon} |u^{\varepsilon}|^2 + \frac{a^{\varepsilon}}{\gamma-1}(\rho^{\varepsilon})^{\gamma} + \frac{\delta}{\beta-1}(\rho^{\varepsilon})^{\beta}\Big) \rightarrow \int\limits_{\Omega}\Big( \rho^{\delta} |u^{\delta}|^2 + \frac{a^{\delta}}{\gamma-1}(\rho^{\delta})^{\gamma} + \frac{\delta}{\beta-1}(\rho^{\delta})^{\beta}\Big) \mbox{ as }\varepsilon\rightarrow 0.
\end{equation*}
Due to the weak lower semicontinuity of the corresponding $L^2$ norms, the weak convergence of $u^{\varepsilon}$ in $L^2(0,T;H^1(\Omega))$, the strong convergence of $\rho^{\varepsilon}$ in $L^p((0,T)\times\Omega),\ 1\leq p < \beta+1$, the strong convergence of
$\chi_{\mc{S}}^{\varepsilon}$ in $C([0,T];L^p(\Omega))$ and the strong convergence of
$P_{\mc{S}}^{\varepsilon}$ in $C([0,T]; C^{\infty}_{loc}(\mb{R}^3))$, we follow the idea explained in \eqref{N1}--\eqref{N6} to pass to the limit as $\varepsilon \rightarrow 0$ in the other terms of inequality \eqref{energy-varepsilon} to establish the energy inequality \eqref{10:45}.

To establish the regularity \eqref{rho:improved}, we use an appropriate test function of the type $$\mc{B}\left((\rho^{\delta})^{\theta} - |\Omega|^{-1}\int\limits_{\Omega}(\rho^{\delta})^{\theta}\right)$$ 
in the momentum equation \eqref{approx3}, where  $\mc{B}$ is the Bogovskii operator. The detailed proof is in the lines of \cite[Section 7.9.5, pages 376-381]{MR2084891} and the extra terms can be treated as we have already explained in \eqref{bogovski:mom}--\eqref{bogovski:extra}. Moreover, we follow the same idea as in the proof of \cref{thm:approxn} (precisely, the calculations in \eqref{epsilon-collision}--\eqref{00:33}) to conclude that there exists $T$ small enough such that if $\operatorname{dist}(\mc{S}_0,\partial \Omega) > 2\sigma$, then
\begin{equation}\label{delta-collision}
 \operatorname{dist}(\mc{S}^{\delta}(t),\partial \Omega) \geq 2\sigma> 0 \quad \forall \ t\in [0,T].
\end{equation}
This settles the proof of \cref{thm:approxn-delta}.
\end{proof}
\section{Proof of the main result}\label{S4}
We have already established the existence of a weak solution $(\mc{S}^{\delta},\rho^{\delta},u^{\delta})$ to system \eqref{approx1}--\eqref{approx:initial} in \cref{thm:approxn-delta}. In this section, we study the convergence analysis and the limiting behaviour of the solution as $\delta\rightarrow 0$ and recover a weak solution to system \eqref{mass:comfluid}--\eqref{initial cond:comp}, i.e., we show \cref{{exist:upto collision}}.
\begin{proof} [Proof of \cref{exist:upto collision}]

 \underline{Step 0: Initial data.}
  We consider initial data $\rho_{\mc{F}_0}$, $q_{\mc{F}_0}$, $\rho_{\mc{S}_0}$, $q_{\mc{S}_0}$ satisfying the conditions \eqref{init}--\eqref{init2}. In this step we present the construction of the approximate initial data $(\rho^{\delta}_0,q^{\delta}_0)$ satisfying \eqref{rhonot}--\eqref{qnot} so that, in the limit $\delta\rightarrow 0$, we can recover the initial data $\rho_{\mc{F}_0}$ and $q_{\mc{F}_0}$ on $\mc{F}_0$.
 We set
 \begin{equation*}
 \rho_0 = \rho_{\mc{F}_0}(1-\mathds{1}_{\mc{S}_0}) + \rho_{\mc{S}_0}\mathds{1}_{\mc{S}_0},
 \end{equation*}
  \begin{equation*}
 q_0 = q_{\mc{F}_0}(1-\mathds{1}_{\mc{S}_0}) + \rho_{\mc{S}_0}u_{\mc{S}_0}\mathds{1}_{\mc{S}_0}.
 \end{equation*}
 Similarly as in \cite[Section 7.10.7, Page 392]{MR2084891}, we can find $\rho^{\delta}_{0} \in L^{\beta}({\Omega)}$ by defining
\begin{equation}\label{init-apprx1}
     \rho^{\delta}_{0}= \mc{K}_{\delta}(\rho_{0}) + \delta,
    \end{equation}
    where $\mc{K}_{\delta}$ is the standard regularizing operator in the space variable.
   Then our initial density satisfies
    \begin{equation}\label{ini-rho}
        \begin{array}{l}
   \rho^{\delta}_{0} \to \rho_{0} \mbox{ strongly in } L^{\gamma }(\Omega)  .
    \end{array}
\end{equation}
We define
\begin{align} \label{init-apprx2}
\overline{{q}^{\delta}_0}= \begin{cases} q_{0}\sqrt{\frac{\rho^{\delta}_{0}}{\rho_{0}}} &\mbox { if } \rho _{0} >0,\\
0 \quad \quad \quad \quad \quad &\mbox { if } \rho _{0} =0.
\end{cases}
\end{align}
From \eqref{init1},  we know that
\begin{equation*}
\frac{|\overline{{q}^{\delta}_0}|}{\sqrt{\rho^{\delta}_{0}}} \in { L^2(\Omega)}.
\end{equation*}
Due to a density argument, there exists
 $h^{\delta} \in W^{1,\infty}({\Omega})$ such that
\begin{equation*}
    \left\|\frac{q^{\delta}_0}{\sqrt{\rho^{\delta}_{0}}} -h^{\delta} \right\|_{L^2(\Omega)}< \delta.
\end{equation*}
Now, we set $ q^{\delta}_0= h^{\delta}\sqrt{\rho^{\delta}_{0}}$, which implies that
 \begin{equation*}
    q^{\delta}_0 \to q_{0} \mbox { in } L^{\frac{2\gamma}{\gamma +1}}(\Omega)
\end{equation*}
and 
\begin{equation*} 
    E^{\delta} [\rho ^{\delta}_0, q^{\delta}_0] \to E[\rho_0,q_0]. 
\end{equation*}

Next we start with the sequence of approximate solutions $\rho^{\delta},u^{\delta}$ of the system \eqref{approx1}--\eqref{approx:initial} (\cref{thm:approxn-delta}). Since the energy $E^{\delta}[\rho ^{\delta}_0,q^{\delta}_0]$ is uniformly bounded with respect to $\delta$, we have
from inequality \eqref{10:45} that 
\begin{multline}\label{again-10:45}
\|\sqrt{\rho^{\delta}} u^{\delta}\|_{L^{\infty}(0,T;L^2(\Omega))}^2 + \|\rho^{\delta}\|_{L^{\infty}(0,T;L^{\gamma}(\Omega))}^2 + \|\sqrt{2\mu^{\delta}}\mathbb{D}(u^{\delta})\|_{L^2((0,T)\times\Omega)}^2 + \|\sqrt{\lambda^{\delta}}\operatorname{div}u^{\delta}\|_{L^2((0,T)\times\Omega)}^2 \\
  + \frac{1}{\delta}\|\sqrt{ \chi^{\delta}_{\mc{S}}}\left(u^{\delta}-P^{\delta}_{\mc{S}}u^{\delta}\right)\|^2_{L^2((0,T)\times \Omega)} \leq C,
  \end{multline}
  with $C$ independent of $\delta$.
  
  \underline{Step 1: Recovery of the transport equation for body.}
   Since $\{u^{\delta},\chi_{\mc{S}}^{\delta}\}$ is a bounded sequence in $L^{2}(0,T; L^2(\Omega)) \times L^{\infty}((0,T)\times \mb{R}^3)$ satisfying
\eqref{approx4}, we can apply \cref{sequential2} to conclude that: up to a subsequence, we have 
\begin{align}
 & u^{\delta} \rightarrow u \mbox{ weakly } \mbox{ in }L^{2}(0,T; L^{2}(\Omega)),\notag\\
& \chi_{\mc{S}}^{\delta} \rightarrow  \chi_{\mc{S}} \mbox{ weakly-}* \mbox{ in }L^{\infty}((0,T)\times \mb{R}^3) \mbox{ and }\mbox{ strongly } \mbox{ in }C([0,T]; L^p_{loc}(\mb{R}^3)) \ (1\leq p<\infty),\label{19:30}
\end{align}
with 
\begin{equation*}
\chi_{\mc{S}}(t,x)=\mathds{1}_{\mc{S}(t)}(x),\quad \mc{S}(t)=\eta_{t,0}(\mc{S}_0),
\end{equation*}
where $\eta_{t,s}\in H^{1}((0,T)^2; C^{\infty}_{loc}(\mb{R}^3))$ is the isometric propagator. Moreover,
\begin{align}
& P_{\mc{S}}^{\delta} u^{\delta} \rightarrow P_{\mc{S}} u \mbox{ weakly } \mbox{ in }L^{2}(0,T; C^{\infty}_{loc}(\mb{R}^3)),\label{prop1}\\
& \eta_{t,s}^{\delta} \rightarrow \eta_{t,s} \mbox{ weakly } \mbox{ in }H^{1}((0,T)^2; C^{\infty}_{loc}(\mb{R}^3)).\notag
\end{align}
Also, we obtain that $\chi_{\mc{S}}$ satisfies 
\begin{equation*}
\frac{\partial {\chi}_{\mc{S}}}{\partial t} + \operatorname{div}(P_{\mc{S}}u\chi_{\mc{S}}) =0 \, \mbox{ in }{\Omega},\quad \chi_{\mc{S}}(t,x)=\mathds{1}_{\mc{S}(t)}(x).
\end{equation*}
Now we set
\begin{equation}\label{uS}
u_{\mc{S}}=P_{\mc{S}}u
\end{equation}
to recover the transport equation \eqref{NO4}. Note that we have already recovered the regularity of $\chi_{\mc{S}}$ in  \eqref{NO1}.

Observe that the fifth term of inequality \eqref{again-10:45} yields 
\begin{equation}\label{12:04}
\sqrt{ \chi^{\delta}_{\mc{S}}}\left(u^{\delta}-P^{\delta}_{\mc{S}}u^{\delta}\right) \rightarrow 0 \mbox{ strongly in } L^2((0,T)\times \Omega).
\end{equation}  
 The strong convergence of $\chi_{\mc{S}}^{\delta}$ and the weak convergence of $u^{\delta}$ and $P^{\delta}_{\mc{S}}u^{\delta}$ imply that \begin{equation}\label{re:solidvel}
 \chi_{\mc{S}}\left(u-u_{\mc{S}}\right)=0.
 \end{equation}
 To analyze the behaviour of the velocity field in the fluid part, we introduce the following continuous extension operator:
 \begin{equation}\label{Eu1}
 \mc{E}_u^{\delta}(t): \left\{ u\in H^{1}(\mc{F}^{\delta}(t)),\ u\cdot \nu=0\mbox{ on }\partial\Omega\right\} \rightarrow  H^{1}(\Omega).
 \end{equation}
 Let us set
 \begin{equation}\label{Eu2}
 u^{\delta}_{\mc{F}}(t,\cdot)=\mc{E}_u^{\delta}(t)\left[u^{\delta}(t,\cdot)|_{\mc{F}^{\delta}}\right].
 \end{equation}
 We have
 \begin{equation}\label{ext:fluid}
 \{u^{\delta}_{\mc{F}}\} \mbox{ is bounded in }L^2(0,T;H^1(\Omega)),\quad  u^{\delta}_{\mc{F}}=u^{\delta}  \mbox{ on } \mc{F}^{\delta},\mbox{ i.e.\ } (1-\chi_{\mc{S}}^{\delta})(u^{\delta}-u^{\delta}_{\mc{F}})=0.
 \end{equation}
 Thus, the strong convergence of $\chi_{\mc{S}}^{\delta}$ and the weak convergence of $u^{\delta}_{\mc{F}}\rightarrow u_{\mc{F}}$ in $L^2(0,T;H^1(\Omega))$ yield that
 \begin{equation}\label{re:fluidvel}
 (1-\chi_{\mc{S}})\left(u-u_{\mc{F}}\right)=0.
 \end{equation}
 By combining the relations \eqref{re:solidvel}--\eqref{re:fluidvel}, we conclude that the limit $u$ of $u^{\delta}$ satisfies $u\in L^{2}(0,T; L^2(\Omega))$ and there exists $u_{\mc{F}}\in L^2(0,T; H^1(\Omega))$, $u_{\mc{S}}\in L^{2}(0,T; \mc{R})$ such that $u(t,\cdot)=u_{\mc{F}}(t,\cdot)$ on $\mc{F}(t)$ and $u(t,\cdot)=u_{\mc{S}}(t,\cdot)$ on $\mc{S}(t)$.

\underline{Step 2: Recovery of the continuity equations.}
We recall that $\rho^{\delta}{\chi}^{\delta}_{\mc{S}}(t,x)$ satisfies \eqref{approx5}, i.e.
\begin{equation*}
\frac{\partial }{\partial t}(\rho^{\delta}{\chi}^{\delta}_{\mc{S}}) + P^{\delta}_{\mc{S}}u^{\delta} \cdot \nabla (\rho^{\delta}{\chi}^{\delta}_{\mc{S}})=0,\quad  (\rho^{\delta}{\chi}^{\delta}_{\mc{S}})|_{t=0}=\rho_0^{\delta}\mathds{1}_{\mc{S}_0}.
\end{equation*}
We proceed as in \cref{sequential2} to conclude that
\begin{align}\label{19:31}
\rho^{\delta}\chi_{\mc{S}}^{\delta} \rightarrow  \rho\chi_{\mc{S}} \mbox{ weakly-}* \mbox{ in }L^{\infty}((0,T)\times \mb{R}^3) &\mbox{ and }\mbox{ strongly } \mbox{ in }C([0,T]; L^p_{loc}(\mb{R}^3)) \ (1\leq p<\infty),
\end{align}
and $\rho\chi_{\mc{S}}$ satisfies
\begin{equation*}
\frac{\partial }{\partial t}(\rho{\chi}_{\mc{S}}) + P_{\mc{S}}u \cdot \nabla (\rho{\chi}_{\mc{S}})=0,\quad  (\rho{\chi}_{\mc{S}})|_{t=0}= \rho_{\mc{S}_0}\mathds{1}_{\mc{S}_0}.
\end{equation*}
We set
\begin{equation}\label{rhoS}
\rho_{\mc{S}}= \rho\chi_{\mc{S}}
\end{equation}
 and use the definition of $u_{\mc{S}}$ in \eqref{uS} to conclude that 
$\rho_{\mc{S}}$ satisfies:
\begin{equation*}
\frac{\partial {\rho}_{\mc{S}}}{\partial t} + \operatorname{div}(u_{\mc{S}}\rho_{\mc{S}}) =0 \, \mbox{ in }(0,T)\times{\Omega},\quad \rho_{\mc{S}}(0,x)=\rho_{\mc{S}_0}(x)\mathds{1}_{\mc{S}_0}\mbox{ in }\Omega.
\end{equation*}
Thus, we recover the equation of continuity \eqref{NO5} for the density of the rigid body.

We introduce the following extension operator:
\begin{equation*}
 \mc{E}_{\rho}^{\delta}(t): \left\{ \rho\in L^{\gamma+\theta}(\mc{F}^{\delta}(t))\right\} \rightarrow  L^{\gamma+\theta}(\Omega),
 \end{equation*}
 given by 
 \begin{equation}\label{Erho}
 \mc{E}_{\rho}^{\delta}(t)\left[\rho^{\delta}(t,\cdot)|_{\mc{F}^{\delta}}\right]=
 \begin{cases}
 \rho^{\delta}(t,\cdot)|_{\mc{F}^{\delta}}&\mbox{ in }\mc{F}^{\delta}(t),\\
 0 &\mbox{ in }\Omega\setminus \mc{F}^{\delta}(t).
 \end{cases}
 \end{equation}
 Let us set
 \begin{equation}\label{15:21}
 \rho^{\delta}_{\mc{F}}(t,\cdot)=E_{\rho}^{\delta}(t)\left[\rho^{\delta}(t,\cdot)|_{\mc{F}^{\delta}}\right].
 \end{equation}
From estimates \eqref{10:45}, \eqref{rho:improved}, \eqref{ext:fluid} and the definition of $\rho^{\delta}_{\mc{F}}$ in \eqref{15:21}, we obtain that
\begin{align}
 u^{\delta}_{\mc{F}}\rightarrow u_{\mc{F}}&\mbox{ weakly in }L^2(0,T; H^1(\Omega)),\label{ag:conv1}\\
 \rho^{\delta}_{\mc{F}}\rightarrow \rho_{\mc{F}}\mbox{ weakly in }L^{\gamma+\theta}((0,T)&\times \Omega),\ \theta=\frac{2}{3}\gamma-1\mbox{ and weakly-}*\mbox{ in } L^{\infty}(0,T;L^{\beta}(\Omega)),\label{ag:conv2}\\
 (\rho^{\delta}_{\mc{F}})^{\gamma}\rightarrow \overline{ \rho^{\gamma}_{\mc{F}}}&\mbox{ weakly in  }L^{\frac{\gamma+\theta}{\gamma}}((0,T)\times\Omega),\label{ag:conv3}\\
 \delta(\rho^{\delta}_{\mc{F}})^{\beta}\rightarrow 0&\mbox{ weakly in }L^{\frac{\beta+\theta}{\beta}}((0,T)\times \Omega)\label{ag:conv4}.
\end{align}
Next, we follow the ideas of \cite[Auxiliary lemma 7.53, Page 384]{MR2084891} to assert: \black if $u_{\mc{F}}, \rho_{\mc{F}}, \overline{ \rho_{\mc{F}}^{\gamma}}$ are defined by \eqref{ag:conv1}--\eqref{ag:conv3}, we have
\begin{itemize}
\item $(\rho_{\mc{F}},u_{\mc{F}})$ satisfies:
\begin{equation}\label{agrho:delta} 
\frac{\partial {\rho_{\mc{F}}}}{\partial t} + \operatorname{div}({\rho}_{\mc{F}} u_{\mc{F}}) =0 \mbox{ in }\mc{D}'([0,T)\times \mb{R}^3).
\end{equation} 
\item The couple $(\rho_{\mc{F}},u_{\mc{F}})$ satisfies the identity
\begin{equation}\label{agrenorm:delta}
\partial_t \overline{b(\rho_{\mc{F}})} + \operatorname{div}(\overline{b(\rho_{\mc{F}})}u_{\mc{F}})+\overline{[b'(\rho_{\mc{F}})\rho_{\mc{F}} - b(\rho_{\mc{F}})]\operatorname{div}u_{\mc{F}}}=0 \mbox{ in }\mc{D}'([0,T)\times \mb{R}^3),
\end{equation}
 for any $b\in C([0,\infty)) \cap C^1((0,\infty))$ satisfying
\eqref{eq:b} and the weak limits $\overline{b(\rho_{\mc{F}})}$ and $\overline{[b'(\rho_{\mc{F}})\rho_{\mc{F}} - b(\rho_{\mc{F}})]\operatorname{div}u_{\mc{F}}}$ being defined in the following sense: 
\begin{align*}
& b(\rho^{\delta}_{\mc{F}}) \rightarrow \overline{b(\rho_{\mc{F}})} \mbox{ weakly-}*\mbox{ in }L^{\infty}(0,T; L^{\frac{\gamma}{1+\kappa_1}}(\mb{R}^3)),\\
& [b'(\rho^{\delta}_{\mc{F}})\rho^{\delta}_{\mc{F}} - b(\rho^{\delta}_{\mc{F}})]\operatorname{div}u^{\delta}_{\mc{F}} \rightarrow \overline{[b'(\rho_{\mc{F}})\rho_{\mc{F}} - b(\rho_{\mc{F}})]\operatorname{div}u_{\mc{F}}} \mbox{ weakly in }L^2(0,T; L^{\frac{2\gamma}{2+2\kappa_1+\gamma}}(\mb{R}^3)).
\end{align*}
\end{itemize}
We outline the main idea of the proof of the asserted result. We derive \eqref{agrho:delta} by letting $\delta\rightarrow 0$ in equation \eqref{approx2} with the help of the convergence of the density in \eqref{ag:conv2} and the convergence of the momentum \cite[Section 7.10.1, page 383]{MR2084891}
\begin{align}\label{ag:product1}
\rho^{\delta}_{\mc{F}}u^{\delta}_{\mc{F}}\rightarrow \rho_{\mc{F}} u_{\mc{F}} \mbox{ weakly-}* \mbox{ in }L^{\infty}(0,T;L^{\frac{2\gamma}{\gamma+1}}(\mb{R}^3)),\mbox{ weakly in }L^2(0,T;L^{\frac{6\gamma}{\gamma+6}}(\mb{R}^3)).
\end{align}
Recall that when we pass to the limit $\varepsilon\rightarrow 0$, we do have $\rho_{\mc{F}}^{\delta} \in L^2((0,T)\times \Omega)$. But in this step, we do not have $\rho_{\mc{F}} \in L^2((0,T)\times \Omega)$.  So, it is not straightforward to obtain the renormalized continuity equation. Observe that this difficulty is not present in the case of $\gamma\geq \frac{9}{5}$ as in that case, $\rho_{\mc{F}}\in L^{\gamma+\theta}((0,T)\times\Omega)\subset L^2((0,T)\times\Omega)$ since $\gamma+\theta=\frac{5}{3}\gamma-1\geq 2$ for $\gamma\geq \frac{9}{5}$.

We use equation \eqref{rho:renorm1} and \cite[Lemma 1.7]{MR2084891} to establish that $\{b(\rho^{\delta}_{\mc{F}})\}$ is uniformly continuous in $W^{-1,s}(\Omega)$ with $s=\min\left\{\frac{6\gamma}{6\kappa_1+6+\gamma},2\right\}$, where the function $b\in C([0,\infty)) \cap C^1((0,\infty))$ satisfies
\eqref{eq:b}. We apply \cite[Lemma 6.2, Lemma 6.4]{MR2084891} to get
\begin{align*}
b(\rho^{\delta}_{\mc{F}}) \rightarrow \overline{b(\rho_{\mc{F}})} &\mbox{ in }C([0,T]; L^{\frac{\gamma}{1+\kappa_1}}(\Omega)),\\
b(\rho^{\delta}_{\mc{F}}) \rightarrow \overline{b(\rho_{\mc{F}})}& \mbox{ strongly in }L^p(0,T; W^{-1,2}(\Omega)), \quad 1\leq p < \infty.
\end{align*}
The above mentioned limits together with \eqref{ag:conv1} help us to conclude 
\begin{equation*}
b(\rho^{\delta}_{\mc{F}})u^{\delta}_{\mc{F}}\rightarrow \overline{b(\rho_{\mc{F}})}u_{\mc{F}} \mbox{ weakly in }L^2\left(0,T; L^{\frac{6\gamma}{6\kappa_1+6+\gamma}}(\Omega)\right).
\end{equation*}
Eventually, we obtain \eqref{agrenorm:delta} by taking the limit $\delta\rightarrow 0$ in \eqref{rho:renorm1}.

\underline{Step 3:  Recovery of the renormalized continuity equation.}
The method  of an effective viscous flux with an appropriate choice of functions \cite[Lemma 7.55, page 386]{MR2084891} helps us to establish boundedness of oscillations of the density sequence and we have an estimate for the amplitude of oscillations \cite[Lemma 7.56, page 386]{MR2084891}:

\begin{equation*}
\limsup_{\delta\rightarrow 0} \int\limits_{0}^T\int\limits_{\Omega} [T_k(\rho^{\delta}_{\mc{F}})-T_k(\rho_{\mc{F}})]^{\gamma+1} \leq \int\limits_{0}^T\int\limits_{\Omega} \left[\overline{\rho_{\mc{F}}^{\gamma}T_k(\rho_{\mc{F}})} - \overline{\rho^{\gamma}_{\mc{F}}}\overline{T_k(\rho_{\mc{F}})}\right],
\end{equation*}
where $T_k(\rho_{\mc{F}})=\min\{\rho_{\mc{F}},k\}$, $k >0$, are cut-off operators and $\overline{T_k(\rho_{\mc{F}})}$, $\overline{\rho_{\mc{F}}^{\gamma}T_k(\rho_{\mc{F}})}$ stand for the weak limits of $T_k(\rho_{\mc{F}}^{\delta})$, $(\rho_{\mc{F}}^{\delta})^{\gamma}T_k(\rho_{\mc{F}}^{\delta})$. This result allows us to estimate the quantities
\begin{align*}
& \sup_{\delta> 0} \|\rho_{\mc{F}}^{\delta}\mathds{1}_{\{\rho_{\mc{F}}^{\delta}\geq k\}}\|_{L^p((0,T)\times\Omega)}, \quad \sup_{\delta> 0} \|T_k(\rho_{\mc{F}}^{\delta})-\rho_{\mc{F}}^{\delta}\|_{L^p((0,T)\times\Omega)},\\ & \|\overline{T_k(\rho_{\mc{F}})}-\rho_{\mc{F}}\|_{L^p((0,T)\times\Omega)},\quad \|T_k(\rho_{\mc{F}})-\rho_{\mc{F}}\|_{L^p((0,T)\times\Omega)} \mbox{  with  }k>0, \ 1\leq p< \gamma+\theta.
\end{align*}
Using the above estimate and taking the renormalized function $b=T_k$ in \eqref{agrenorm:delta}, after several computations we obtain \cite[Lemma 7.57, page 388]{MR2084891}:  Let $b\in C([0,\infty)) \cap C^1((0,\infty))$ satisfy \eqref{eq:b} with $\kappa_1+1\leq \frac{\gamma+\theta}{2}$ and let $u_{\mc{F}}$, $\rho_{\mc{F}}$ be defined by \eqref{ag:conv1}--\eqref{ag:conv2}. Then we obtain the renormalized continuity equation \eqref{NO3}:
\begin{equation*}
\partial_t b(\rho_{\mc{F}}) + \operatorname{div}(b(\rho_{\mc{F}})u_{\mc{F}}) + (b'(\rho_{\mc{F}})-b(\rho_{\mc{F}}))\operatorname{div}u_{\mc{F}}=0 \mbox{ in }\, \mc{D}'([0,T)\times {\Omega}) .
\end{equation*}
So far, we have recovered the transport equation of the body \eqref{NO4}, the continuity equation \eqref{NO2} and the renormalized one \eqref{NO3}. It remains to prove the momentum equation \eqref{weak-momentum} and establish the energy inequality \eqref{energy}.

\underline{Step 4: Recovery of  the momentum equation.}
Notice that the test functions in the weak formulation of momentum equation \eqref{weak-momentum} belong to the space $V_T$ (the space is defined in \eqref{def:test}), which is a space of discontinuous functions. Precisely, 
\begin{equation*}
\phi=(1-\chi_{\mc{S}})\phi_{\mc{F}} + \chi_{\mc{S}}\phi_{\mc{S}}\mbox{ with }\phi_{\mc{F}}\in \mc{D}([0,T);\mc{D}(\overline{\Omega})),\quad \phi_{\mc{S}}\in \mc{D}([0,T);\mc{R}),
\end{equation*}
satisfying 
\begin{equation*}
\phi_{\mc{F}}\cdot \nu=0 \mbox{ on }\partial\Omega,\quad \phi_{\mc{F}}\cdot \nu= \phi_{\mc{S}}\cdot \nu\mbox{ on }\partial\mc{S}(t).
\end{equation*}
Whereas, if we look at the test functions in momentum equation \eqref{approx3} in the $\delta$-approximation, we see that it involves an $L^p(0,T; W^{1,p}(\Omega))$-regularity. Hence we  approximate this discontinuous test function by a sequence of test functions that belong to $L^p(0,T; W^{1,p}(\Omega))$. The idea is to construct an approximation $\phi^{\delta}_{\mc{S}}$ of $\phi$ without jumps at the interface such that 
\begin{equation}\label{cond1-good}
\phi^{\delta}_{\mc{S}}(t,x)=\phi_{\mc{F}}(t,x) \quad \forall\ t\in (0,T),\ x\in \partial\mc{S}^{\delta}(t), 
\end{equation}
and 
\begin{equation}\label{cond2-good}
\phi^{\delta}_{\mc{S}}(t,\cdot)\approx \phi_{\mc{S}}(t,\cdot)\mbox{ in }\mc{S}^{\delta}(t)\mbox{ away from a }\delta^{\vartheta}\mbox{ neighborhood of }\partial\mc{S}^{\delta}(t)\mbox{ with }\vartheta >0.
\end{equation}
In the spirit of \cite[Proposition 5.1]{MR3272367}, at first, we  give the precise result regarding this construction and then we will continue the proof of \cref{exist:upto collision}.
\begin{proposition}\label{approx:test}
Let $\phi \in V_T$ and $\vartheta >0$. Then there exists a sequence 
$$\phi^{\delta} \in H^1(0,T; L^{2}(\Omega)) \cap L^r(0,T; W^{1,{r}}(\Omega)),\mbox{ where } r=\max\left\{\beta+1, \frac{\beta+\theta}{\theta}\right\},\ \beta \geq \max\{8,\gamma\}\mbox{ and }\theta=\frac{2}{3}\gamma -1 $$ 
of the form
\begin{equation}\label{form:phi}
\phi^{\delta}=(1-\chi^{\delta}_{\mc{S}})\phi_{\mc{F}} + \chi^{\delta}_{\mc{S}}\phi^{\delta}_{\mc{S}}
\end{equation}
that satisfies for all $p\in [1,\infty)$:
\begin{enumerate}
\item $\|\chi^{\delta}_{\mc{S}}(\phi^{\delta}_{\mc{S}}-\phi_{\mc{S}})\|_{L^p((0,T)\times \Omega))}=\mc{O}(\delta^{\vartheta/p})$,
\item $\phi^{\delta} \rightarrow \phi$ strongly in $L^p((0,T)\times \Omega)$,
\item $\|\phi^{\delta}\|_{L^p(0,T;W^{1,p}(\Omega))}=\mc{O}(\delta^{-\vartheta(1-1/p)})$,
\item $\|\chi^{\delta}_{\mc{S}}(\partial_t + P^{\delta}_{\mc{S}}u^{\delta}\cdot\nabla)(\phi^{\delta}-\phi_{\mc{S}})\|_{L^2(0,T;L^p(\Omega))}=\mc{O}(\delta^{\vartheta/p})$,
\item $(\partial_t + P^{\delta}_{\mc{S}}u^{\delta}\cdot\nabla)\phi^{\delta}\rightarrow (\partial_t + P_{\mc{S}}u\cdot \nabla)\phi$ weakly in $L^2(0,T;L^p(\Omega))$.
\end{enumerate}
\end{proposition}

We give the proof of \cref{approx:test} at the end of this section. Next we continue the proof of \cref{exist:upto collision}.

\underline{Step 4.1: Linear terms of the momentum equation.}
We use $\phi^{\delta}$ (constructed in  \cref{approx:test}) as the test function in \eqref{approx3}. Then we take the limit $\delta\rightarrow 0$ in \eqref{approx3} to recover equation \eqref{weak-momentum}. Let us analyze the passage to the limit in the linear terms. To begin with, we recall the following convergences of the velocities of the fluid part and the solid part, cf.\ \eqref{ag:conv1} and \eqref{prop1}:
\begin{align*}
& (1-\chi_{\mc{S}}^{\delta})u^{\delta}_{\mc{F}}= (1-\chi^{\delta}_{\mc{S}})u^{\delta},\quad u^{\delta}_{\mc{F}}\rightarrow u_{\mc{F}}\mbox{ weakly in }L^2(0,T;H^1(\Omega)),\\
& u^{\delta}_{\mc{S}}=P^{\delta}_{\mc{S}}u^{\delta},\ u_{\mc{S}}=P_{\mc{S}}u,\quad u^{\delta}_{\mc{S}}\rightarrow u_{\mc{S}}\mbox{ weakly in }L^2(0,T; C^{\infty}_{loc}(\mb{R}^3)).
\end{align*}
Let us start with the diffusion term $2\mu^{\delta}\mathbb{D}(u^{\delta}):\mathbb{D}(\phi^{\delta}) + \lambda^{\delta}\operatorname{div}u^{\delta}\mathbb{I} : \mathbb{D}(\phi^{\delta})$ in \eqref{approx3}. We write
\begin{align*}
\int\limits_0^T\int\limits_{\Omega} 2\mu^{\delta}\mathbb{D}(u^{\delta}):\mathbb{D}(\phi^{\delta}) &= 
\int\limits_0^T\int\limits_{\Omega} \Big(2\mu_{\mc{F}}(1-\chi^{\delta}_{\mc{S}})\mb{D}(u_{\mc{F}}^{\delta}) + \delta^2 \chi_{\mc{S}}^{\delta}\mb{D}(u^{\delta})\Big): \mb{D}(\phi^{\delta})\\
&=\int\limits_0^T\int\limits_{\Omega} 2\mu_{\mc{F}}(1-\chi^{\delta}_{\mc{S}})\mb{D}(u_{\mc{F}}^{\delta}):\mb{D}(\phi_{\mc{F}}) + \delta^2\int\limits_0^T\int\limits_{\Omega} \chi_{\mc{S}}^{\delta}\mb{D}(u^{\delta}): \mb{D}(\phi^{\delta}).
\end{align*}
The strong convergence of $\chi^{\delta}_{\mc{S}}$ to $\chi_{\mc{S}}$ and the weak convergence of $u_{\mc{F}}^{\delta}$ to $u_{\mc{F}}$ imply that 
\begin{equation*}
\int\limits_0^T\int\limits_{\Omega} 2\mu_{\mc{F}}(1-\chi^{\delta}_{\mc{S}})\mb{D}(u_{\mc{F}}^{\delta}):\mb{D}(\phi_{\mc{F}}) \rightarrow \int\limits_0^T\int\limits_{\Omega} 2\mu_{\mc{F}}(1-\chi_{\mc{S}})\mb{D}(u_{\mc{F}}):\mb{D}(\phi_{\mc{F}}).
\end{equation*}
We know from  \eqref{again-10:45}, definition of $\mu^{\delta}$ in \eqref{approx-viscosity} and \cref{approx:test} (with $p=2$ case) that
\begin{equation*}
 \|\delta\chi^{\delta}_{\mc{S}}\mathbb{D}(u^{\delta})\|_{L^2((0,T)\times\Omega)} \leq C,\quad \|\phi^{\delta}\|_{L^2(0,T;H^1(\Omega))}=\mc{O}({\delta}^{-\vartheta/2}).
\end{equation*}
These estimates yield that
\begin{equation} \label{alpha2}
\left|\delta^2\int\limits_0^T\int\limits_{\Omega} \chi_{\mc{S}}^{\delta}\mb{D}(u^{\delta}): \mb{D}(\phi^{\delta})\right|\leq \delta\|\delta\chi^{\delta}_{\mc{S}}\mathbb{D}(u^{\delta})\|_{L^2((0,T)\times\Omega)}\|\mb{D}(\phi^{\delta})\|_{L^2(0,T;L^2(\Omega))}\leq C\delta^{1-\vartheta/2}.
\end{equation}
If we consider $\vartheta<2$ and $\delta\rightarrow 0$, we have
\begin{equation*}
\delta^2\int\limits_0^T\int\limits_{\Omega} \chi_{\mc{S}}^{\delta}\mb{D}(u^{\delta}): \mb{D}(\phi^{\delta}) \rightarrow 0. 
\end{equation*}
Hence, 
\begin{equation*}
\int\limits_0^T\int\limits_{\Omega} \Big(2\mu^{\delta}\mathbb{D}(u^{\delta}):\mathbb{D}(\phi^{\delta}) + \lambda_{\mc{F}}\operatorname{div}u_{\mc{F}}\mathbb{I} : \mathbb{D}(\phi^{\delta})\Big)\rightarrow \int\limits_0^T\int\limits_{\mc{F}(t)} \Big(2\mu_{\mc{F}}\mb{D}(u_{\mc{F}}) + \lambda_{\mc{F}}\operatorname{div}u_{\mc{F}}\mathbb{I}\Big):\mb{D}(\phi_{\mc{F}})
\end{equation*}
as $\delta\rightarrow 0$. Next we consider the boundary term on $\partial\Omega$ in \eqref{approx3}. The weak convergence of $u_{\mc{F}}^{\delta}$ to $u_{\mc{F}}$ in $L^2(0,T;H^1(\Omega))$ yields
\begin{equation*}
\int\limits_0^T\int\limits_{\partial \Omega} (u^{\delta} \times \nu)\cdot (\phi^{\delta} \times \nu)=\int\limits_0^T\int\limits_{\partial \Omega} (u_{\mc{F}}^{\delta}\times \nu)\cdot (\phi_{\mc{F}}\times \nu)
\rightarrow\int\limits_0^T\int\limits_{\partial \Omega} (u_{\mc{F}}\times \nu)\cdot (\phi_{\mc{F}}\times \nu) \mbox{ as }\delta\rightarrow 0.
\end{equation*}
To deal with the boundary term on $\partial \mc{S}^{\delta}(t)$ we do a change of variables such that this term becomes an integral on the fixed boundary $\partial \mc{S}_0$. Then we pass to the limit as $\delta\rightarrow 0$ and afterwards transform back to the moving domain. Next, we  introduce the notation $r^{\delta}_{\mc{S}}=P^{\delta}_{\mc{S}}\phi^{\delta}$ to write the following:
\begin{align*}
\int\limits_0^T\int\limits_{\partial \mc{S}^{\delta}(t)} [(u^{\delta}-P^{\delta}_{\mc{S}}u^{\delta})\times \nu]\cdot [(\phi^{\delta}-P^{\delta}_{\mc{S}}\phi^{\delta})\times \nu]=& \int\limits_0^T\int\limits_{\partial \mc{S}^{\delta}(t)} [(u^{\delta}_{\mc{F}}-u^{\delta}_{\mc{S}})\times \nu]\cdot [(\phi^{\delta}_{\mc{F}}-r^{\delta}_{\mc{S}})\times \nu]\\=& \int\limits_0^T\int\limits_{\partial \mc{S}_0} [(U^{\delta}_{\mc{F}}-U^{\delta}_{\mc{S}})\times \nu]\cdot [(\Phi^{\delta}_{\mc{F}}-R^{\delta}_{\mc{S}})\times \nu]
,\end{align*}
where we denote by capital letters the corresponding velocity fields and test functions in the fixed domain. By \cref{approx:test} we have that $\phi^{\delta} \rightarrow \phi$ strongly in $L^2(0,T;L^6(\Omega))$. Hence we obtain, as in \cref{sequential2}, that
\begin{equation*}
r^{\delta}_{\mc{S}} \rightarrow r_{\mc{S}}=P_{\mc{S}}\phi \mbox{ strongly in }
L^2(0,T; C^{\infty}_{loc}(\mb{R}^3)).
\end{equation*}
Now using \cite[Lemma A.2]{MR3272367}, we obtain the convergence in the fixed domain
\begin{equation*}
R^{\delta}_{\mc{S}} \rightarrow R_{\mc{S}} \mbox{ strongly in }
L^2(0,T; H^{1/2}(\partial\mc{S}_0)).
\end{equation*}
Similarly, the convergences of $u^{\delta}_{\mc{F}}$ and $u^{\delta}_{\mc{S}}$ with \cite[Lemma A.2]{MR3272367} imply
\begin{equation*}
U^{\delta}_{\mc{F}}\rightarrow U_{\mc{F}},\ U^{\delta}_{\mc{S}}\rightarrow U_{\mc{S}}\quad\mbox{weakly in }L^2(0,T;H^1(\Omega)).
\end{equation*}
These convergence results and going back to the moving domain gives
\begin{multline*}
\int\limits_0^T\int\limits_{\partial \mc{S}^{\delta}(t)} [(u^{\delta}-P^{\delta}_{\mc{S}}u^{\delta})\times \nu]\cdot [(\phi^{\delta}-P^{\delta}_{\mc{S}}\phi^{\delta})\times \nu]=  \int\limits_0^T\int\limits_{\partial \mc{S}_0} [(U^{\delta}_{\mc{F}}-U^{\delta}_{\mc{S}})\times \nu]\cdot [(\Phi^{\delta}_{\mc{F}}-R^{\delta}_{\mc{S}})\times \nu]\\ \rightarrow \int\limits_0^T\int\limits_{\partial \mc{S}_0} [(U_{\mc{F}}-U_{\mc{S}})\times \nu]\cdot [(\Phi_{\mc{F}}-R_{\mc{S}})\times \nu]= \int\limits_0^T\int\limits_{\partial \mc{S}(t)} [(u_{\mc{F}}-u_{\mc{S}})\times \nu]\cdot [(\phi_{\mc{F}}-\phi_{\mc{S}})\times \nu].
\end{multline*}
The penalization term can be estimated in the following way: 
\begin{align}
\left|\frac{1}{\delta}\int\limits_0^T\int\limits_{\Omega} \chi^{\delta}_{\mc{S}}(u^{\delta}-P^{\delta}_{\mc{S}}u^{\delta})\cdot (\phi^{\delta}-P^{\delta}_{\mc{S}}\phi^{\delta})\right|=&\left|\frac{1}{\delta}\int\limits_0^T\int\limits_{\Omega} \chi^{\delta}_{\mc{S}}(u^{\delta}-P^{\delta}_{\mc{S}}u^{\delta})\cdot ((\phi^{\delta}-\phi_{\mc{S}})-P^{\delta}_{\mc{S}}(\phi^{\delta}-\phi_{\mc{S}}))\right|\notag\\=\left|\frac{1}{\delta}\int\limits_0^T\int\limits_{\Omega} \chi^{\delta}_{\mc{S}}(u^{\delta}-P^{\delta}_{\mc{S}}u^{\delta})\cdot (\phi^{\delta}_{\mc{S}}-\phi_{\mc{S}})\right| &\leq \delta^{-1/2}\frac{1}{\delta^{1/2}}\left\|\sqrt{ \chi^{\delta}_{\mc{S}}}\left(u^{\delta}-P^{\delta}_{\mc{S}}u^{\delta}\right)\right\|_{L^2((0,T)\times \Omega)}\left\|\sqrt{\chi^{\delta}_{\mc{S}}}(\phi^{\delta}_{\mc{S}}-\phi_{\mc{S}})\right\|_{L^2(0,T;L^2(\Omega)))}\notag\\&
\leq C\delta^{-1/2+\vartheta/2},\label{12:27}
\end{align}
where we have used the estimates obtained from \eqref{again-10:45} and \cref{approx:test}. By choosing $\vartheta>1$ and taking $\delta\rightarrow 0$, the penalization term vanishes.  Note that we also need $\vartheta <2$ in view of \eqref{alpha2}.

\underline{Step 4.2: Nonlinear terms of the momentum equation.}
In this step, we analyze the following terms: 
\begin{multline}\label{convection}
\int\limits_0^T\int\limits_{\Omega} \rho^{\delta} \left(u^{\delta}\cdot \frac{\partial}{\partial t}\phi +  u^{\delta} \otimes u^{\delta} : \nabla \phi^{\delta}\right) = \int\limits_0^T\int\limits_{\Omega} \rho^{\delta}_{\mc{F}}(1-\chi_{\mc{S}}^{\delta}) u^{\delta}_{\mc{F}}\cdot \frac{\partial}{\partial t}\phi_{\mc{F}} + \int\limits_0^T\int\limits_{\Omega} \rho^{\delta}_{\mc{F}}(1-\chi_{\mc{S}}^{\delta}) u^{\delta}_{\mc{F}} \otimes u^{\delta}_{\mc{F}} : \nabla \phi_{\mc{F}} \\+  \int\limits_0^T\int\limits_{\Omega}\rho^{\delta}\chi^{\delta}_{\mc{S}}(\partial_t+u^{\delta}_{\mc{S}}\cdot\nabla)\phi^{\delta}\cdot u^{\delta}
\end{multline}
The strong convergence of $\chi^{\delta}_{\mc{S}}$ to $\chi_{\mc{S}}$ 
and the weak convergence of $\rho^{\delta}u^{\delta}_{\mc{F}}$ to $\rho_{\mc{F}}u_{\mc{F}}$ (see \eqref {ag:product1}) imply 
\begin{equation}\label{A1}
\int\limits_0^T\int\limits_{\Omega} \rho^{\delta}_{\mc{F}}(1-\chi_{\mc{S}}^{\delta}) u^{\delta}_{\mc{F}}\cdot \frac{\partial}{\partial t}\phi_{\mc{F}} \rightarrow \int\limits_0^T\int\limits_{\Omega} \rho_{\mc{F}}(1-\chi_{\mc{S}}) u_{\mc{F}}\cdot \frac{\partial}{\partial t}\phi_{\mc{F}} \quad\mbox{ as }\delta\rightarrow 0.
\end{equation}
 We use the convergence result for the convective term from  \cite[Section 7.10.1, page 384]{MR2084891} 
\begin{equation*}
\rho^{\delta}_{\mc{F}}(u^{\delta}_{\mc{F}})_i (u^{\delta}_{\mc{F}})_j \rightarrow \rho_{\mc{F}} (u_{\mc{F}})_i (u_{\mc{F}})_j\mbox{ weakly in }L^2(0,T;L^{\frac{6\gamma}{4\gamma + 3}}(\Omega)), \quad \forall i,j\in \{1,2,3\}, 
\end{equation*}
to pass to the limit in the second term of the right-hand side of \eqref{convection}:
\begin{equation}\label{A2}
\int\limits_0^T\int\limits_{\Omega} \rho^{\delta}_{\mc{F}}(1-\chi_{\mc{S}}^{\delta}) u^{\delta}_{\mc{F}} \otimes u^{\delta}_{\mc{F}} : \nabla \phi_{\mc{F}} \rightarrow \int\limits_0^T\int\limits_{\Omega} \rho_{\mc{F}}(1-\chi_{\mc{S}}) u_{\mc{F}} \otimes u_{\mc{F}} : \nabla \phi_{\mc{F}}.
\end{equation}
Next we consider the third term on the right-hand side of \eqref{convection}:
\begin{multline*}
\int\limits_0^T\int\limits_{\Omega}\rho^{\delta}\chi^{\delta}_{\mc{S}}(\partial_t+u^{\delta}_{\mc{S}}\cdot\nabla)\phi^{\delta}\cdot u^{\delta}=\int\limits_0^T\int\limits_{\Omega}\rho^{\delta}\chi_{\mc{S}}^{\delta}(\partial_t+u^{\delta}_{\mc{S}}\cdot\nabla)(\phi^{\delta}-\phi_{\mc{S}})\cdot u^{\delta} + \int\limits_0^T\int\limits_{\Omega}\rho^{\delta}\chi_{\mc{S}}^{\delta}\partial_t\phi_{\mc{S}}\cdot u^{\delta}\\ + \int\limits_0^T\int\limits_{\Omega}\rho^{\delta}\chi_{\mc{S}}^{\delta}(u^{\delta}_{\mc{S}}\cdot\nabla)\phi_{\mc{S}}\cdot u^{\delta}=: T_1^{\delta} + T_2^{\delta}+ T_3^{\delta}.
\end{multline*}
We write 
\begin{equation*}
T_1^{\delta} = \int\limits_0^T\int\limits_{\Omega}\rho^{\delta}\chi_{\mc{S}}^{\delta}(\partial_t+u^{\delta}_{\mc{S}}\cdot\nabla)(\phi^{\delta}-\phi_{\mc{S}})\cdot (u^{\delta}-P_{\mc{S}}^{\delta}u^{\delta}) + \int\limits_0^T\int\limits_{\Omega}\rho^{\delta}\chi_{\mc{S}}^{\delta}(\partial_t+u^{\delta}_{\mc{S}}\cdot\nabla)(\phi^{\delta}-\phi_{\mc{S}})\cdot P_{\mc{S}}^{\delta}u^{\delta}.
\end{equation*}
We estimate these terms in the following way:
\begin{multline*}
\left|\int\limits_0^T\int\limits_{\Omega}\rho^{\delta}\chi_{\mc{S}}^{\delta}(\partial_t+u^{\delta}_{\mc{S}}\cdot\nabla)(\phi^{\delta}-\phi_{\mc{S}})\cdot (u^{\delta}-P_{\mc{S}}^{\delta}u^{\delta})\right| \\ \leq \|\rho^{\delta}\chi_{\mc{S}}^{\delta}\|_{L^{\infty}((0,T)\times \Omega)}\|\chi^{\delta}_{\mc{S}}(\partial_t + P^{\delta}_{\mc{S}}u^{\delta}\cdot\nabla)(\phi^{\delta}-\phi_{\mc{S}})\|_{L^2(0,T;L^6(\Omega))}\frac{1}{\delta^{1/2}}\left\|\sqrt{ \chi^{\delta}_{\mc{S}}}\left(u^{\delta}-P^{\delta}_{\mc{S}}u^{\delta}\right)\right\|_{L^2((0,T)\times \Omega)}\delta^{1/2},
\end{multline*}
\begin{multline*}
\left|\int\limits_0^T\int\limits_{\Omega}\rho^{\delta}\chi_{\mc{S}}^{\delta}(\partial_t+u^{\delta}_{\mc{S}}\cdot\nabla)(\phi^{\delta}-\phi_{\mc{S}})\cdot P_{\mc{S}}^{\delta}u^{\delta}\right|\\ \leq \|\rho^{\delta}\chi_{\mc{S}}^{\delta}\|_{L^{\infty}((0,T)\times \Omega)}\|\chi^{\delta}_{\mc{S}}(\partial_t + P^{\delta}_{\mc{S}}u^{\delta}\cdot\nabla)(\phi^{\delta}-\phi_{\mc{S}})\|_{L^2(0,T;L^6(\Omega))}\|P^{\delta}_{\mc{S}}u^{\delta}\|_{L^2((0,T)\times \Omega)},
\end{multline*}
where we have used $\rho^{\delta}\chi_{\mc{S}}^{\delta} \in L^{\infty}((0,T)\times \Omega)$ as it is a solution to \eqref{approx5}.
Moreover, by \cref{approx:test} (with the case $p=6$), we know that for $\vartheta > 0$$$\|\chi^{\delta}_{\mc{S}}(\partial_t + P^{\delta}_{\mc{S}}u^{\delta}\cdot\nabla)(\phi^{\delta}-\phi_{\mc{S}})\|_{L^2(0,T;L^6(\Omega))}=\mc{O}(\delta^{\vartheta/6}).$$
Hence, 
\begin{equation}\label{A3}
T_1^{\delta} \rightarrow 0\mbox{ as } \delta\rightarrow 0.
\end{equation}
Observe that 
\begin{equation*}
T_2^{\delta}= \int\limits_0^T\int\limits_{\Omega}\rho^{\delta}\chi_{\mc{S}}^{\delta}\partial_t\phi_{\mc{S}}\cdot (u^{\delta}-P^{\delta}_{\mc{S}}u^{\delta}) + \int\limits_0^T\int\limits_{\Omega}\rho^{\delta}\chi_{\mc{S}}^{\delta}\partial_t\phi_{\mc{S}}\cdot P^{\delta}_{\mc{S}}u^{\delta}.
\end{equation*}
Now we use the following convergences:
\begin{itemize}
\item the strong convergence of $\sqrt{ \chi^{\delta}_{\mc{S}}}\left(u^{\delta}-P^{\delta}_{\mc{S}}u^{\delta}\right) \rightarrow 0$ in $ L^2((0,T)\times \Omega)$ (see the fifth term of inequality \eqref{again-10:45}), 
\item the strong convergence of $\chi^{\delta}_{\mc{S}}$ to $\chi_{\mc{S}}$ (see the convergence in \eqref{19:30}),
\item the weak convergence of  $\rho^{\delta}\chi_{\mc{S}}^{\delta}P^{\delta}_{\mc{S}}u^{\delta}\mbox{ to }\rho \chi_{\mc{S}}P_{\mc{S}}u$ (see the convergences in \eqref{19:31} and \eqref{prop1}), 
\end{itemize}
to deduce  
\begin{equation*}
T_2^{\delta} \rightarrow  \int\limits_0^T\int\limits_{\mc{S}(t)}\rho \chi_{\mc{S}}\partial_t \phi_{\mc{S}}\cdot P_{\mc{S}}u\quad \mbox{ as }\quad\delta\rightarrow 0.
\end{equation*}
Recall the definition  of $u_{\mc{S}}$ in  \eqref{uS} and the definition of $\rho_{\mc{S}}$ in  \eqref{rhoS}
to conclude
\begin{equation}\label{A4}
T_2^{\delta} \rightarrow  \int\limits_0^T\int\limits_{\mc{S}(t)}\rho_{\mc{S}}\partial_t \phi_{\mc{S}}\cdot u_{\mc{S}}\quad \mbox{ as }\quad\delta\rightarrow 0.
\end{equation}
Notice that 
\begin{equation}\label{A5}
T_3^{\delta}= \int\limits_0^T\int\limits_{\Omega}\rho^{\delta}\chi^{\delta}_{\mc{S}}(u^{\delta}_{\mc{S}}\cdot\nabla)\phi_{\mc{S}}\cdot u^{\delta} = \int\limits_0^T\int\limits_{\Omega}\rho^{\delta}\chi^{\delta}_{\mc{S}}(u^{\delta}_{\mc{S}}\otimes u^{\delta}_{\mc{S}}): \nabla\phi_{\mc{S}}= \int\limits_0^T\int\limits_{\Omega}\rho^{\delta}\chi^{\delta}_{\mc{S}}(u^{\delta}_{\mc{S}}\otimes u^{\delta}_{\mc{S}}): \mb{D}(\phi_{\mc{S}})=0.
\end{equation}
Eventually, combining the results \eqref{A1}--\eqref{A5}, we obtain 
\begin{equation*}
\int\limits_0^T\int\limits_{\Omega} \rho^{\delta} \left(u^{\delta}\cdot \frac{\partial}{\partial t}\phi +  u^{\delta} \otimes u^{\delta} : \nabla \phi\right) \rightarrow \int\limits_0^T\int\limits_{\mc{F}(t)} \rho_{\mc{F}}u_{\mc{F}}\cdot \frac{\partial}{\partial t}\phi_{\mc{F}} + \int\limits_0^T\int\limits_{\mc{S}(t)} \rho_{\mc{S}}u_{\mc{S}}\cdot \frac{\partial}{\partial t}\phi_{\mc{S}} + \int\limits_0^T\int\limits_{\mc{F}(t)} (\rho_{\mc{F}}u_{\mc{F}} \otimes u_{\mc{F}}) : \nabla \phi_{\mc{F}}.
\end{equation*}

\underline{Step 4.3: Pressure term of the momentum equation.}
We use the definition of $\phi^{\delta}$
\begin{equation*}
\phi^{\delta}=(1-\chi^{\delta}_{\mc{S}})\phi_{\mc{F}} + \chi^{\delta}_{\mc{S}}\phi^{\delta}_{\mc{S}},
\end{equation*}
to write
 \begin{equation*}
\int\limits_0^T\int\limits_{\Omega} \left(a^{\delta}(\rho^{\delta})^{\gamma} + {\delta} (\rho^{\delta})^{\beta} \right)\mathbb{I} : \mathbb{D}(\phi^{\delta}) = \int\limits_0^T\int\limits_{\Omega} \left[a_{\mc{F}} (1-\chi^{\delta}_{\mc{S}})(\rho^{\delta}_{\mc{F}})^{\gamma}+ {\delta}(1-\chi^{\delta}_{\mc{S}}) (\rho_{\mc{F}}^{\delta})^{\beta}\right]\mathbb{I} : \mathbb{D}(\phi_{\mc{F}}), 
 \end{equation*}
 where we have used the fact that $\operatorname{div}\phi_{\mc{S}}^{\delta}=0$.
 Due to the strong convergence of $\chi^{\delta}_{\mc{S}}$ to $\chi_{\mc{S}}$ and the weak convergence of $(\rho_{\mc{F}}^{\delta})^{\gamma}$, $(\rho_{\mc{F}}^{\delta})^{\beta}$ in \eqref{ag:conv3}, \eqref{ag:conv4} we obtain \begin{equation*}
 \int\limits_0^T\int\limits_{\Omega} a_{\mc{F}} (1-\chi^{\delta}_{\mc{S}})(\rho^{\delta}_{\mc{F}})^{\gamma}\mathbb{I} : \mathbb{D}(\phi_{\mc{F}}) \rightarrow  \int\limits_0^T\int\limits_{\Omega} a_{\mc{F}} (1-\chi_{\mc{S}})\overline{(\rho_{\mc{F}})^{\gamma}}\mathbb{I} : \mathbb{D}(\phi_{\mc{F}}) \mbox{  as  }\delta \to 0,
 \end{equation*}
 and
  \begin{equation*}
  \int\limits_0^T\int\limits_{\Omega}  {\delta}(1-\chi^{\delta}_{\mc{S}}) (\rho_{\mc{F}}^{\delta})^{\beta}\mathbb{I} : \mathbb{D}(\phi_{\mc{F}}) \rightarrow 0 \mbox{  as  }\delta \to 0.
   \end{equation*}
 
 In order to establish \eqref{weak-momentum}, it only remains to show that $\overline{\rho^{\gamma}_{\mc{F}}}=\rho^{\gamma}_{\mc{F}}$. This is equivalent to establishing some strong convergence result of the sequence $\rho^{\delta}_{\mc{F}}$. We follow the idea of \cite[Lemma 7.60, page 391]{MR2084891} to prove:  Let $\{\rho^{\delta}_{\mc{F}}\}$ be the sequence and $\rho_{\mc{F}}$ be its weak limit from \eqref{ag:conv2}. Then, at least for a chosen subsequence, 
 \begin{equation*}
 \rho^{\delta}_{\mc{F}}\rightarrow \rho_{\mc{F}} \mbox{  in  }L^p((0,T)\times\Omega),\quad 1\leq p < \gamma+\theta.
 \end{equation*}
  This immediately yields $\overline{\rho^{\gamma}_{\mc{F}}}=\rho^{\gamma}_{\mc{F}}$. Thus, we have recovered the weak form of the momentum equation.  


\underline{Step 5: Recovery of the energy estimate.} We derive from \eqref{10:45} that
\begin{multline*}
\int\limits_{\Omega}\Big( \rho^{\delta} |u^{\delta}|^2 + \frac{a_{\mc{F}}}{\gamma-1}(1-\chi^{\delta}_{\mc{S}})(\rho^{\delta}_{\mc{F}})^{\gamma} +  \frac{\delta}{\beta-1}(\rho^{\delta}_{\mc{F}})^{\beta}\Big) + \int\limits_0^T\int\limits_{\Omega} \Big(2\mu_{\mc{F}}(1-\chi^{\delta}_{\mc{S}})|\mb{D}(u_{\mc{F}}^{\delta})|^2 + \lambda_{\mc{F}}(1-\chi^{\delta}_{\mc{S}})|\operatorname{div}u^{\delta}_{\mc{F}}|^2\Big)  
 + \alpha \int\limits_0^T\int\limits_{\partial \Omega} |u^{\delta} \times \nu|^2 \\
 + \alpha \int\limits_0^T\int\limits_{\partial \mc{S}^{\delta}(t)} |(u^{\delta}-P^{\delta}_{\mc{S}}u^{\delta})\times \nu|^2 
   \leq \int\limits_0^T\int\limits_{\Omega}\rho^{\delta} g^{\delta} \cdot u^{\delta}
 + \int\limits_{\Omega}\Bigg( \frac{|q_0^{\delta}|^2}{\rho_0^{\delta}}\mathds{1}_{\{\rho_0^{\delta}>0\}} + \frac{a}{\gamma-1}(\rho_0^{\delta})^{\gamma} + \frac{\delta}{\beta-1}(\rho_0^{\delta})^{\beta} \Bigg).
\end{multline*}
To see the limiting behaviour of the above inequality as $\delta$ tends to zero, we observe that the limit as $\delta\to 0$ is similar to the limits $\varepsilon\rightarrow 0$ or $N\rightarrow \infty$ limit. Hence we obtain the energy inequality \eqref{energy}.

\underline{Step 6: Rigid body is away from boundary.} It remains to check that there exists $T$ small enough such that if $\operatorname{dist}(\mc{S}_0,\partial \Omega) > 2\sigma$, then
\begin{equation}\label{final-collision}
 \operatorname{dist}(\mc{S}(t),\partial \Omega) \geq \frac{3\sigma}{2}> 0 \quad \forall \ t\in [0,T].
\end{equation}
Let us introduce the following notation:
\begin{equation*}
(\mc{U})_{\sigma} = \left\{x\in \mathbb{R}^3\mid \operatorname{dist}(x,\mc{U})<\sigma\right\},
\end{equation*}
for an open set $\mc{U}$ and $\sigma > 0$. We recall the following result \cite[Lemma 5.4]{MR3272367}:
Let $\sigma > 0$. There exists $\delta_0 >0$ such that for all $0< \delta \leq \delta_0$,
\begin{equation}\label{cond:col1}
\mc{S}^{\delta}(t) \subset (\mc{S}(t))_{\sigma/4} \subset (\mc{S}^{\delta}(t))_{\sigma/2},\quad \forall\ t\in [0,T].
\end{equation}
Note that condition \eqref{cond:col1} and the relation \eqref{delta-collision}, i.e., $\operatorname{dist}(\mc{S}^{\delta}(t),\partial \Omega) \geq 2\sigma> 0$ for all $t\in [0,T]$ 
imply our required estimate \eqref{final-collision}. Thus, we conclude the proof of \cref{exist:upto collision}.
\end{proof}

It remains to prove \cref{approx:test}.
The main difference between \cref{approx:test} and  \cite[Proposition 5.1]{MR3272367} is the time regularity of the approximate test functions.
 Since here we only have weak convergence of $u^{\delta}$ in $L^2(0,T;L^2(\Omega))$, according to \cref{sequential2} we have convergence of $\eta^{\delta}_{t,s}$ in $H^{1}((0,T)^2; C^{\infty}_{loc}(\mb{R}^3))$. In \cite[Proposition 5.1]{MR3272367}, they have weak convergence of $u^{\delta}$ in $L^{\infty}(0,T;L^2(\Omega))$, which yields higher time regularity of the propagator $\eta^{\delta}_{t,s}$ in $W^{1,\infty}((0,T)^2; C^{\infty}_{loc}(\mb{R}^3))$. Still, the construction of the approximate function is essentially similar, which is why we skip the details and only present the main ideas of the proof here.

\begin{proof}[Proof of \cref{approx:test}]
The proof relies on the construction of the approximation $\phi^{\delta}_{\mc{S}}$ of $\phi_{\mc{S}}$ so that we can avoid the jumps at the interface for the test functions such that \eqref{cond1-good}--\eqref{cond2-good} holds. 

 The idea is to write the test functions in Lagrangian coordinates through the isometric propagator $\eta^{\delta}_{t,s}$ so that we can work on the fixed domain. Let $\Phi_{\mc{F}}$, $\Phi_{\mc{S}}$ and $\Phi^{\delta}_{\mc{S}}$ be the transformed quantities in the fixed domain related to $\phi_{\mc{F}}$, $\phi_{\mc{S}}$ and $\phi^{\delta}_{\mc{S}}$ respectively:
\begin{equation}\label{chofv}
\phi_{\mc{S}}(t,\eta^{\delta}_{t,0}(y)) = J_{\eta_{t,0}^{\delta}}\Big|_{y} (\Phi_{\mc{S}}(t,y)),\quad \phi_{\mc{F}}(t,\eta^{\delta}_{t,0}(y)) = J_{\eta_{t,0}^{\delta}}\Big|_{y} \Phi_{\mc{F}}(t,y)\quad\mbox{  and  }\quad\phi^{\delta}_{\mc{S}}(t,\eta^{\delta}_{t,0}(y)) =  J_{\eta_{t,0}^{\delta}}\Big|_{y} \Phi^{\delta}_{\mc{S}}(t,y),
\end{equation}
where $J_{\eta_{t,0}^{\delta}}$ is the Jacobian matrix of $\eta_{t,0}^{\delta}$. Note that if we define
\begin{equation*}
\Phi^{\delta}(t,y)=(1-\chi^{\delta}_{\mc{S}})\Phi_{\mc{F}} + \chi^{\delta}_{\mc{S}}\Phi^{\delta}_{\mc{S}},
\end{equation*}
then the definition of $\phi^{\delta}$ in \eqref{form:phi} gives
\begin{equation}\label{chofv1}
\phi^{\delta}(t,\eta^{\delta}_{t,0}(y)) = J_{\eta_{t,0}^{\delta}}\Big|_{y} (\Phi^{\delta}(t,y)).
\end{equation}
 Thus, the construction of the approximation $\phi^{\delta}_{\mc{S}}$ satisfying \eqref{cond1-good}--\eqref{cond2-good} is equivalent to building the approximation $\Phi^{\delta}_{\mc{S}}$ so that there is no jump for the function $\Phi^{\delta}$ at the interface and the following holds:
\begin{equation*}
\Phi^{\delta}_{\mc{S}}(t,x)=\Phi_{\mc{F}}(t,x) \quad \forall\ t\in (0,T),\ x\in \partial\mc{S}_0, 
\end{equation*}
and 
\begin{equation*}
\Phi^{\delta}_{\mc{S}}(t,\cdot)\approx \Phi_{\mc{S}}(t,\cdot)\mbox{ in }\mc{S}_0\mbox{ away from a }\delta^{\vartheta}\mbox{ neighborhood of }\partial\mc{S}_0\mbox{ with }\vartheta >0.
\end{equation*}
Explicitly, we set (for details, see \cite[Pages 2055-2058]{MR3272367}):
\begin{equation}\label{decomposePhi}
\Phi^{\delta}_{\mc{S}} = \Phi^{\delta}_{\mc{S},1} + \Phi^{\delta}_{\mc{S},2}, 
\end{equation}
with
\begin{equation}\label{phis1}
\Phi^{\delta}_{\mc{S},1}= \Phi_{\mc{S}} + \chi (\delta^{-\vartheta}z) \left[(\Phi_{\mc{F}}-\Phi_{\mc{S}}) - ((\Phi_{\mc{F}}-\Phi_{\mc{S}})\cdot e_z) e_z\right],
\end{equation}
 where $\chi : \mathbb{R} \rightarrow [0,1]$ is a smooth truncation function which is equal to 1 in a neighborhood of 0 and $z$ is a coordinate transverse to the boundary $\partial \mc{S}_0 = \{z=0\}$. Moreover, to make $\Phi^{\delta}_{\mc{S}}$ divergence-free in $\mc{S}_0$, we need to take $\Phi^{\delta}_{\mc{S},2}$ such that 
 \begin{equation*}
 \operatorname{div}\Phi^{\delta}_{\mc{S},2}=-\operatorname{div}\Phi^{\delta}_{\mc{S},1}\quad\mbox{in }\mc{S}_0,\quad \Phi^{\delta}_{\mc{S},2}=0\quad\mbox{on }\partial\mc{S}_0.
 \end{equation*}
Observe that, the explicit form \eqref{phis1} of $\Phi^{\delta}_{\mc{S},1}$ yields
\begin{equation}\label{divphis1}
\operatorname{div}\Phi^{\delta}_{\mc{S},2}=-\operatorname{div}\Phi^{\delta}_{\mc{S},1} = -\chi(\delta^{-\vartheta}z)\operatorname{div}\left[(\Phi_{\mc{F}}-\Phi_{\mc{S}}) - ((\Phi_{\mc{F}}-\Phi_{\mc{S}})\cdot e_z) e_z\right].
\end{equation}
Thus, the expressions \eqref{phis1}--\eqref{divphis1} give us: for all $p<\infty$,
\begin{align}\label{Phi11}
\|\Phi^{\delta}_{\mc{S},1}-\Phi_{\mc{S}}\|_{H^1(0,T; L^p(\mc{S}_0))} &\leq C\delta^{\vartheta/p},
\\
\label{Phi12}
\|\Phi^{\delta}_{\mc{S},1}-\Phi_{\mc{S}}\|_{H^1(0,T; W^{1,p}(\mc{S}_0))} &\leq C\delta^{-\vartheta(1-1/p)},
\end{align}
and
\begin{equation}\label{Phi2}
\|\Phi^{\delta}_{\mc{S},2}\|_{H^1(0,T; W^{1,p}(\mc{S}_0))}\leq C\|\chi({\delta}^{-\vartheta}z)\operatorname{div}\left[(\Phi_{\mc{F}}-\Phi_{\mc{S}}) - ((\Phi_{\mc{F}}-\Phi_{\mc{S}})\cdot e_z) e_z\right]\|_{H^1(0,T; L^{p}(\mc{S}_0))} \leq C\delta^{\vartheta/p}.
\end{equation}
Using the decomposition \eqref{decomposePhi} of $\Phi_{\mc{S}}^{\delta}$ and the estimates \eqref{Phi11}--\eqref{Phi12}, \eqref{Phi2}, we obtain
\begin{align*}
\|\Phi^{\delta}_{\mc{S}}-\Phi_{\mc{S}}\|_{H^1(0,T; L^p(\mc{S}_0))} &\leq C\delta^{\vartheta/p},\\
\|\Phi^{\delta}_{\mc{S}}-\Phi_{\mc{S}}\|_{H^1(0,T; W^{1,p}(\mc{S}_0))} &\leq C\delta^{-\vartheta(1-1/p)}.
\end{align*}
Furthermore, we combine the above estimates with the uniform bound of the propagator $\eta^{\delta}_{t,0}$ in $H^1(0,T; C^{\infty}(\Omega))$ to obtain
\begin{align}\label{est:Phi}
\left\|J_{\eta_{t,0}^{\delta}}|_{y}(\Phi^{\delta}_{\mc{S}}-\Phi_{\mc{S}})\right\|_{H^1(0,T; L^p(\mc{S}_0))} &\leq C\delta^{\vartheta/p},\\
\label{est:dPhi}
\left\|J_{\eta_{t,0}^{\delta}}|_{y}(\Phi^{\delta}_{\mc{S}}-\Phi_{\mc{S}})\right\|_{H^1(0,T; W^{1,p}(\mc{S}_0))} &\leq C\delta^{-\vartheta(1-1/p)}.
\end{align}
Observe that due to the change of variables \eqref{chofv} and estimate \eqref{est:Phi}:
\begin{equation}\label{est1}
\|\chi^{\delta}_{\mc{S}}(\phi^{\delta}_{\mc{S}}-\phi_{\mc{S}})\|_{L^p((0,T)\times \Omega))}\leq C\|J_{\eta_{t,0}^{\delta}}|_{y}(\Phi^{\delta}_{\mc{S}}-\Phi_{\mc{S}})\|_{L^p((0,T)\times \mc{S}_0)}\leq C
\delta^{\vartheta/p}.
\end{equation}
Since
\begin{equation*}
\|\phi^{\delta}-\phi\|_{L^p((0,T)\times \Omega))}\leq \|(\chi^{\delta}_{\mc{S}}-\chi_{\mc{S}})\phi_{\mc{F}}\|_{L^p((0,T)\times \Omega))} + \|\chi^{\delta}_{\mc{S}}(\phi_{\mc{S}}^{\delta}-\phi_{\mc{S}})\|_{L^p((0,T)\times \Omega))} + \|(\chi^{\delta}_{\mc{S}}-\chi_{\mc{S}})\phi_{\mc{S}}\|_{L^p((0,T)\times \Omega))},
\end{equation*}
using the strong convergence of $\chi^{\delta}_{\mc{S}}$ and the estimate \eqref{est1}, we conclude that
\begin{equation*}
\phi^{\delta} \rightarrow \phi\mbox{ strongly in   }L^p((0,T)\times \Omega).
\end{equation*}
We use estimate \eqref{Phi12} and the relation \eqref{chofv1} to obtain
\begin{equation*}
\|\phi^{\delta}\|_{L^p(0,T;W^{1,p}(\Omega))}\leq \delta^{-\vartheta(1-1/p)}.
\end{equation*}
Moreover, the change of variables \eqref{chofv} and estimate \eqref{est:Phi} give
\begin{align}\label{timechi}
\begin{split} \|\chi^{\delta}_{\mc{S}}(\partial_t + P^{\delta}_{\mc{S}}u^{\delta}\cdot\nabla)(\phi^{\delta}-\phi_{\mc{S}})\|_{L^2(0,T;L^p(\Omega))} &\leq C\left\|\frac{d}{dt}\left(J_{\eta_{t,0}^{\delta}}\Big|_{y}(\Phi^{\delta}_{\mc{S}}-\Phi_{\mc{S}})\right)\right\|_{L^2(0,T;L^p(\mc{S}_0))}\\ 
& \leq C\left\|J_{\eta_{t,0}^{\delta}}|_{y}(\Phi^{\delta}_{\mc{S}}-\Phi_{\mc{S}})\right\|_{H^1(0,T;L^p(\mc{S}_0))}\leq C\delta^{\vartheta/p}.
\end{split}
\end{align}
The above estimate \eqref{timechi}, strong convergence of $\chi_{\mc{S}}^{\delta}$ to $\chi_{\mc{S}}$ in $C(0,T;L^p(\Omega))$ and weak convergence of $P^{\delta}_{\mc{S}}u^{\delta}$ to $P_{\mc{S}} u \mbox{ weakly } \mbox{ in }L^{2}(0,T; C^{\infty}_{loc}(\mb{R}^3))$, give us
\begin{equation*}
(\partial_t + P^{\delta}_{\mc{S}}u^{\delta}\cdot\nabla)\phi^{\delta}\rightarrow (\partial_t + P_{\mc{S}}u\cdot \nabla)\phi \mbox{ weakly in } L^2(0,T;L^p(\Omega)), 
\end{equation*}
where 
\begin{equation*}
\phi^{\delta}=(1-\chi^{\delta}_{\mc{S}})\phi_{\mc{F}} + \chi^{\delta}_{\mc{S}}\phi^{\delta}_{\mc{S}}\quad\mbox{ and }\quad \phi=(1-\chi_{\mc{S}})\phi_{\mc{F}} + \chi_{\mc{S}}\phi_{\mc{S}}.
\end{equation*}
 \end{proof}

\section*{Acknowledgment}
{\it \v S. N. and A. R. have been supported by the Czech Science Foundation (GA\v CR) project GA19-04243S. The Institute of Mathematics, CAS is supported by RVO:67985840.}

\section*{Compliance with Ethical Standards}
\section*{Conflict of interest}
The authors declare that there are no conflicts of interest.
\bibliography{reference1}

\begin{thebibliography}{10}

\bibitem{BG}
{\sc M.~Boulakia and S.~Guerrero}, {\em A regularity result for a solid-fluid
  system associated to the compressible {N}avier-{S}tokes equations}, Ann.
  Inst. H. Poincar\'{e} Anal. Non Lin\'{e}aire, 26 (2009), pp.~777--813.

\bibitem{CN}
{\sc N.~Chemetov and {\v S}.~Ne{\v c}asov{\'a}}, {\em The motion of the rigid
  body in the viscous fluid including collisions. {G}lobal solvability result},
  Nonlinear Anal. Real World Appl., 34 (2017), pp.~416--445.

\bibitem{CST}
{\sc C.~Conca, J.~S. Martin, and M.~Tucsnak}, {\em Existence of solutions for
  the equations modelling the motion of a rigid body in a viscous fluid},
  Commun. Partial Differential Equation, 25 (2000), pp.~1019--1042.

\bibitem{DEES1}
{\sc B.~Desjardins and M.~J. Esteban}, {\em Existence of weak solutions for the
  motion of rigid bodies in a viscous fluid}, Arch. Rational Mech. Anal., 146
  (1999), pp.~59--71.

\bibitem{DEES2}
\leavevmode\vrule height 2pt depth -1.6pt width 23pt, {\em On weak solutions
  for fluid-rigid structure interaction: Compressible and incompressible
  model}, Commun. Partial Differential Equations, 25 (2000), pp.~1399--1413.

\bibitem{DiPerna1989}
{\sc R.~DiPerna and P.~Lions}, {\em Ordinary differential equations, transport
  theory and sobolev spaces.}, Inventiones mathematicae, 98 (1989),
  pp.~511--548.

\bibitem{F4}
{\sc E.~Feireisl}, {\em On the motion of rigid bodies in a viscous compressible
  fluid}, Arch. Rational Mech. Anal., 167 (2003), pp.~281--308.

\bibitem{F3}
\leavevmode\vrule height 2pt depth -1.6pt width 23pt, {\em On the motion of
  rigid bodies in a viscous incompressible fluid}, Journal of Evolution
  Equations, 3 (2003), pp.~419--441.

\bibitem{EF70}
\leavevmode\vrule height 2pt depth -1.6pt width 23pt, {\em Dynamics of viscous
  compressible fluids}, Oxford University Press, Oxford, 2004.

\bibitem{FHN}
{\sc E.~Feireisl, M.~Hillairet, and {\v S}.~Ne{\v c}asov{\' a}}, {\em On the
  motion of several rigid bodies in an incompressible non-newtonian fluid},
  Nonlinearity, 21 (2008), pp.~1349--1366.

\bibitem{MR3729430}
{\sc E.~Feireisl and A.~Novotn\'{y}}, {\em Singular limits in thermodynamics of
  viscous fluids}, Advances in Mathematical Fluid Mechanics,
  Birkh\"{a}user/Springer, Cham, 2017.

\bibitem{MR1867887}
{\sc E.~Feireisl, A.~Novotn\'{y}, and H.~Petzeltov\'{a}}, {\em On the existence
  of globally defined weak solutions to the {N}avier-{S}tokes equations}, J.
  Math. Fluid Mech., 3 (2001), pp.~358--392.

\bibitem{G2}
{\sc G.~P. Galdi}, {\em On the motion of a rigid body in a viscous liquid: a
  mathematical analysis with applications}, in Handbook of mathematical fluid
  dynamics, {V}ol. {I}, North-Holland, Amsterdam, 2002, pp.~653--791.

\bibitem{GGH13}
{\sc M.~Geissert, K.~G{\"o}tze, and M.~Hieber}, {\em {$L^p$}-theory for strong
  solutions to fluid-rigid body interaction in {N}ewtonian and generalized
  {N}ewtonian fluids}, Trans. Amer. Math. Soc., 365 (2013), pp.~1393--1439.

\bibitem{GH}
{\sc D.~G{\' e}rard-Varet and M.~Hillairet}, {\em Regularity issues in the
  problem of fluid structure interaction}, Archive for Rational Mechanical
  Analysis, 2010 (195), pp.~375--407.

\bibitem{MR3272367}
{\sc D.~G{\'e}rard-Varet and M.~Hillairet}, {\em Existence of weak solutions up
  to collision for viscous fluid-solid systems with slip}, Comm. Pure Appl.
  Math., 67 (2014), pp.~2022--2075.

\bibitem{GHC}
{\sc D.~G{\' e}rard-Varet, M.~Hillairet, and C.~Wang}, {\em The influence of
  boundary conditions on the contact problem in a 3d navier-stokes flow}, J.
  Math. Pures Appl., 103 (2015), pp.~1--38.

\bibitem{GLSE}
{\sc M.~D. Gunzburger, H.~C. Lee, and A.~Seregin}, {\em Global existence of
  weak solutions for viscous incompressible flow around a moving rigid body in
  three dimensions}, J. Math. Fluid Mech., 2 (2000), pp.~219--266.

\bibitem{HMTT}
{\sc B.~H. Haak, D.~Maity, T.~Takahashi, and M.~Tucsnak}, {\em Mathematical
  analysis of the motion of a rigid body in a compressible
  {N}avier-{S}tokes-{F}ourier fluid}, Math. Nachr., 292 (2019), pp.~1972--2017.

\bibitem{HES}
{\sc T.~Hesla}, {\em Collision of smooth bodies in a viscous fluid: A
  mathematical investigation}, 2005.

\bibitem{HiMu}
{\sc M.~Hieber and M.~Murata}, {\em The {$L^p$}-approach to the fluid-rigid
  body interaction problem for compressible fluids}, Evol. Equ. Control Theory,
  4 (2015), pp.~69--87.

\bibitem{HIL}
{\sc M.~Hillairet}, {\em Lack of collision between solid bodies in a 2d
  incompressible viscous flow}, Comm. Partial Differential Equations, 32
  (2007), pp.~1345--1371.

\bibitem{HT}
{\sc M.~Hillairet and T.~Takahashi}, {\em Collisions in three-dimensional fluid
  structure interaction problems}, SIAM J. Math. Anal., 40 (2009),
  pp.~2451--2477.

\bibitem{HOST}
{\sc K.-H. Hoffmann and V.~N. Starovoitov}, {\em On a motion of a solid body in
  a viscous fluid. {T}wo-dimensional case}, Adv. Math. Sci. Appl., 9 (1999),
  pp.~633--648.

\bibitem{KrNePi2}
{\sc O.~Kreml, {\v S}.~Ne{\v c}asov{\'a}, and T.~Piasecki}, {\em Weak-strong
  uniqueness for the compressible fluid-rigid body interaction}, J.
  Differential Equations, 268 (2020), pp.~4756--4785.

\bibitem{kuku}
{\sc P.~Kuku\v{c}ka}, {\em On the existence of finite energy weak solutions to
  the {N}avier-{S}tokes equations in irregular domains}, Math. Methods Appl.
  Sci., 32 (2009), pp.~1428--1451.

\bibitem{LI4}
{\sc P.-L. Lions}, {\em Mathematical topics in fluid dynamics, Vol.2,
  Compressible models}, Oxford Science Publication, Oxford, 1998.

\bibitem{MOF}
{\sc H.~K. Moffat}, {\em Viscous and resistive eddies near a sharp corner}, J.
  Fluid Mech., 18 (1964), pp.~1--18.

\bibitem{NP}
{\sc J.~Neustupa and P.~Penel}, {\em A weak solvability of the
  {N}avier-{S}tokes equation with {N}avier's boundary condition around a ball
  striking the wall}, in Advances in mathematical fluid mechanics, Springer,
  Berlin, 2010, pp.~385--407.

\bibitem{MR2084891}
{\sc A.~Novotn\'{y} and I.~Stra\v{s}kraba}, {\em Introduction to the
  mathematical theory of compressible flow}, vol.~27 of Oxford Lecture Series
  in Mathematics and its Applications, Oxford University Press, Oxford, 2004.

\bibitem{roy2019stabilization}
{\sc A.~Roy and T.~Takahashi}, {\em Stabilization of a rigid body moving in a
  compressible viscous fluid}, Journal of Evolution Equations,  (2020).

\bibitem{SST}
{\sc J.~A. San~Mart{\'\i}n, V.~Starovoitov, and M.~Tucsnak}, {\em Global weak
  solutions for the two-dimensional motion of several rigid bodies in an
  incompressible viscous fluid}, Arch. Ration. Mech. Anal., 161 (2002),
  pp.~113--147.

\bibitem{SER3}
{\sc D.~Serre}, {\em Chute libre d'un solide dans un fluide visqueux
  incompressible. {E}xistence}, Jap. J. Appl. Math., 4 (1987), pp.~99--110.

\bibitem{T}
{\sc T.~Takahashi}, {\em Analysis of strong solutions for the equations
  modeling the motion of a rigid-fluid system in a bounded domain}, Adv.
  Differential Equations, 8 (2003), pp.~1499--1532.

\bibitem{Wa}
{\sc C.~Wang}, {\em Strong solutions for the fluid-solid systems in a 2-d
  domain}, Asymptot. Anal., 89 (2014), pp.~263--306.

\end{thebibliography}
\bibliographystyle{siam}

\end{document}